\documentclass[12pt,reqno]{amsart}

\numberwithin{equation}{section} %%% Equations numbered by section. If you don't want it, please delete it.

\usepackage{fullpage}
\usepackage{tikz}
\usetikzlibrary{positioning}
\usepackage{hhline}
\usepackage{tabularx}
\usepackage{booktabs}
\usepackage{nomencl}
\usepackage{etoolbox}
\begin{document}

%%%%%%%%%%%%%%%%
\def\fontindex{\arabic}
\def\fontindexeq{\arabic}
\def\fonttitre{\textsf}
\newcounter{thh}
\newtheorem{thm}[thh]{\fonttitre{Theorem}}
\newtheorem{question}{\fonttitre{Question}}
\newtheorem{proposition}[thh]{\fonttitre{Proposition}}
\newtheorem*{pro*}{\fonttitre{Proposition}}
\newtheorem{corollary}[thh]{\fonttitre{Corollary}}
\newtheorem*{coro*}{\fonttitre{Corollary}}
\newtheorem{lemma}[thh]{\fonttitre{Lemma}}
\theoremstyle{definition} \newtheorem{remark}[thh]{\fonttitre{Remark}}
\newtheorem*{rems}{\fonttitre{Remarks}}
\newenvironment{com}{\par}{}
\newtheorem{defi}[thh]{\fonttitre{Definition}}
\newtheorem*{defi*}{\fonttitre{Définition}}
\newtheorem{conj}{\fonttitre{Conjecture}}
\newtheorem*{exes}{\fonttitre{Examples}}
\newtheorem*{exe}{\fonttitre{Example}}
\newtheorem{nota}[thh]{\fonttitre{Notation}}
\newtheorem*{nota*}{\fonttitre{Notation}}
\newenvironment{prf}{\begin{proof}}{\end{proof}}
%%
%%%%%%%%%%%%%%%%
\renewcommand{\thethh}{\arabic{section}.\fontindex{thh}}
\renewcommand{\theequation}{\arabic{section}.\fontindex{equation}}
%%%%%%%%%%%%%%%%
\def\vsp{\vspace{1mm}}
\def\th#1{\vspace{1mm}\noindent{\bf #1}\quad}
\def\proof{\vspace{1mm}\noindent{\it Proof}\quad}
\def\no{\nonumber}
\newenvironment{prof}[1][Proof]{\noindent\textit{#1}\quad }
{\hfill $\Box$\vspace{0.7mm}}
\def\q{\quad} \def\qq{\qquad}
\allowdisplaybreaks[4]
%%%%%%%%%%%%%%%%%%%%%%%%%%%%%%%%%%%%%%%%%%%%%%%%%%%%%%%%%%%%%%%%%%%%%%%%%%%%%%%%%%%%%%%%%%%%%%%
%%-------------------       Beginning of  Author's Definitions       -------------------%%
%%                     Note: You may add your own definitions here.

%%%%%%%% packages for the tabl of notation
%%\usepackage{tabularx}
%\usepackage{booktabs}
%\usepackage[labelfont=bf,format=plain,justification=raggedright,singlelinecheck=false]{caption}

%%%%%%%% nomenclature
%\usepackage{nomencl}
% HOW TO RUN IT
% pdflatex WF-with-nomen-final-arxiv.tex
% makeindex WF-with-nomen-final-arxiv.nlo -s nomencl.ist -o WF-with-nomen-final-arxiv.nls
% pdflatex WF-amse-with-nomen-final-arxiv.tex

\setlength{\nomitemsep}{-\parsep}
\renewcommand{\nomname}{List of symbols}
\makenomenclature

\renewcommand\nomgroup[1]{%
  \item[\bfseries
  \ifstrequal{#1}{A}{Function spaces}{%
  \ifstrequal{#1}{B}{Dual pairs}{%
  \ifstrequal{#1}{C}{Orbital integrals and their limits}{%
  \ifstrequal{#1}{D}{An integral over the slice through a nilpotent element}{%
  \ifstrequal{#1}{E}{Other symbols}{%
  }}}}}%
]}
%%%%%%%%%%%%%%%%%%%%
%% Some Greek letters for the Lie algebras, ....
\def\Chc{Chc}
\def\chc{\mathop{\hbox{\rm chc}}\nolimits}
\def\ch{\mathop{\hbox{\rm ch}}\nolimits}
\def\n{\mathfrak n}
\def\g{\mathfrak g}
\def\z{\mathfrak z}
\def\y{\mathrm y}
\def\h{\mathfrak h}
\def\sp{\mathfrak {sp}}
\def\sl{\mathfrak {sl}}
\def\o{\mathfrak o}

\def\gl{\mathfrak gl}
\def\u{\mathfrak u}
\def\k{\mathfrak k}
\def\p{\mathfrak p}
\def\q{\mathfrak q}

\def\R{\mathbb{R}}
\def\C{\mathbb{C}}
\def\Ze{\mathbb{Z}}
\def\Ha{\mathbb{H}}
\def\Ca{\mathcal{C}}
\def\m{\mathfrak m}
\def\b{\mathfrak b}
\def\l{\mathfrak l}
\def\z{\mathfrak z}
\def\c{\mathfrak c}
\def\ta{\mathfrak t}
\def\so{\mathfrak s_{\overline 0}}
\def\ss1{\mathfrak s_{\overline 1}}
\def\hso{\mathfrak h_{\overline 0}}
\def\hs1{\mathfrak h_{\overline 1}}

\def\Br{\mathrm{B}}
\def\Ar{\mathrm{A}}
\def\Pg{\mathrm{P}}
\def\supp{\mathrm{supp}}
\def\Op{\mathrm{Op}}
\def\Ker{\mathrm{Ker}}

\def\deta{\mathfrak{d}}
\def\fbot{\underline{\bot}}

\def\Io{\mathcal{I}}

\def\npb{\text{\rm npb}}
%% Groups
\def\G{\mathrm{G}}
\def\N{\mathrm{N}}
\def\Zg{\mathrm{Z}}
\def\E{\mathrm{E}}
\def\Qg{\mathrm{Q}}
\def\SOg{\mathrm{S}\mathrm{O}}
\def\K{\mathrm{K}}
\def\H{\mathrm{H}}
\def\M{\mathrm{M}}
\def\Zg{\mathrm{Z}}
\def\Sg{\mathrm{S}}
\def\Spin{\mathrm{Spin}}

\def\L{\mathrm{L}}
\def\Bbb{\mathbb}
\def\Ni{\mathrm{N}}
\def\N{\mathrm{N}}
\def\A{\mathrm{A}}
\def\H{\mathrm{H}}
\def\T{\mathrm{T}}
\def\GL{\mathrm{GL}}
\def\SL{\mathrm{SL}}
\def\SO{\mathrm{SO}}
\def\Sp{\mathrm{Sp}}
\def\Mp{\mathrm{Mp}}
\def\Og{\mathrm{O}}
\def\Ug{\mathrm{U}}
\def\Se{\mathcal{S}}
\def\Eg{\mathrm{E}}
\def\Fg{\mathrm{F}}
\def\id{\mathrm{id}}
\def\B{\mathrm{B}}
\def\Mg{\mathrm{M}}

%% anticommutants
\newcommand{\anticomm}[2]{\null^{#1}#2}
\newcommand{\danticomm}[2]{\null^{\anticomm{#1}{#2}}#2}

\def \t{\tilde}
\def \wt{\widetilde}
\newcommand{\reg}[1]{ {#1}^{reg}}
\newcommand{\preg}[1]{ {#1}^{\prime reg}}
\newcommand{\pt}[1]{ \t{#1}^{\prime}}
\newcommand{\pregt}[1]{ \t{#1}^{\prime reg}}
\newcommand{\regt}[1]{ \t{#1}^{reg}}
\newcommand{\saut}[1]{ \langle #1 \rangle }

\newcommand{\un}[1]{\underline{#1}}
\def\ii{\mathsf{i}}
\def\CSC{\operatorname{CSC}}
\newcommand{\OP}{\mathop{\rm{OP}}}

%% Vector spaces
\def\W{\mathsf{W}}
\def\Wv{\mathrm{W}}
\def\W+{\mathrm{W}_{\BB C}}
\def\Vv{\mathrm{V}}
\def\Uv{\mathrm{U}}
\def\V{\mathsf{V}}
\def\U{\mathsf{U}}
\def\X{\mathsf{X}}
\def\Y{\mathsf{Y}}
\def\Xv{\mathrm{X}}
\def\Yv{\mathrm{Y}}
\def\Sy{\mathsf{S}}
\def\symm{\mathsf{s}}

%% The general letter for a field
\def\Dc{\mathbb {D}}
\def\Zb{\mathbb {Z}}
%% Smooth functions with compact support and associated distributions
\def\D{\mathcal{D}}

%% A matrix which induces a complex structure
\def\J{\mathcal{J}}
\def\Id{\mathcal{I}}

%% Harish Chandra Integral
\def\Hc{\mathcal {H}}

\def\We{\mathcal{W}}
%% Some useful symbols
\def\End{\mathop{\hbox{\rm End}}\nolimits}
\def\diag{\mathop{\hbox{\rm diag}}\nolimits}
\def\det{\mathop{\hbox{\rm det}}\nolimits}
\def\ad{\mathop{\hbox{\rm ad}}\nolimits}
\def\Ad{\mathop{\hbox{\rm Ad}}\nolimits}
\def\Hom{\mathop{\hbox{\rm Hom}}\nolimits}
\def\Re{\mathop{\hbox{\rm Re}}\nolimits}
\def\Im{\mathop{\hbox{\rm Im}}\nolimits}
\def\Lie{\mathop{\hbox{\rm Lie}}\nolimits}
\def\tr{\mathop{\hbox{\rm tr}}\nolimits}
\def\WF{\mathop{\hbox{\rm WF}}\nolimits}
\def\sign{\mathop{\hbox{\rm sign}}\nolimits}
\def\sgn{\mathop{\hbox{\rm sgn}}\nolimits}
\def\Diff{\mathop{\hbox{\rm Diff}}\nolimits}
\def\Dp{\mathop{\hbox{\rm D}}\nolimits}
\def\vol{\mathop{\hbox{\rm vol}}\nolimits}
\newcommand\diesis[1]{#1^\sharp}

\def\gr{\mathop{\hbox{\rm gr}}\nolimits}
\def\lim{\mathop{\hbox{\rm lim}}\nolimits}

\def\sym{\mathop{\hbox{\rm Sym}}\nolimits}
\newcommand\inner[2]{\langle #1,#2\rangle}

\def\Res{\mathop{\hbox{\rm Res}}\nolimits}

\def\Ind{\mathop{\hbox{\rm Ind}}\nolimits}
\def\supp{\mathop{\hbox{\rm supp}}\nolimits}
\def\Vect{\mathop{\hbox{\rm Vect}}\nolimits}

\def\ker{\mathop{\hbox{\rm ker}}\nolimits}
\def\im{\mathop{\hbox{\rm im}}\nolimits}
%% A slice named U
\def\U{\mathcal{U}}
\def\Go{\mathcal{G}}
\def\Zo{\mathcal{Z}}
\def\Vo{\mathcal{V}}
\def\Z{\mathcal{Z}}
\def\I{\mathcal{I}}
\def\P{\mathcal{P}}
\def\HP{\mathcal{H}\mathcal{P}}
\def\Oo{\mathcal{O}}
\def\Symm{\mathcal{S}}

%% The Schwartz space
\def\Ss{\mathcal{S}}
\def\Ms{\mathcal{M}}
\def\Hs{\mathcal{H}}
\def\SHs{\mathcal{SH}}
\def\Ps{\mathcal{P}}

%%%%%%%%%%%%%%%%%
\def\nn{\nonumber}

%%-------------------         the end of  Author's Definitions           -------------------%%

%\author{McKee, M., Pasquale, A. and Przebinda, T.}                             
%%%  appear on the head of even pages  %%%

%%% Running Title, appear on the head of odd pages  %%%

\title[Dual pairs with one member compact]{The wave front set correspondence\\ for dual pairs with one member compact}

\author{M. McKee}
\address{Department of Mathematics, University of Iowa, 
Iowa City, IA 52242, USA}
\email{mark-mckee@uiowa.edu}
\author{A. Pasquale}
\address{Université de Lorraine, CNRS, IECL, F-57000 Metz, France}
\email{angela.pasquale@univ-lorraine.fr}

\author{T. Przebinda}
\address{Department of Mathematics, University of Oklahoma, Norman, OK 73019, USA}
\email{przebinda@gmail.com}

\thanks{
The second author is grateful to the University of Oklahoma for hospitality and financial support. The third author gratefully acknowledges hospitality and financial support from the Universit\'e de Lorraine and partial support from NSA (Grant No.  H98230-13-1-0205)  
and NSF (Grant No. DMS-2225892).}                 
           
\begin{abstract}
Let $\Wv$ be a real symplectic space and $(\G,\G')$ an irreducible dual pair in $\Sp(\Wv)$, in the sense of Howe, with $\G$ compact. 
Let $\wt\G$ be the preimage of $\G$ in the metaplectic group $\wt{\Sp}(\Wv)$. Given an irreducible unitary representation $\Pi$ of $\wt\G$ that occurs in the restriction of the Weil representation to $\wt\G$, let $\Theta_\Pi$ denote its character. We prove that, for a suitable embedding $T$ of $\wt{\Sp}(\Wv)$ in the space of tempered distributions on $\Wv$, the distribution $T(\check\Theta_\Pi)$ admits an asymptotic limit, and the limit is a nilpotent orbital integral. 
As an application, we compute the wave front set of $\Pi'$, the representation of $\wt{\G'}$
dual to $\Pi$,  by elementary means.
\end{abstract}

        % the keywords

% MR(2000) Subject Classification
\maketitle%

\tableofcontents
%%
%\medskip

\section{\textbf{Introduction}}
\label{Introduction}

Let $(\G, \G')$ be an irreducible reductive dual pair with $\G$ compact. Thus there is a division algebra $\Bbb D=\R, \C$ or $\Ha$ with an involution $\Bbb D\ni a\to\overline a\in\Bbb D$  over $\R$,  a finite dimensional 
right $\Bbb D$-vector space $\Vv$, with a positive definite hermitian 
form $(\cdot,\cdot)$, a finite dimensional 
right $\Bbb D$-vector space $\Vv'$ with a skew-hermitian 
form $(\cdot,\cdot)'$ so that $\G$ is the isometry group of $(\cdot,\cdot)$ and $\G'$ is the isometry group of $(\cdot,\cdot)'$. 
\footnote{
We use the notation $\G'$ for the second member of a dual pair because it is the centralizer of $\G$ in $\Sp(\Wv)$. We also use the notation $\cdot '$ for all the objects associated with $\G'$, such as $\g'$, $\Pi'$, ... . Unfortunately, this collides with the usual notation for the dual of a linear topological space in functional analysis, also used in this paper, such as $\mathcal{D}'(\R^n)$, $\mathcal{S}'(\R^n)$, ... . We hope the reader will guess from the context the correct meaning of the notation. 
}
Explicitly, $(\G, \G')$ is one of the pairs
\footnote{
The notation for Lie groups is as in Howe \cite{HoweTrans}. In particular, we denote the quaternion unitary group $\Ug_d(\Ha)$ by $\Sp_d$.
} 
\begin{equation} \label{classificationGG'}
(\Og_d, \Sp_{2n}(\R))\,,  \qquad  (\Ug_{d}, \Ug_{p,q})\,, \qquad        
(\Sp_d, \Og^*_{2n}). 
\end{equation}
These groups act on $\Wv=\Hom_{\mathbb{D}}(\Vv,\Vv')$ via post-multiplication and pre-multiplication by the inverse. 
We set $d=\dim_\Bbb D\Vv$ and $d'=\dim_\Bbb D\Vv'$.

There is a map
\[
\Hom_{\mathbb{D}}(\Vv,\Vv')\ni w\to w^*\in \Hom_{\mathbb{D}}(\Vv',\Vv)
\]
defined by
\[
(w v,v')'=(v,w^*v') \qquad (v\in \Vv\,,\  v'\in \Vv')\,,
\]
a non-degenerate symplectic form $\langle \cdot,\cdot\rangle$ on the real vector space $\Wv$
\[
\langle w',w\rangle=\tr_{\Bbb D/\R}(w^*w') \qquad (w,w'\in \Wv)\,,
\]
preserved by the actions of $\G$ and $\G'$. 
Here $\tr_{\Bbb D/\R}$ denotes the trace of an endomorphism considered over $\R$. 
Moreover, we have the unnormalized moment maps
\begin{equation}
\label{tautau'}
\tau:\Wv\ni w\to w^*w\in\g\,,\ \ \ \tau':\Wv\ni w\to ww^*\in\g'\,,
\end{equation}
where $\g$ and $\g'$ are the Lie algebras of
$\G$ and $\G'$,
respectively.
These maps are $\G\G'$-equivariant in the sense that
\[
\tau(gg'(w))=g\tau(w)g^{-1}\,,\ \ \ \tau'(gg'(w))=g'\tau'(w)g'{}^{-1}\qquad (g\in\G\,,\  g'\in \G'\,,\ w\in \Wv)\,.
\]
In particular the fiber $\tau^{-1}(0)\subseteq \Wv$ is a union of $\G\G'$-orbits, which are well known and easy to describe. We collect the relevant facts in the two lemmas below. Since we could not find a reference, their proofs are provided in Appendices \ref{proof of lemma 1} and \ref{proof of lemma 2}.
\begin{lemma}
\label{structure of t'-1tau(0)}
Let $m$ be the minimum of $d$ and the Witt index of the form $(\cdot ,\cdot)'$.  
In particular, $d=m$ means that the pair $(\G,\G')$ is in the stable range with $\G$ the smaller member.  
Then
\begin{equation}
\label{tauinverse0}
\tau^{-1}(0)=\Oo_m\cup\Oo_{m-1}\cup\dots\cup\Oo_0\,,
\end{equation}
where:
\begin{eqnarray}
&\text{\tiny$\bullet$}& \Oo_k\subseteq \Hom(\Vv,\Vv') \ \text{ is the subset of elements with isotropic range and 
rank $k$\,,} \nn \\
&\text{\tiny$\bullet$}& \Oo_k\cup\Oo_{k-1}\cup\dots\cup\Oo_0 \ \text{ is the closure of $\Oo_k$ for $0\leq k\leq m$,}
\nn  \\
\label{structure of t'-1tau(0)0}
&\text{\tiny$\bullet$}& \dim\Oo_k=\dim_\R(\Dc)\cdot((d'-k)k+(d-k)d)+\dim_\R \mathcal{H}_k(\Dc)
\end{eqnarray}
and
$$
\dim_\R \mathcal{H}_k(\Dc)=\dim_\R(\Dc)\cdot \frac{k(k-1)}{2}+k
$$
is the dimension,  over $\R$, of the space $\mathcal{H}_k(\Dc)$ of hermitian matrices of size $k$ 
with entries in $\Dc$.

Set $\Oo'_k=\tau'(\Oo_k)$. Then
$$
\tau'\tau^{-1}(0)=\mathcal O'_m\cup\mathcal O'_{m-1}\cup\dots\cup\mathcal O'_0
$$
where:
\begin{eqnarray}\label{structure of t'-1tau(0)1}
&\text{\tiny$\bullet$}& \mathcal O'_k\cup\mathcal O'_{k-1}\cup\dots\cup\mathcal O'_0 \ \text{is the closure of $\mathcal O'_k$ for $0\leq k\leq m$\,,} \nn \\
&\text{\tiny$\bullet$}& \dim\mathcal O'_k=d'k\,\dim_\R(\Dc)-2\dim_\R\SHs_k(\Dc)\,,
\end{eqnarray}
and
\begin{equation}\label{structure of t'-1tau(0)2}
2\dim_\R\SHs_k(\Dc)=\left\{
\begin{array}{lll}
k(k-1)\ &\ \text{if}\ \ \Dc=\R,\\
2k^2\ &\ \text{if}\ \ \Dc=\C,\\
2k(2k+1)\ &\ \text{if}\ \ \Dc=\Ha
\end{array}\right.
\end{equation}
is twice the dimension,  over $\R$, of the space $\SHs_k(\Dc)$ of skew-hermitian matrices of size $k$ with entries in $\Dc$.
\end{lemma}
For an open set $U$ in a finite dimensional real vector space and $t>0$ such that $tU\subseteq U$, let $M_t^*:
\D'(U)\to \D'(U)
$ 
denote the pullback of distributions defined by the submersion $M_t: U\ni v\to tv\in U$, \cite[Example 6.1.4]{Hormander}. In particular a distribution $u\in 
\D'(U)
$ is 
homogeneous of 
degree $a\in\C$ if $M_t^*u=t^au$
for every $t>0$.

\begin{lemma}\label{muSNK as a tempered homogeneous distribution}
For each $k=0, 1, 2, \dots, m$, the orbital integral $\mu_{\Oo_k}$ is a $\G\G'$-invariant distribution on $\Wv$, homogeneous of degree $\deg \mu_{\Oo_k}=\dim \Oo'_k-\dim \Wv$. 
\end{lemma}
Recall the embedding of the metaplectic group $\wt\Sp(\Wv)$ into the space of the tempered distributions $\Ss'(\Wv)$,
\begin{equation}
\label{T}
T:\wt\Sp(\Wv)\to \Ss'(\Wv)\,,
\end{equation}
\cite[Definition 4.23]{AubertPrzebinda_omega} and the corresponding Weil representation 
\cite[Theorem 4.27]{AubertPrzebinda_omega}. Let $\wt\G$ be the preimage of $\G$ in $\wt\Sp(\Wv)$.

The main goal of this article is to prove the following theorem and its corollary.
\begin{thm}
\label{the dilation limit of intertwining distribution, intro}
Let $\Theta_{\Pi}$ be the character be an irreducible representation $\Pi$ of $\wt\G$ that occurs in the restriction of the Weil representation to $\wt\G$.
Then, in the topology of $\Ss'(\Wv)$,
\[
t^{\deg \mu_{\Oo_m}}M_{t^{-1}}^* T(\check\Theta_{\Pi})\underset{t\to 0+}{\longrightarrow}C\mu_{\Oo_m},
\]
where $C\ne 0$, 
\[
T(\check\Theta_{\Pi})=
\int_{\G} \Theta_{\Pi}(\t g^{-1}) T(\t g)\,dg\,,
\]
$dg$ is a Haar measure on the group $\G$ and the product $\Theta_{\Pi}(\t g^{-1}) T(\t g)$ does not depend on the element $\t g$ in the preimage of $g$ in $\G$.
\end{thm}

\begin{remark}
If $\,\Pi$ is an irreducible admissible representation of a real reductive group $\G$ with Gelfand-Kirillov dimension $\kappa$, then \cite{BarVogAs} shows that there is a function $u_\kappa$, 
homogeneous of degree $-\kappa$ and defined on the set $\g^{rs}$ of regular semisimple elements of the Lie algebra $\g$ 
of $\G$, such that
\begin{equation}
\label{BV}
\lim_{t\to 0^+} t^\kappa \Theta_\Pi(\exp(tx))=u_\kappa(x) \qquad (x\in\g^{rs})\,. 
\end{equation}
The function $u_\kappa$ extends to a tempered distribution on $\g$. Its Fourier transform 
is a sum of nilpotent orbital integrals over nilpotent orbits of the same dimension $2\kappa$. However, the Fourier transform of the left-hand side of \eqref{BV} might even not be well defined. 
On the other hand, Theorem \ref{the dilation limit of intertwining distribution, intro} shows that for $\G$ compact, $T(\check{\Theta}_\Pi)$ admits an asymptotic limit, and the limit is a nilpotent orbital integral on $\Wv$. 
\end{remark}

The limit in Theorem \ref{the dilation limit of intertwining distribution, intro} was previously computed in \cite[Theorem 6.12]{Przebindaunitary}, even for dual pairs with a noncompact $\G$, but only on an open dense subset of $\Wv$. The explicit formula for the intertwining distribution 
from \cite{McKeePasqualePrzebindaWCSymmetryBreaking} -- see also section \ref{A limit of an intertwining distribution} -- allows us to compute the limit on the entire space $\Wv$.

Let $\Pi'$ be the irreducible representation of $\wt\G'$ corresponding to $\Pi$ in the 
Howe's correspondence. As a corollary of Theorem \ref{the dilation limit of intertwining distribution, intro}, we obtain an elementary computation of $WF(\Pi')$, the wave front of the character $\Theta_{\Pi'}$ at the identity.

\begin{corollary}\label{noproofWF}
For any representation $\Pi\otimes\Pi'$ that occurs in the restriction of the Weil representation to the dual pair $(\wt\G, \wt\G')$, 
\[
WF(\Pi')=\tau'(\tau^{-1}(0))=\overline{\Oo_m}\,.
\]
\end{corollary}

In \cite{Przebindaunitary}, the wave front set was determined using a computation of the Gelfand-Kirillov dimension
and Vogan's results in \cite{VoganGelfand}. For completeness, one should also recall that this dimension was independently computed in \cite{Przebindaunitary}, \cite{notyk} and \cite{EnrightWillenbring2004}. 
In this paper, we do not use the notion of Gelfand-Kirillov dimension. 
  
The proofs of Theorem \ref{the dilation limit of intertwining distribution, intro} and Corollary \ref{noproofWF} are given in sections \ref{A limit of an intertwining distribution} and \ref{WFSPI'}, respectively.  

%\newpage

%\section{List of notation}
%\mbox{}

\begin{center}
\includegraphics[scale=.24]{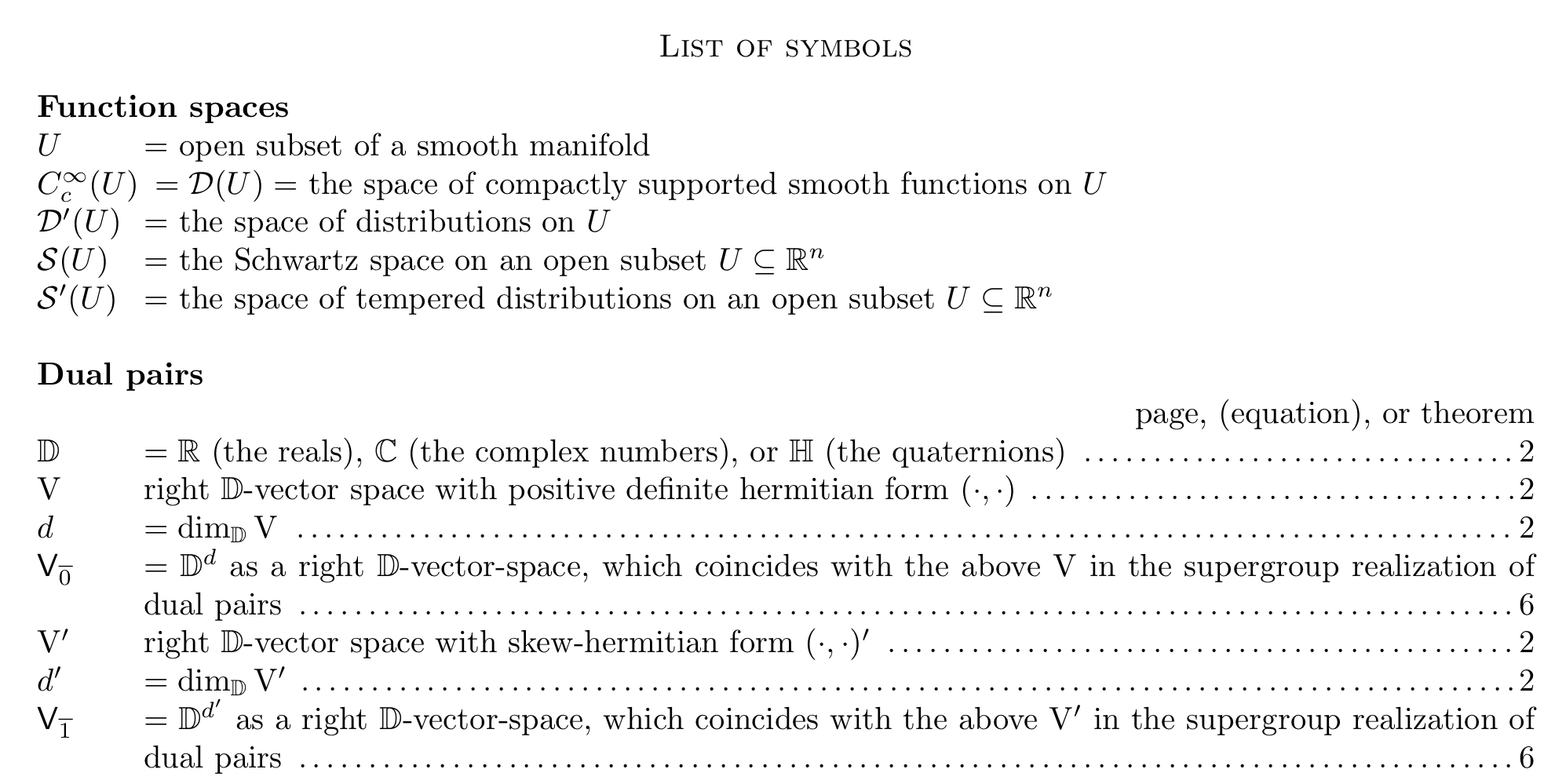}
\end{center}

\begin{center}
\includegraphics[scale=.47]{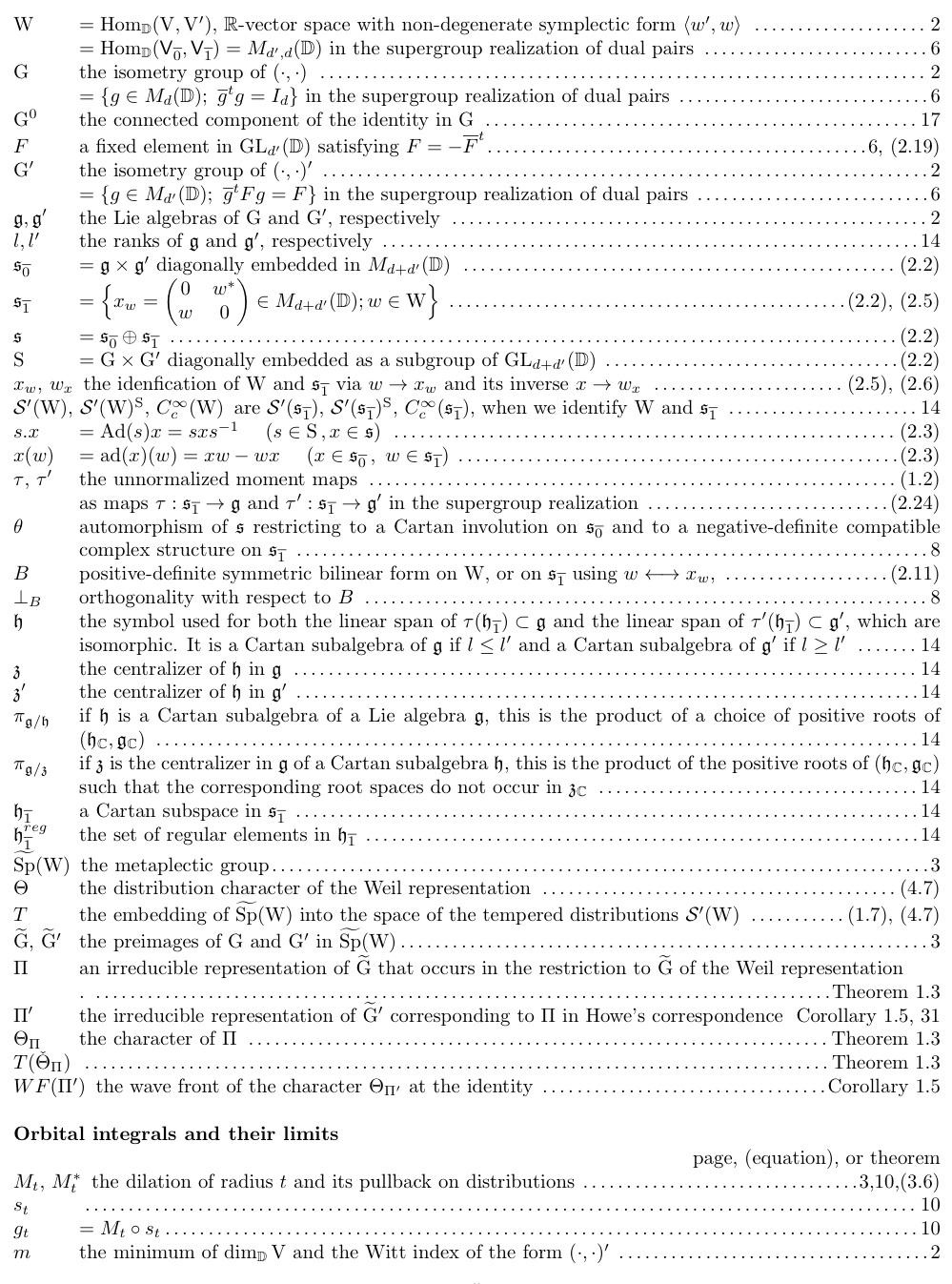}
\end{center}

\begin{center}
\includegraphics[scale=.32]{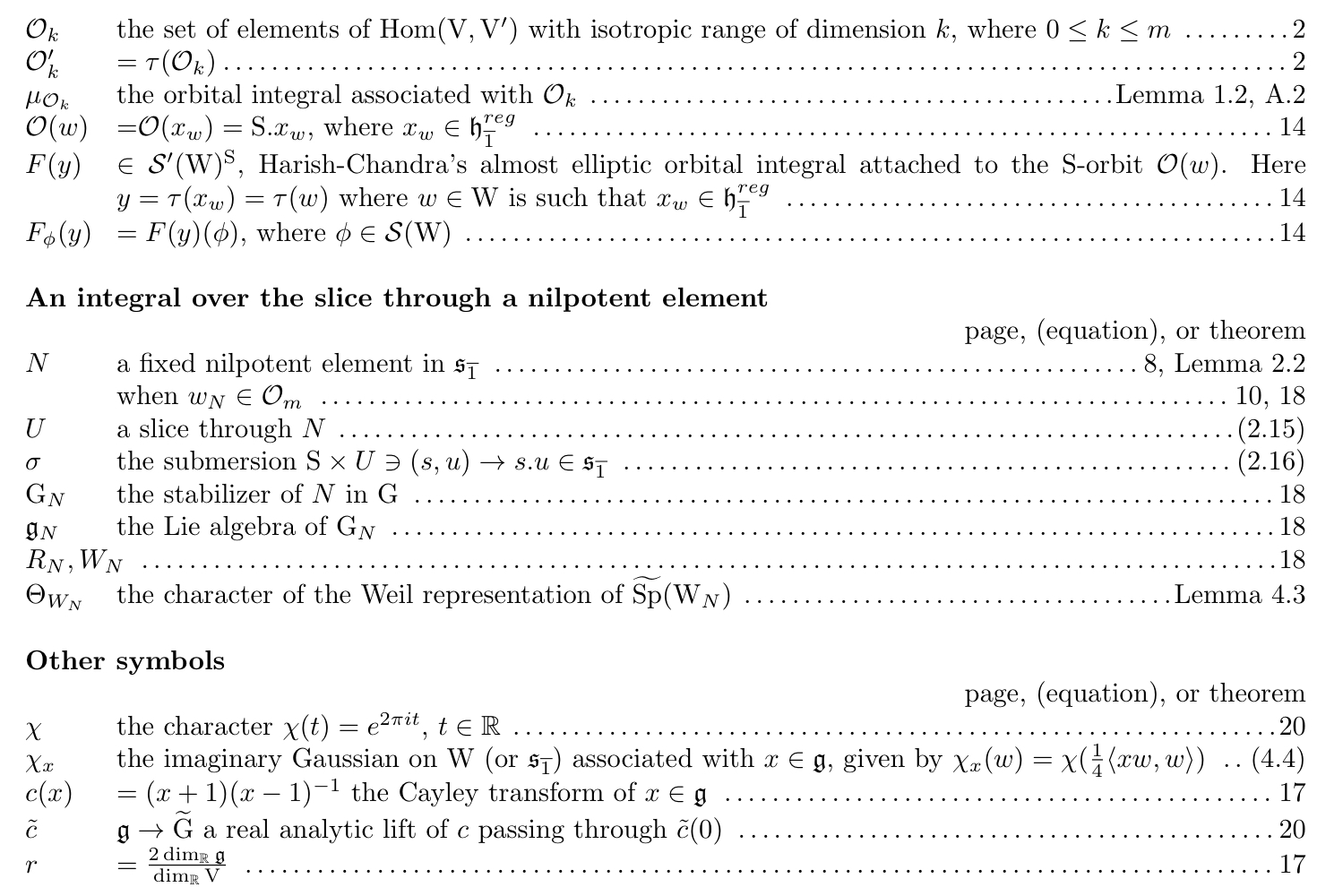}
\end{center}

\section{\textbf{A slice through a nilpotent element in the symplectic space}}
\label{A slice through nilpotent element}
\setcounter{thh}{0}
We will need the realization of the dual pair $(\G, \G')$ as a supergroup $(\Sg, \mathfrak s)$, \cite{PrzebindaLocal}. 
We present it in terms of matrices. 

Consider $\V_{\overline 0}=\Dc^d$ as a 
right vector space over $\Dc$ via
\[
av:=v a
 \qquad (v\in\V_{\overline 0},\ a\in \Dc)\,.
\]
The space $\End_\Dc(\V_{\overline 0})$ may be identified with the space of square matrices $M_d(\Dc)$ acting on $\Dc^d$ via left multiplication.
Let 
\[
(v,v')=\overline v^tv' \qquad (v,v'\in \Dc^d\,).
\]
This is a positive definite hermitian form on $\Dc^d$. The isometry group of this form is
\[
\G=\{g\in M_d(\Dc);\ \overline g^tg=I_d\}\,.
\]
Similarly, $\V_{\overline 1}=\Dc^{d'}$ is a left vector space over $\Dc$ and
\[
\G'=\{g\in M_{d'}(\Dc);\ \overline g^tFg=F\}\,,
\]
for a suitable $F=-\overline F^t\in \GL_{d'}(\Dc)$. This is the isometry group of the form
\[
(v,v')'=\overline v^tFv' \qquad (v,v'\in \Dc^{d'})\,.
\]
Set
\[
\Wv=\Hom_\Dc(\V_{\overline 0}, \V_{\overline 1})=M_{d',d}(\Dc)\,,
\]
with symplectic form
\begin{equation}\label{thesymplecticform}
\langle w', w\rangle =\tr_{\Dc/\R}(w^*w') \qquad (w,w'\in M_{d',d}(\Dc))\,,
\end{equation}
where $w^*=\overline w^tF$. 
Let 
\begin{eqnarray}
\label{super liealgebra}
&&\so=\g\times \g' \quad\text{diagonally embedded in $M_{d+d'}(\Dc)$}\,, \nn\\
&&\ss1=\Big\{ \begin{pmatrix}
0 & w^*\\
w & 0
\end{pmatrix} \in M_{d+d'}(\Dc); w\in \Wv\Big\}\,,\\
&&\Sg=\G\times\G' \quad\text{diagonally embedded as a subgroup of $\GL_{d+d'}(\Dc)$}\,. \nn
\end{eqnarray}
Then $(\Sg, \mathfrak s)$ is a real Lie supergroup, i.e. a real Lie group $\Sg$ together with a real Lie superalgebra $\mathfrak s=\so\oplus \ss1$, whose even component $\so$ is the Lie algebra of $\Sg$. We denote by $[\cdot,\cdot]$ the Lie superbracket on $\mathfrak{s}$. It agrees with the Lie bracket on $\so$ and with the anticommutator $\{x,y\}=xy+yx$ on $\ss1$. 
%Given $x\in \ss1$, its anticommutant in $\ss1$ is 
%$$
%\null^x\ss1=\{z\in \ss1:[x,z]=0\}\,.
%$$

The group $\Sg$ acts on $\mathfrak s$ by conjugation. We shall employ the notation 
\begin{eqnarray}
\label{S adjoint action on s}
&&s.x=\Ad(s)x=sxs^{-1}  \qquad (s\in \Sg\,,\ x\in \mathfrak{s})\,,\\
\label{s0 adjoint action on s}
&&x(w)=\ad(x)(w)=xw-wx \qquad (x\in \so\,,\ w\in \ss1)\,.
\end{eqnarray}
We shall also write
\begin{equation}
\label{x_w}
\Wv=M_{d',d}(\Dc)\in w \longrightarrow x_w=\begin{pmatrix}
0 & w^*\\
w & 0
\end{pmatrix}\in \ss1
\end{equation}
for the natural vector space isomorphism between $\Wv$ and $\ss1$, 
and 
\begin{equation}
\label{w_x}
\Wv=M_{d',d}(\Dc)\in w_x \longleftarrow x\in \ss1
\end{equation}
for its inverse. 
Under this isomorphisms, the adjoint action of $g\in \G\subseteq \Sg$ on $\ss1$ becomes the action on $\Wv$ by right multiplication by $g^{-1}$.
Similarly, the adjoint action of $g'\in \G'\subseteq \Sg$ on $\ss1$ becomes the action on
$\Wv$ by left multiplication by $g'$. Explicitly,
\begin{eqnarray}
\label{adjoint-actionG}
&&g.x_w=x_{wg^{-1}} \qquad (g\in \G, w\in \Wv)\,,\\
\label{adjoint-actionG'}
&&g'.x_w=x_{g'w} \qquad (g'\in \G', w\in \Wv)\,.
\end{eqnarray}
For an endomorphism $h\in \End(\Wv)$, we denote by the same symbol the corresponding endomorphism of $\ss1$, given by
\begin{equation}
\label{homom-correspondence}
h(x_w)=x_{h(w)} \qquad (w\in \Wv)\,.
\end{equation}
Notice that 
two elements $w,w'\in M_{d',d}(\Dc)$, viewed as members of $\ss1$, anticommute if and only if
\begin{equation}\label{anticommutation relations in *}
ww'{}^*+w'w^*=0\ \ \text{and}\ \ w^*w'+w'{}^*w=0\,.
\end{equation}

\begin{remark}
\label{rem:notation_W}
The unified realization of the dual pair and the symplectic space in the Lie supergroup $(\Sg,\ss1)$ is convenient in many computations. Distinguishing between the symplectic space $\Wv$ and its 
isomorphic space $\ss1$
 makes the matrix algebra more transparent. Still, most of the representation-theoretic applications of Howe duality prefer focusing on the symplectic space $\Wv$ rather than on $\ss1$. So, later in the paper, when working on orbital integrals in section \ref{Limits of orbital integrals}, we will choose to come back to the symplectic picture, which in practice corresponds to identifying $\Wv$ and $\ss1$ under the isomorphism \eqref{x_w}. With this identification, we will for instance write $g.w$, $g'.w$ or $s.w$ instead of $g.x_w$, $g'.x_w$ or $s.x_w$, as we did in \eqref{adjoint-actionG} and \eqref{adjoint-actionG'}. Correspondingly, the $\Sg$-orbit $\Sg.x_w$ of $x_w\in \ss1$ will be written $\Sg.w$, called the $\Sg$-orbit of $w\in \Wv$ and denoted $\mathcal{O}(w)$. This identification will allow us to refer to the existing literature on the subjet without any serious change of notation.
\end{remark}

We denote by $\theta$ the automorphism of $\mathfrak s$ defined in \cite[sec. 2.1]{PrzebindaLocal}. 
See also \cite[\S 5.3]{DaszKrasPrzebindaK-S2}. 
The construction of $\theta$ is done case-by-case and we shall not need these details. 
It can also be found in \cite[Proposition 1.1 and \S 2]{BaoSun_Algebraic-Smooth}. 
Its restriction to $\so$ is a Cartan involution and the restriction of $-\theta$ 
to $\ss1$ is a positive definite compatible complex structure. Using \eqref{x_w}, we can think of $\theta$ and $\langle \cdot ,\cdot \rangle$ as maps either on $\ss1$ or $\Wv$. The bilinear form $B(\cdot ,\cdot )=-\langle \theta\cdot ,\cdot \rangle$ is symmetric and positive definite. 
Moreover, $-\theta(w)=F^{-1}w$ 
for $w\in \Wv$. Hence
\begin{equation}\label{TheFormBonMatrices}
B(w', w)=\tr_{\Dc/\R}(\overline w^t w')    
\qquad 
(w,w'\in \Wv)\,.
\end{equation}
We can now get into the topic of this section. 
Fix an element $N\in \ss1$. Then $N+[\so,N]\subseteq \ss1$ may be thought of as the tangent space at $N$ to the $\Sg$-orbit in $\ss1$ through $N$. Denote by $[\so,N]^{\perp_B}\subseteq \ss1$ the $B$-orthogonal complement of $[\so,N]$. Since the form $B$ is positive definite, we have a direct sum orthogonal decomposition
\begin{equation}\label{decomposition of ss1 with respect to N}
\ss1=[\so,N]\oplus [\so,N]^{\perp_B}\,.
\end{equation}
Consider the map
\begin{equation}\label{the main submersion}
\sigma:\Sg\times \left(N+ [\so,N]^{\perp_B}\right)\ni (s,u)\to s. u\in \ss1\,.
\end{equation}
The derivative 
of $\sigma$ 
at $(s,u)$ coincides with the following linear map:
\[
\so\oplus[\so,N]^{\perp_B}\ni (X,Y)\to [X,s.u]+s.Y\in \ss1\,.
\]
Therefore the range of the derivative of $\sigma$ at $(s,u)$ is equal to
\begin{equation}\label{the range of the derivative}
[\so,s.u]+s.[\so,N]^{\perp_B}=s.\left([\so,u]+[\so,N]^{\perp_B}\right)\,.
\end{equation}
Let
\begin{equation}\label{definition of U}
U=\{u\in N+ [\so,N]^{\perp_B};\ [\so,u]+[\so,N]^{\perp_B}=\ss1\}\,.
\end{equation}
%%
%The equality (\ref{decomposition of ss1 with respect to N}) implies that  $N\in U$ and $U$ is the slice through $N$.
Then $U$ is the maximal open neighborhood of $N$ in $N+ [\so,N]^{\perp_B}$ such that the map
\begin{equation}\label{the main submersion on U -1}
\sigma:\Sg\times U\ni (s,u)\to s.u\in \ss1
\end{equation}
is a submersion. Therefore $\sigma(\Sg\times U)\subseteq \ss1$ is an open $\Sg$-invariant subset and
\begin{equation}\label{the main submersion on U}
\sigma:\Sg\times U\ni (s,u)\to s.u\in \sigma(\Sg\times U)
\end{equation}
is a surjective submersion. The title of this section refers to the set $U$ and a nilpotent element $N\in\ss1$. Here, nilpotent means nilpotent as a matrix; see \eqref{super liealgebra}. Notice that $N\in \ss1$ is nilpotent if and only if $\tau(w_N)\in \g$ is nilpotent, i.e. equal to
$0$ since $\G$ is compact. By \eqref{tauinverse0}, it follows that 
$w_N\in \Oo_k$ for some 
$k\in \{0,1,\dots,m\}$.
We shall use the map (\ref{the main submersion on U}) to study the $\Sg$-orbital integrals in $\ss1$.

\begin{lemma}\label{lemma I.4}
Keep the notation of Lemma \ref{structure of t'-1tau(0)}, and let $N\in \ss1$ such that $w_N\in \Oo_k$.
Then the map
\begin{equation}\label{lemma I.4.1}
N+ [\so,N]^{\perp_B}\ni u\to u^2\in \so
\end{equation}
is proper (i.e. the preimage of a compact set is compact).
\end{lemma}
\begin{prof}
We can choose the matrix $F$ as follows:
\begin{equation}\label{matrix F.I}
F=\begin{pmatrix}
0 & 0 & I_k\\
0 & F' & 0\\
-I_k & 0 & 0
\end{pmatrix}
\end{equation}
with $0\leq k\leq m$, where $m$ is the minimum of $d$ and the Witt index of the form $(\cdot,\cdot )'$, as in Lemma \ref{structure of t'-1tau(0)}, and $F'$ is a suitable element in $\GL_{d'-2k}(\Dc)$ satisfying $F'=-\overline{F'}^t$. Then, with the block decomposition of an element $M_{d',d}(\Dc)=M_{d',k}(\Dc)\oplus M_{d',d-k}(\Dc)$ dictated by (\ref{matrix F.I}),
\[
\left(
\begin{array}{lll}
w_1 & w_4\\
w_2 & w_5\\
w_3 & w_6
\end{array}
\right)^*=
\left(
\begin{array}{lll}
-\overline w_3^t & \overline w_2^tF' & \overline w_1^t\\
-\overline w_6^t & \overline w_5^tF' & \overline w_4^t
\end{array}
\right)\,.
\]
By the assumptions, we may choose $N=\begin{pmatrix}
0 & w_N^*\\
w_N & 0
\end{pmatrix}$ where
\begin{equation}\label{the nilpotent element}
w_N=\left(
\begin{array}{lll}
I_k & 0\\
0 & 0\\
0 & 0
\end{array}
\right)\,.
\end{equation}
Notice that 
\[
[\so,N]^{\perp_B}=\theta\left([\so,N]^{\perp}\right)=\theta\left({}^N\ss1\right)={}^{\theta N}\ss1\,,
\]
where $``\perp"$ is the orthogonal complement with respect to the symplectic form and the second equality is taken from \cite[Lemma 3.1]{PrzebindaLocal}. Since, 
\[
w_{\theta N}=-F^{-1}w_N=
\left(
\begin{array}{lll}
0 & 0\\
0 & 0\\
-I_k & 0
\end{array}
\right)\,
\]
a straightforward computation using (\ref{anticommutation relations in *}) shows that
$[\so,N]^{\perp_B}=\left\{x\in \ss1; w_x\in \Wv_{[\so,N]^{\perp_B}}\right\}$, where
\begin{equation}\label{the the set U}
\Wv_{[\so,N]^{\perp_B}}=\left\{
w=\left(
\begin{array}{lll}
0 & 0\\
0 & w_5\\
w_3 & w_6
\end{array}
\right) \in \Wv; \, w_3=-\overline w_3^t
\right\}\,.
\end{equation}
Let $x=x_w$ with $w$ as in \eqref{the the set U}. Then the image of $N+x$ under the map (\ref{lemma I.4.1}) consists of pairs of matrices
\begin{equation}\label{I.first eq}
\left(
\begin{array}{lll}
I_k & 0\\
0 & w_5\\
w_3 & w_6
\end{array}
\right)
\left(
\begin{array}{lll}
I_k & 0\\
0 & w_5\\
w_3 & w_6
\end{array}
\right)^*
=
\left(
\begin{array}{lll}
w_3 & 0 & I_k\\
-w_5\overline w_6^t & w_5\overline w_5^t F' & 0\\
-w_3\overline w_3^t-w_6\overline w_6^t & w_6\overline w_5^t F' & w_3
\end{array}
\right)\in \g'
\end{equation}
and
\begin{equation}\label{I.second eq}
\left(
\begin{array}{lll}
I_k & 0\\
0 & w_5\\
w_3 & w_6
\end{array}
\right)^*
\left(
\begin{array}{lll}
I_k & 0\\
0 & w_5\\
w_3 & w_6
\end{array}
\right)
=
\left(
\begin{array}{lll}
2w_3 & w_6\\
-\overline w_6^t & \overline w_5^t F' w_5
\end{array}
\right)\in \g\,.
\end{equation}
If the set of these pairs varies through a compact set, so do the $w_3$, $w_6$ and $w_5\overline w_5^t F'$.
Hence the claim follows.
\end{prof}

The maps $\tau$, $\tau'$ from \eqref{tautau'} can be considered as maps 
$\tau:\ss1\to \g$ and $\tau':\ss1 \to \g'$
by setting
\begin{equation}
\label{tau-bis}
\tau(x_w)=\tau(w)=w^*w \quad\text{and}\quad  \tau'(x_w)=\tau'(w)=ww^* \qquad (w\in \Wv)\,, 
\end{equation}
or equivalently, 
$$
\tau(x)=x^2|_{\V_{\overline{0}}} \quad\text{and}\quad \tau'(x)=x^2|_{\V_{\overline{1}}} 
\qquad (x\in \ss1)\,,
$$   
where $|_{\V_{\overline{0}}}$ and $|_{\V_{\overline{1}}}$ respectively indicate the selection of the upper diagonal block of size $d$ or the lower diagonal block of size $d'$. 

\begin{corollary}\label{tau is proper}
If $k=m$, then the restriction $\tau|_{N+[\so,N]^{\perp_B}}$ of
$\tau:\ss1\to \g$ to $N+[\so,N]^{\perp_B}$
is proper.
\end{corollary}
\begin{prof}
This follows from the formula (\ref{I.second eq}). Indeed, it is enough to see that the map
\[
w_5\to \overline w_5^t F' w_5
\]
is proper. The variable $w_5$ does not exist unless $\Dc=\C$ and $d>m$. This means that $m$ is the Witt index of the form $(\cdot,\cdot )'$. Hence $iF'$ is a definite hermitian matrix. Therefore the above map is proper.
\end{prof}
\begin{corollary}\label{tau is proper again}
Suppose $k=m$.
If $E\subseteq \ss1$ is a subset such that $\tau(E)\subseteq \g$ is bounded, then
\[
E\cap \left( N+[\so,N]^{\perp_B}\right)
\]
is bounded.
\end{corollary}
\begin{prof}
This is immediate from Corollary \ref{tau is proper}.
\end{prof}
\section{\textbf{Limits of orbital integrals}}
\label{Limits of orbital integrals}
\setcounter{thh}{0}
Since we are interested in $\Sg$-invariant distributions, we want to see dilations by $t>0$ in $\ss1$ as transformations in the slice $U$ modulo the adjoint action of the group $\Sg$. This will be accomplished in Lemma \ref{lemma about g(t)} below.

For $t>0$ let
\[
s_t=\left(
\begin{array}{lll}
t^{-1}I_k & 0 & 0\\
0 & I & 0\\
0 & 0 & tI_k
\end{array}
\right)
\]
where the blocks are as in (\ref{matrix F.I}). Then $s_t\in \G'$.
Define isomorphisms $s_t$, $M_t$, $g_t$ of $\Wv=M_{d',d}(\Dc)$ by
\begin{eqnarray*}
&&s_t(w)=s_tw \qquad (w\in \Wv)\,,\\
&&M_t(w)=tw \qquad (w\in \Wv)\,,
\end{eqnarray*}
and $g_t=M_t\circ s_t$, i.e. 
$$
g_t(w)=ts_tw \qquad (w\in \Wv)\,. 
$$
Explicitly, %%
\begin{equation}\label{gt acting on NperpB}
g_t\left(
\begin{array}{lll}
w_1 & w_4\\
w_2 & w_5\\
w_3 & w_6
\end{array}
\right)
=
\left(
\begin{array}{lll}
w_1 & w_4\\
tw_2 & tw_5\\
t^2 w_3 & t^2w_6
\end{array}
\right).
\end{equation}
We denote by the same symbols the corresponding linear isomorphisms of $\ss1$, as in \eqref{homom-correspondence}. In particular, 
$$
g_t(x)=t s_t.x \qquad (x\in \ss1)\,.
$$

\begin{lemma}\label{lemma about g(t)}
The linear map $g_t\in \GL(\ss1)$ preserves $[\so,N]^{\perp_B}$, $N+[\so,N]^{\perp_B}$ and the subset $U\subseteq N+[\so,N]^{\perp_B}$ defined in \eqref{definition of U}. In fact,
\begin{equation}\label{gt acting on NperpB11}
\tau|_U\circ g_t|_U=M_{t^2}\circ \tau|_U\,.
\end{equation}
Furthermore, for $\sigma$ as in \eqref{the main submersion on U -1},
\begin{equation}\label{gt acting on NperpB and sigma}
g_t\circ \sigma=\sigma\circ(\Ad(s_t) \times g_t|_{N+[\so,N]^{\perp_B}})\,,
\end{equation}
where $g_t|_{N+[\so,N]^{\perp_B}}$ on the right-hand side stands for the restriction of $g_t$ to $N+[\so,N]^{\perp_B}$. 
In particular, the subset $\sigma(\Sg\times U)\subseteq \ss1$ is closed under multiplication by  positive reals. Moreover, the determinant of the derivative $g_t'$ of the map $g_t:\ss1\to\ss1$ is
\begin{equation}\label{det gt}
\det(g_t')=t^{\dim\ss1}\,,
\end{equation}
and 
\begin{equation}\label{det gt'}
\det((g_t|_{N+[\so,N]^{\perp_B}})')=t^{\dim\ss1-\dim\Oo'_k}\,.
\end{equation}
\end{lemma}
\begin{prof}
The preservation of $[\so,N]^{\perp_B}$ and $N+[\so,N]^{\perp_B}$ follows from \eqref{gt acting on NperpB}, \eqref{the nilpotent element} and \eqref{the the set U}. The equality (\ref{gt acting on NperpB11}) follows from (\ref{gt acting on NperpB}) and (\ref{I.second eq}). Notice that
$$
\Big[ \begin{pmatrix}
y & 0\\ 0 & y'
\end{pmatrix}, g_t u\Big]=g_t \Big[ \begin{pmatrix}
y & 0\\ 0 & \Ad(s_t^{-1})y'
\end{pmatrix}, u\Big]
\qquad (y\in\g,\, y'\in \g',\, t>0,\, u\in U)\,.
$$
So
\[
[\so,g_t u]=g_t[\so,u] \qquad (t>0,\, u\in U)\,.
\]
Hence
\[
[\so,g_t u]+[\so,N]^{\perp_B}=[\so,g_t u]+g_t[\so,N]^{\perp_B}=g_t([\so,u]+[\so,N]^{\perp_B})\,.
\]
This implies that the set $U$ is also preserved.

To verify (\ref{gt acting on NperpB and sigma}),
we notice that for $s\in \Sg$ and $u\in N+[\so,N]^{\perp_B}$ we have
\begin{eqnarray*}
g_t\circ \sigma(s,u)&=&g_t(s.u)=t(s_ts).u=(s_tss_t^{-1}).(ts_t.u)\\
&=&\sigma(s_tss_t^{-1}, g_tu)=\sigma\circ(\Ad(s_t) \times g_t|_{N+[\so,N]^{\perp_B}})(s,u)\,.
\end{eqnarray*}
Fix $t>0$. The conjugation by $s_{t^{-1}}$ preserves $\sigma(\Sg\times U)$ because $s_{t^{-1}}\in\Sg$. Since multiplication by $t$ coincides with $g_t\circ s_{t^{-1}}$, \eqref{gt acting on NperpB and sigma} implies that $\sigma(\Sg\times U)$ is preserved under the multiplication by $t$. 

Since $g_t'=(M_t\circ s_t)'=M_t\circ s_t$ and since $\det s_t=1$, (\ref{det gt}) is obvious. 

In order to verify (\ref{det gt'}) we proceed as follows. 
The derivative of the map $g_t|_{N+[\so,N]^{\perp_B}}$ coincides with the following linear map
\[
\left(
\begin{array}{lll}
0 & 0\\
0 & w_5\\
w_3 & w_6
\end{array}
\right)
\to
\left(
\begin{array}{lll}
0 & 0\\
0 & tw_5\\
t^2w_3 & t^2w_6
\end{array}
\right)\,.
\]
%The determinant of this map is equal to 
%the determinant of 
%\[
%\left(
%\begin{array}{lll}
%0 & w_4\\
%0 & w_5\\
%w_3 & w_6
%\end{array}
%\right)
%\to
%\left(
%\begin{array}{lll}
%0 &  tw_4\\
%0 & tw_5\\
%t^2 w_3 & t w_6
%\end{array}
%\right)\,.
%\]
By \eqref{the the set U}, the determinant of this map is equal to 
\[
t^{2\dim_\R \SHs_k(\Dc)}t^{d'(d-k)\dim_\R\Dc}\,.
\]
Since, by (\ref{structure of t'-1tau(0)1}),
\[
2\dim_\R \SHs_k(\Dc)+d'(d-k)\dim_\R\Dc=\dim \ss1-\dim \Oo'_k\,,
\]
(\ref{det gt'}) follows.
\end{prof}

Next we consider an $\Sg$-invariant distribution $F$ on $\sigma(\Sg\times U)$. 
The following lemma proves that the restriction of $F$ to $U$ exists and that the restriction of the $t$-dilation of $F$ is equal to $(g_t|_U)^*$ applied to $F|_U$.

\begin{lemma}\label{invariant distributions and g(t)}
Suppose $F\in \mathcal D'(\sigma(\Sg\times U))^\Sg$. Then the intersection of the wave front set of $F$ with the conormal bundle to $U$ is zero, so that the restriction 
$F|_U$ is well defined. Furthermore, $\sigma^*F=\mu_\Sg\otimes F|_{U}$, where $\mu_\Sg$ is a Haar measure on $\Sg$. Moreover, for $t>0$,
\begin{equation}\label{invariant distributions and g(t)1}
M_t^*F=g_t^*F\,,
\end{equation}
and
\begin{equation}\label{invariant distributions and g(t)3}
(M_{t}^*F)|_U=(g_t|_U)^*F|_U\,.
\end{equation}
\end{lemma}
\begin{prof}
Since $s_t^*F=F$ %(because $s_t\in \Sg$) 
we see that $g_t^*F=M_t^*s_t^*F=M_t^*F$ and \eqref{invariant distributions and g(t)1} follows.

The wave front set of $F$ is contained in the union of the conormal bundles to the $\Sg$-orbits through elements of $\ss1$. This is because the characteristic variety of the system of differential equations expressing the condition that this distribution is annihilated by the action of the Lie algebra $\so$ coincides with that set. The intersection of this set with the conormal bundle to $U$ is zero. Indeed, at each point $u\in U$, this intersection is equal to the annihilator of both, the tangent space to $U$ at $u$ and  the tangent space to the $\Sg$-orbit through $u$. Since by \eqref{the main submersion on U -1} the map $\sigma$ is submersive,  these tangent spaces add up to the whole tangent space to $\ss1$ at $u$. Hence the annihilator is zero.
Therefore $F$ restricts to $U$.  The formula $\sigma^*F=\mu_\Sg\otimes F|_{U}$ follows from the diagram
\[
U\longrightarrow \Sg\times U \overset{\sigma}{\longrightarrow} \sigma(\Sg\times U),\ \ u \to (1,u)\to u\,,
\]
which shows that the restriction to $U$ equals the composition of $\sigma^*$ and the pullback via the embedding of $U$ into $\Sg\times U$.
By combining this with \eqref{invariant distributions and g(t)1} we deduce \eqref{invariant distributions and g(t)3}.
\end{prof}

The following lemma shows that the computation of limits of weighted dilations of $\Sg$-invariant distributions on $\Wv$ may be accomplished by computing weighted limits on the slice $U$.

\begin{lemma}\label{invariant distributions and g(t) and limits}
Suppose $F, F_0\in \mathcal D'(\sigma(\Sg\times U))^\Sg$ and $a\in\C$ are such that
\[
t^a(g_{t^{-1}}|_U)^*F|_U \underset{t\to 0+}{\longrightarrow}F_0|_U\,.
\]
Then
\[
t^a M_{t^{-1}}^*F\underset{t\to 0+}{\longrightarrow}F_0
\]
in $\mathcal D'(\sigma(\Sg\times U))$.
\end{lemma}
\begin{prof}
Proposition \ref{I.1} shows that it suffices to see that
\[
\sigma^*\left(t^a M_{t^{-1}}^*F\right)\underset{t\to 0+}{\longrightarrow}\sigma^*F_0\,.
\]
But Lemma \ref{invariant distributions and g(t)} implies
\[
\sigma^*\left(t^a M_{t^{-1}}^*F\right)=\mu_\Sg\otimes t^a(g_{t^{-1}}|_U)^*F|_U
\ \ \ \text{and}\ \ \ \sigma^*F_0=\mu_\Sg\otimes F_0|_{U}\,.
\]
Hence the claim follows.
\end{prof}

Now we are ready to compute the limit of the weighted dilatation of the unnormalized almost semisimple orbital integral $\mu_{\mathcal O}$.
\begin{proposition}\label{limit of orbits I}
Let $\mathcal O\subseteq \sigma(\Sg\times U)$ be an $\Sg$-orbit 
and let $\mu_{\mathcal O}\in \mathcal D'(\ss1)$ be the corresponding orbital integral. 
%Assume $\mu_{\mathcal O}$ is $\Sg$ - invariant.
Then
\begin{eqnarray}\label{the limit of regular orbital integral}
\underset{t\to 0+}{\lim}t^{\deg \mu_{\Oo_m}}M_{t^{-1}}^*\mu_{\Oo}|_{\sigma(\Sg\times U)}=\mu_{\mathcal O}|_{U}(U)\,\mu_{\Oo_m}|_{\sigma(\Sg\times U)}\,,
\end{eqnarray}
where $\mu_{\Oo_m}\in \mathcal D'(\sigma(\Sg\times U))$ is the orbital integral on the orbit $\Oo_m=\Sg. N$ normalized so that $\mu_{\Oo_m}|_{U}$ is the Dirac delta at $N$ and the convergence is in $\mathcal D'(\sigma(\Sg\times U))$. 
\end{proposition}
Before the proof, we make two remarks. 
First, the scalar $\mu_{\mathcal O}|_{U}(U)$ may be thought of as the volume of the intersection $\mathcal O\cap U$. This volume is finite because the restriction $\mu_{\mathcal O}|_{U}$ is a distribution on $U$ with support equal to the closure of $\mathcal O\cap U$, which is compact by Corollary \ref{tau is proper again}, since $\tau(\Oo)$ is a $\G$-orbit and therefore bounded. Hence $\mu_{\mathcal O}|_{U}$ applies to any smooth function on $U$, in particular to the indicator function $\Bbb I_U$, equal to $1$ on $U$. Thus $\mu_{\mathcal O}|_{U}(U)=\mu_{\mathcal O}|_{U}(\Bbb I_U)$.

The second remark is that our normalization of $\mu_{\Oo_m}$ does not depend on the normalization of $\mu_{\mathcal O}$, which is absorbed by the factor $\mu_{\mathcal O}|_{U}(U)$. 

\begin{prof}
By  the definition of pull-back and (\ref{det gt'}) 
\[
\mu_{\mathcal O}|_{U}(\psi\circ g_{t}|_U)=
t^{\dim \Oo'_m-\dim\ss1}(g_{t^{-1}}|_U)^*\mu_{\mathcal O}|_{U}(\psi)\,.
\]
(Indeed, for a distribution equal to a function $f(x)$ times the Lebesgue measure,
\begin{equation}
\label{transformation_gt}
g_{t^{-1}}^* f(\psi)=\int_{\ss1} f(g_{t^{-1}}x)\psi(x)\,dx=
|\det(g_{t}')| \int_{\ss1} f(x)\psi(g_{t}x)\,dx\,.
\end{equation}
Since $\mu_{\mathcal O}|_{U}$ is a limit of such functions, it has the same transformation property.) 
We see from (\ref{gt acting on NperpB}) that
\[
\underset{t\to 0}{\lim}\ g_tu=N \qquad (u\in U)\,.
\]
Hence, for any $\psi\in C_c^\infty(U)$,
\begin{equation}\label{additional1}
\underset{t\to 0}{\lim}\ \mu_{\mathcal O}|_{U}(\psi\circ g_t)=\mu_{\mathcal O}|_{U}(\psi(N)\Bbb I_U)
=\mu_{\mathcal O}|_{U}(\Bbb I_U)\psi(N)=\mu_{\mathcal O}|_{U}(U)\psi(N)\,.
\end{equation}
Replacing $\Oo$ with $\Oo_m$ in \eqref{additional1} we see that the restriction of $\mu_{\Oo_m}$ to $U$ is a multiple of the Dirac delta at $N$.
Thus (\ref{the limit of regular orbital integral}) follows from Lemma \ref{invariant distributions and g(t) and limits} with $F_0=\mu_{\Oo_m}$.
\end{prof}

Next, we want to compute the limit of the weighted dilations of the normalized almost elliptic orbital integrals. We need some additional notation. 
%%
% $F(y)$, we still need to compute the weight. This will be done in Corollary \ref{limit of orbital integrals and Harish-Chandra integral on W} below.
%%

For $x,y\in \mathfrak{s}_{\overline{1}}$, let $\{x,y\}=xy+yx\in \mathfrak{s}_{\overline{0}}$ denote their anticommutator. Let $x \in \ss1$ be fixed. 
The anticommutant and the double anticommutant of $x$ in $\ss1$ are  
\begin{eqnarray*}
\anticomm{x}{\ss1}&=&\{y \in \ss1:\{x,y\}=0\}\,,\\
\danticomm{x}{\ss1}&=&\bigcap_{y \in \anticomm{x}{\ss1}} \anticomm{y}{\ss1}\,,
\end{eqnarray*}
respectively. 
A semisimple element $x \in \ss1$ is said to be regular if it is nonzero and $\dim(S.x) \geq \dim(S.y)$ for all semisimple $y \in \ss1$.  
A Cartan subspace $\hs1$ of $\ss1$ is defined as the double anticommutant of a regular semisimple element $x \in \ss1$. The Cartan subspaces of $\ss1$ are classified in \cite[\S 6]{PrzebindaLocal}. See also \cite[\S 4]{McKeePasqualePrzebindaSuper} 
and \cite[\S 2.2]{McKeePasqualePrzebindaWCestimates} for additional information.
We denote by $\reg{\hs1}$ the set of regular elements in $\hs1$. 
As in \cite[(13)--(15)]{McKeePasqualePrzebindaWCestimates} 
the linear spans of $\tau(\h_{\overline 1})$ and $\tau'(\h_{\overline 1})$ will be identified and both denoted by $\h$. 

Let $l$ and $l'$ denote the ranks of $\g$ and $\g'$, respectively. 
Then $\h\subseteq\g$ is a Cartan subalgebra of $\g$ if $l\leq l'$ and $\h\subseteq\g'$ is a Cartan subalgebra of $\g'$ otherwise. One can check that $d>d'$ is equivalent to $l>l'$ except for $(\G,\G')=(\Og_{2l+1},\Sp_{2l'})$ with $l'=l$.

Let $\z\subseteq \g$ and $\z'\subseteq \g'$ be the centralizers of $\h$.
Suppose $\h$ is a Cartan subalgebra of $\g$ and fix a set of positive roots of $(\h_\C, \g_\C)$. 
Let $\pi_{\g/\h}$ denote the product of all positive roots and let $\pi_{\g/\z}$ denote the 
product of all positive roots such that the corresponding root spaces do not occur in $\z_\C$. 
Similar notations will be used when $\h$ is a Cartan subalgebra of $\g'$.

Harish-Chandra's almost elliptic orbital integral $F(y)\in \Ss'(\Wv)^\Sg$ attached to the $\Sg$-orbit $\mathcal{O}(w)$ was defined in \cite[Definition 3.2]{McKeePasqualePrzebindaWCestimates}.
Here $y\in \cup_{\h_{\overline 1}} \tau(\h_{\overline 1}^{reg})$, the union being on the family of 
mutually non-$\Sg$-conjugate Cartan subspaces of $\ss1$, and $w\in \Wv$ is such that $x_w\in \h_{\overline 1}^{reg}$ and $y=\tau(x_w)=\tau(w)$. Observe that, by classification, \cite[\S 6]{PrzebindaLocal}, all Cartan subspaces $\h_{\overline 1}\subseteq \ss1$ are $\Sg$-conjugate except when $(\G,\G')=(\Ug_l,\Ug_{p,q})$ with $l<p+q$. Besides these exceptional cases, the above union 
reduces therefore to one term. Following Harish-Chandra's notation, we shall write $F_\phi(y)$ for $F(y)(\phi)$, where $\phi\in \Ss(\Wv)$. 

As indicated in Remark \ref{rem:notation_W}, in the following we will adopt the notation from \cite{McKeePasqualePrzebindaWCestimates} (and references therein) and identify $\ss1$ and $\Wv$ by means of the isomorphism \eqref{x_w}.
So, for instance, $\mathcal{O}(w)$ means $\mathcal{O}(x_w)=\Sg.x_w$ and we write $w\in \h_{\overline 1}^{reg}$ instead of $x_w\in \h_{\overline 1}^{reg}$. Moreover, 
$\Ss'(\Wv)^\Sg=\Ss'(\ss1)^\Sg$, $\Ss'(\Wv)=\Ss'(\ss1)$, and $C_c^\infty(\Wv)=C_c^\infty(\ss1)$.

We refer to \cite[Theorems 3.4 and 3.6]{McKeePasqualePrzebindaWCestimates} for the differentiable extension and regularity properties of the map $y \to F(y)$. 
These properties of are different when $l>l'$ or $l\leq l'$. These two cases have therefore to be treated separately. 

In fact, when $l>l'$, then $F(y)$ turns out to be a constant multiple of Harish-Chandra's orbital integral; see \cite[(39)]{McKeePasqualePrzebindaWCestimates}. When $l\leq l'$, then $F(y)$ can still be related to Harish-Chandra's orbital integral, but the situation is more involved: the differential extension of $F(y)$, up to a specific order, is on the set $\h\cap \tau(\Wv)$. We refer to \cite[Theorem 3.6 and (72)]{McKeePasqualePrzebindaWCestimates} for more details. 

\begin{corollary}\label{limit of orbital integrals and Harish-Chandra integral on W}
Let $l>l'$.
Assume (for the construction of $U$) that $k=m$. Then, 
\begin{equation}\label{second limit formula}
\underset{t\to 0+}{\lim}\ t^{\deg\mu_{\Oo_m}} M_{t^{-1}}^*F(y)|_{\sigma(\Sg\times U)}=
F(y)|_{U}(U)\,\mu_{\Oo_m}|_{\sigma(\Sg\times U)}\,.
\end{equation}
\end{corollary}
\begin{prof}
The statement (\ref{second limit formula}) is immediate from Proposition \ref{limit of orbits I}. 
\end{prof}

As in \cite{HC-57DifferentialOperators}, we identify the symmetric algebra on $\g$ with $\C[\g]$, the algebra of the polynomials on $\g$, using the invariant symmetric bilinear form $B$ on $\g$.
\begin{lemma}\label{the main localized to U lemma}
Assume that $l\leq l'$. 
Let $y\in \h\cap \tau(\Wv)$ and let $Q\in \C[\h]$ be such that $\deg (Q)$ is small enough so that, by \cite[Theorem 3.6]{McKeePasqualePrzebindaWCestimates}, 
$\partial(Q)F(y)$ exists. Then
\begin{equation}\label{the main localized to U lemma1}
t^{\deg\mu_{\Oo_{m}}} M_{t^{-1}}^*\partial(Q)F(y)|_{\sigma(\Sg\times U)}
\underset{t\to 0+}{\longrightarrow }C\mu_{\Oo_m}
\end{equation}
in $\mathcal D'(\sigma(\Sg\times U))$, where $C=\partial(Q)F(y)|_{U}(\Bbb I_U)$ is the value of the  compactly supported distribution $\partial(Q)F(y)|_{U}$ on $U$ applied to the indicator function $\Bbb I_U$.
\end{lemma}
\begin{prof}
We see from Lemma \ref{invariant distributions and g(t) and limits} that it suffices to prove the lemma with (\ref{the main localized to U lemma1}) replaced by
\begin{equation}\label{the main localized to U lemma1'}
t^{\deg\mu_{\Oo_{m}}} \big(g_{t^{-1}}|_U\big)^*\partial(Q)F(y)|_{U}
\underset{t\to 0+}{\longrightarrow }C\delta_{N}|_U\,.
\end{equation}
Let $\psi\in C_c^\infty(U)$. Lemma \ref{muSNK as a tempered homogeneous distribution}, the argument of \eqref{transformation_gt}, and the equality \eqref{det gt'} show that
\[
t^{\deg\mu_{\Oo_{m}}} \big(g_{t^{-1}}|_U\big)^*\partial(Q)F(y)|_{U}(\psi)=\partial(Q)F(y)|_{U}(\psi\circ g_t)\,.
\]
Since $\partial(Q)F(y)|_{U}$ is a compactly supported distribution on $U$,
\begin{eqnarray*}
\partial(Q)F(y)|_{U}(\psi\circ g_t)&\underset{t\to 0+}{\longrightarrow }&\partial(Q)F(y)|_{U}(\psi(N)\Bbb I_U)\\
&=&\partial(Q)F(y)|_{U}(\Bbb I_U)\delta_{N}(\psi)\,.
\end{eqnarray*}
\vskip -6mm 
\end{prof}

Next we show that the convergence of Lemma \ref{the main localized to U lemma} happens not only in distributions in $\mathcal D'(\sigma(\Sg\times U))$ but also in 
$\Ss'(\Wv)$. This generalization will require Harish-Chandra's Regularity Theorem.
\begin{proposition}\label{the main limit pro}
Let $y\in \h\cap \tau(\Wv)$. If $l\leq l'$ let $Q\in \C[\h]$ be such that $\deg (Q)$ is small enough so that, by \cite[Theorem 3.6]{McKeePasqualePrzebindaWCestimates}, $\partial(Q)F(y)$ exists. If $l>l'$ set $\partial(Q)F(y)=F(y)$. Then, 
\begin{equation}\label{the main limit pro1}
t^{\deg\mu_{\Oo_{m}}} M_{t^{-1}}^*\partial(Q)F(y)
\underset{t\to 0+}{\longrightarrow }C\mu_{\Oo_m}
\end{equation}
in the topology of $\Ss'(\Wv)$, where $C=\partial(Q)F(y)|_{U}(\Bbb I_U)$. 
Moreover, there is a seminorm $q$ on $\Ss(\Wv)$ and $N\geq 0$ such that
\begin{multline}\label{the main limit pro1'}
\left|t^{\deg\mu_{\Oo_{m}}} M_{t^{-1}}^*\partial(Q)F_\phi(y)\right|\leq (1+|y|)^N q(\phi) 
\qquad (0<t\leq 1,\ y\in \h\cap \tau(\Wv),\ \phi\in \Ss(\Wv))\,.
\end{multline}
\end{proposition}
\begin{prof}
Since the pull-back
\[
\Ss(\g')\ni\psi\to \psi\circ \tau'\in\Ss(\Wv)
\]
is well defined and continuous, we have a push-forward of tempered distributions
\[
\Ss'(\Wv)\ni u\to \tau'_*u\in \Ss'(\g')
\,,\ \ \ \tau'_*u(\psi)=u(\psi\circ \tau')\,,
\]
see \cite[(6.1)]{PrzebindaUnipotent}.
If $l> l'$ then $\tau'_*(F(y))$ is a constant multiple of a semisimple orbital integral supported on the $\G'$-orbit through $y$ in $\g'$; see \cite[(39)--(40)]{McKeePasqualePrzebindaWCestimates}.  As  a distribution, it is annihilated by the ideal in $\C[\g']^{\G'}$ of the polynomials vanishing on that orbit. This is an ideal of finite codimension.

We shall prove a similar statement about $\tau'_*(\partial(Q)F(y))$ in the case $l\leq l'$. 
According to  \cite[(75) 
for $\G=\Og_{2l+1}$ with $l\leq l'$, and 
(72) otherwise]{McKeePasqualePrzebindaWCestimates}, we may complete $\h$ to an elliptic Cartan subalgebra $\h'=\h\oplus\h''\subseteq \g'$ and there is a positive constant $C$ such that for $\psi\in \Ss(\g')$
\begin{eqnarray}\label{the main limit pro2}
\tau'_*(\partial(Q)F(y))(\psi)&=&\partial(Q)\tau'_*(F(y))(\psi)\\
&=&C\, \partial(Q\t\pi_{\z'/\h'})\left(\pi_{\g'/\h'}(y+y'')
\int_{\G'}\psi(g.(y+y''))\,dg\right)\Big|_{y''=0}\,,\nn
\end{eqnarray}
where $y''\in\h''$, $\t\pi_{\z'/\h'}=\pi_{\z'/\h'}^{short}$ (the product of the positive short roots) if $\G=\Og_{2l+1}$ with $l<l'$, and $\t\pi_{\z'/\h'}=\pi_{\z'/\h'}$ otherwise. Let $P\in\C[\g']^{\G'}$. Then
\begin{eqnarray}\label{the main limit pro3}
&&\partial(Q\t\pi_{\z'/\h'})\left(\pi_{\g'/\h'}(y+y'')\int_{\G'}(P\psi)(g.(y+y''))\,dg\right)\Big|_{y''=0}\\
&=&\partial(Q\t\pi_{\z'/\h'})\left(P(y+y'')\pi_{\g'/\h'}(y+y'')\int_{\G'}\psi(g.(y+y''))\,dg\right)\Big|_{y''=0}\,.\nn
\end{eqnarray}
By commuting the operators of multiplication by a polynomial with differentiation, we may write
\[
\partial(Q\t\pi_{\z'/\h'})P(y+y'')=\sum_{|\alpha|\leq\deg(Q\t\pi_{\z'/\h'})}P_\alpha(y+y'')\partial^\alpha\,,
\]
where $\partial^\alpha=\prod_{j=1}^{l'}\partial(J_j')^{\alpha_j}$ for $\alpha=(\alpha_1,\dots, \alpha_{l'})$. Hence, (\ref{the main limit pro3}) is equal to
\begin{eqnarray}\label{the main limit pro4}
\sum_{|\alpha|\leq\deg(Q\t\pi_{\z'/\h'})}P_\alpha(y)\partial^\alpha\left(\pi_{\g'/\h'}(y+y'')\int_{\G'}\psi(g.(y+y''))\,dg\right)\Big|_{y''=0}\,.
\end{eqnarray}
We see from (\ref{the main limit pro2})--(\ref{the main limit pro4}) that the range of the map
\begin{eqnarray}\label{the main limit pro5}
\C[\g']^{\G'}\ni P \to \tau'_*(\partial(Q)F(y))\cdot P\in \Ss'(\g')
\end{eqnarray}
is contained in the space spanned by the distributions
\[
\partial^\alpha\left(\pi_{\g'/\h'}(y+y'')\int_{\G'}\psi(g.(y+y''))\,dg\right)\Big|_{y''=0} \qquad (|\alpha|\leq\deg(Q\t\pi_{\z'/\h'}))\,.
\]
In particular this range is finite dimensional. Therefore the distribution \eqref{the main limit pro2} is annihilated by an ideal of finite co-dimension in $\C[\g']^{\G'}$. 

Hence, in any case ($l>l'$ or $l\leq l'$), the Fourier transform
\begin{equation}\label{the main limit pro2'}
\left(\tau'_*(\partial(Q)F(y))\right)^\wedge\in\Ss'(\g')
\end{equation}
is annihilated by an ideal of finite co-dimension in $\partial(\C[\g']^{\G'})$. 
Here $\partial(\C[\g']^{\G'})$ is the algebra of $\G'$-invariant constant-coefficient differential operators on $\g'$. 
Now Harish-Chandra Regularity Theorem \cite[Theorem 1, page 11]{HC-65InvariantEigendistributionsLieAlg} implies that the distribution \eqref{the main limit pro2'}
is a locally integrable function whose restriction to the set of the regular semisimple elements has a known structure. Specifically, Harish-Chandra's formula for the radial component of a $\G'$-invariant differential operator with constant coefficients on $\g'$ together with \cite[Lemma 19]{HC-64b} 
shows
that the restriction
\[
\pi_{\g'/\h'}\left(\tau'_*(\partial(Q)F(y))\right)^\wedge\,|_{\reg{\h'}}
\]
is annihilated by an ideal of finite co-dimension in $\partial(\C[\h'])$. Hence, for any connected component $C(\reg{\h'})\subseteq \reg{\h'}$ there is an exponential polynomial $\sum_j p_je^{\lambda_j}$ such that
\begin{equation}\label{the main limit pro2''}
\left(\tau'_*(\partial(Q)F(y))\right)^\wedge\,|_{C(\reg{\h'})}=\frac{1}{\pi_{\g'/\h'}}\sum_j p_je^{\lambda_j}\,;
\end{equation}
see e.g. \cite[Lemma 2, Appendix to 8.3.1]{WarnerII}.
Let
\[
p(x)=\sum_j p_j(x)e^{\lambda_j(x)} \qquad (x\in C(\reg{\h'}))\,.
\]
This function extends analytically beyond the connected component and for any $k=1,2,3,\dots$ we have Taylor's formula, as in \cite{Hormander}, 
\begin{equation}\label{the main limit pro2'''}
p(x)=\sum_{|\alpha|<k}\partial^\alpha p(0)\frac{x^\alpha}{\alpha!} +k\int_0^1(1-s)^{k-1}\sum_{|\alpha|=k}\partial^\alpha p(sx)\,ds\frac{x^\alpha}{\alpha!}\,.
\end{equation}
Since the distribution \eqref{the main limit pro2'} is tempered, the real parts of the $\lambda_j$ are non-positive on $C(\reg{\h'})$. Furthermore, the $\lambda_j$ depend linearly on $y$ and the $p_j$ depend polynomially on $y$. Therefore a straightforward argument shows that there is $N>0$ such that
\begin{equation}\label{the main limit pro2''''}
\left|\partial^\alpha p(tx)\right|\leq \text{constant}\cdot (1+|y|)^{N}(1+|x|)^{N}\sum_{|\alpha|=k}\left|\frac{x^\alpha}{\alpha!}\right|\,.
\end{equation}
%%
%Hence, as a tempered distribution on $\Ss(C(\reg{\h'_1}))$, the error term in \eqref{the main limit pro2'''} is bounded by a seminorm on that space times a distribution %homogeneous of degree $k$. 
Hence \eqref{the main limit pro2'} is a finite sum of homogeneous distributions, of possibly negative degrees, plus the error term which is bounded by \eqref{the main limit pro2''''}.
Thus
there is an integer $a$ such that the following limit exists in $\Ss'(\g')$:
\begin{eqnarray}\label{the main limit pro6}
\underset{t\to 0+}{\lim}\ t^aM_t^*\left(\tau'_*(\partial(Q)F(y))\right)^\wedge\,.
\end{eqnarray}
Moreover,  there is a seminorm $q$ on $\Ss(\g')$ and $N\geq 0$ such that
\begin{multline}\label{the main limit pro1''}
\left|t^aM_t^*\left(\tau'_*(\partial(Q)F(y))\right)\hat{}(\psi)\right|\leq (1+|y|)^N q(\psi)
\qquad (0<t\leq 1,\ y\in \h\cap \tau(\Wv),\ \psi\in \Ss(\g'))\,.
\end{multline}
By taking the inverse Fourier transform we see that 
there is an integer $b$ such that the following limit exists in $\Ss'(\g')$:
\begin{eqnarray}\label{the main limit pro7}
\underset{t\to 0+}{\lim}\ t^b M_{t^{-1}}^*\tau'_*(\partial(Q)F(y))\,.
\end{eqnarray}
Moreover,  there is a seminorm $q$ on $\Ss(\g')$ and $N\geq 0$ such that
\begin{multline}\label{the main limit pro1'''}
\left|\underset{t\to 0+}{\lim}\ t^b M_{t^{-1}}^*\tau'_*(\partial(Q)F(y))(\psi)\right|\leq (1+|y|)^N q(\psi)
 \qquad (0<t\leq 1,\ y\in \h\cap \tau(\Wv),\ \psi\in \Ss(\g'))\,.
\end{multline}
Notice that the following equivalent formulas hold:
\begin{align}\label{dilation relations 3}
(\psi\circ\tau')_t&=t^{2\,\dim\,\g'-\dim\,\Wv}\psi_{t^2}\circ\tau' \qquad (\psi\in \Ss(\g'))\,, 
\nn\\
\tau'_*(M_{t^{-1}}^* u)&=t^{\dim\Wv-2\dim\g'} M_{t^{-2}}^*\tau'_*(u) \qquad (u\in\Ss'(\Wv))\,.
\end{align}
The injectivity of the map $\tau'_*$, see Corollary \cite[(6)]{McKeePasqualePrzebindaWCestimates}, and (\ref{dilation relations 3}) imply that there is an integer $n$ such that the following limit exists in $\Ss'(\Wv)$,
\begin{eqnarray}\label{the main limit pro8}
\underset{t\to 0+}{\lim}\ t^n M_{t^{-1}}^*\partial(Q)F(y)\,.
\end{eqnarray}
Now Lemma \ref{the main localized to U lemma} shows that $n=\deg\mu_{\Oo_m}$ and the proposition follows.
\end{prof}
\section{\textbf{An integral over the slice through a nilpotent element}}
\label{An integral over the slice through a nilpotent element} 
\setcounter{thh}{0}
\subsection{\textbf{Normalization of measures}}
\label{Normalization of measures}
\setcounter{thh}{0}
Recall from section \ref{A slice through nilpotent element} the positive definite symmetric bilinear  form $B(\cdot,\cdot)=-\langle\theta\cdot,\cdot\rangle$ on  $\mathfrak s$. We normalize the Lebesgue measure on $\mathfrak s$ so that the volume of unit cube, defined in terms of $B(\cdot,\cdot)$, is $1$.

Let $\G^0\subseteq \G$ denote the connected component of the identity 
and set $-\G^0=\{-g ; g\in \G^0\}$.
Recall that for our compact group $\G$, the Cayley transform $c(x)=(x+1)(x-1)^{-1}$ maps $\g$ onto $-\G^0$. Notice that $\G=\G^0=-\G^0$ if $\G=\Ug_d$ or $\Sp_d$.
Set $r=\frac{2\dim_\R\g}{\dim_\R\Vv}$. Then, as checked in \cite[(3.11)]{PrzebindaUnipotent}, one may normalize the Haar measure on the group $\G$ so that 
\[
dc(x)=|\det_\R(1-x)|^{-r}\,dx \qquad (x\in \g)\,.
\]
The proof presented in \cite[(3.11)]{PrzebindaUnipotent} is valid for $\G\ne\Og_{2n+1}$. In the case $\G= \Og_{2n+1}$ a parallel argument works too. This is different than the normalization given in \cite[Theorem 1.14]{HelgasonGeomtric}.

Having normalized the measures, we may study the distributions on $\Wv$, $\g$ and $\G$ as ``generalized functions", in the sense that they are derivatives of continuous functions multiplied by the corresponding measures, as in \cite[section 6.3]{Hormander}.

\subsection{\textbf{Some geometry of the moment map}}
\label{Some geometry of the moment map}
Fix an element $N\in \ss1$ such that $w_N\in\Oo_m$, see Lemma \ref{structure of t'-1tau(0)}. Let $\G_N\subseteq \G$ be the stabilizer of $N$ and let $\g_N\subseteq \g$ be the Lie algebra of $\G_N$. Then we have a direct sum decomposition, orthogonal with respect to the form $B(\cdot,\cdot)$ of section \ref{A slice through nilpotent element},
\[
\g=\g_N\oplus \g_N^{\perp_B}\,.
\]
Recall the subspaces 
\[
[\so,N]^{\perp_B}\subseteq\ss1 \qquad\text{and}\qquad \Wv_{[\so,N]^{\perp_B}}\subseteq\Wv
\]
defined in \eqref{decomposition of ss1 with respect to N} and \eqref{the the set U}. 
Let $R_N\subseteq \Wv_{[\so,N]^{\perp_B}}$ denote the radical of the restriction of the symplectic form $\langle\cdot,\cdot\rangle$ to $\Wv_{[\so,N]^{\perp_B}}$, and let $\Wv_N\subseteq \Wv_{[\so,N]^{\perp_B}}$ denote the orthogonal complement of $R_N$ with respect to the form $B(\cdot,\cdot)$. Then 
either $\Wv_N=0$ or the restriction of the symplectic form
 $\langle\cdot,\cdot\rangle$ to $\Wv_N$ is non-degenerate and
\[
\Wv_{[\so,N]^{\perp_B}}=R_N\oplus \Wv_N\,.
\]
In the notation of the proof of Lemma \ref{lemma I.4}, we have
\begin{eqnarray*}
&&R_N=\left\{
\begin{pmatrix}
0 & 0\\
0 & 0\\
w_3 & w_6
\end{pmatrix};\ w_3\in \SHs_m(\Dc)\,,\ \ w_6\in M_{m,d-m}(\Bbb D)
\right\}\\
&&\Wv_N=\left\{
\begin{pmatrix}
0 & 0\\
0 & w_5\\
0 & 0
\end{pmatrix};\ w_5\in M_{d'-2m,d-m}(\Bbb D)
\right\}
\end{eqnarray*}
if $d>m$, and
\begin{eqnarray*}
&&R_N=\left\{
\begin{pmatrix}
0 \\
0 \\
w_3 
\end{pmatrix};\ w_3\in \SHs_m(\Dc)\,,\ \ w_6\in M_{m,d-m}(\Bbb D)
\right\} \\
&&\Wv_N=0
\end{eqnarray*}
if $d=m$. 

Suppose $d>m$. Then $\Wv_N=0$ if and only if $d'=2m$, i.e. $(\cdot,\cdot)'$ is split. 
Hence $\Wv_N\neq 0$ if and only if $\G'=\Ug_{p,q}$ with $p=m<q=m+(d'-2m)$. 

\begin{lemma}\label{geometryoftau}
The map
\begin{equation}\label{geometryoftau1}
R_N\ni v\to \tau(w_N+v)\in \g_N^{\perp_B}
\end{equation}
is an $\R$-linear bijection. The absolute value of the determinant of the matrix of this map defined in terms of any orthonormal basis is equal to 
\begin{equation}\label{geometryoftau2}
2^{\dim_\R\SHs_m(\Dc)+\frac{1}{2}\dim_\R M_{m,d-m}(\Bbb D)}
=2^{\frac{1}{2}\dim_\R\SHs_m(\Dc)}
2^{\frac{1}{2}\dim\g_N^{\perp_B}}\,.
\end{equation}
\end{lemma}
\begin{prof}
 An orthonormal basis of $R_N$ consists of the matrices
\begin{alignat*}{2}
&\frac{1}{\sqrt{2}}\left(E_{p,q}-E_{q,p}\right)\,, \quad &&1\leq p<q\leq m\,,\\
&E_{r,s}\,,  \quad &&m<r,s\leq d\,,
\end{alignat*}
if $\Bbb D=\R$, and of the matrices 
\begin{alignat*}{2}
&\frac{1}{\sqrt{2}}\left(E_{p,q}-E_{q,p}\right)\,,\ \ \ &&1\leq p<q\leq m\,,\\
&\gamma E_{p,p}\,,\ \ \ &&1\leq p\leq m\,,\\
&\frac{\gamma}{\sqrt{2}}\left(E_{p,q}+E_{q,p}\right)\,,\ \ \ &&1\leq p< q\leq m\,,\\
&\gamma' E_{r,s}\,,\ \ \ &&m<r,s\leq d\,,
\end{alignat*}
where $\gamma=i$, $\gamma'=1, i$ if $\Bbb D=\C$, and 
$\gamma=i, j, k$, $\gamma'=1, i, j, k$, if $\Bbb D=\Ha$. 

An orthonormal basis of $\g$ consists of the matrices
\[
\frac{1}{\sqrt{2}}\left(E_{p,q}-E_{q,p}\right)\,,\ \ \ 1\leq p<q\leq d\,,
\]
if $\Bbb D=\R$, and of the matrices
\begin{alignat*}{2}
&\frac{1}{\sqrt{2}}\left(E_{p,q}-E_{q,p}\right)\,,\ \ \ &&1\leq p<q\leq d\,,\\
&\gamma E_{p,p}\,,\ \ \ &&1\leq p\leq m\,,\\
&\frac{\gamma}{\sqrt{2}}\left(E_{p,q}+E_{q,p}\right)\,,\ \ \ &&1\leq p\leq q\leq d\,,
\end{alignat*}
where $\gamma=i$ if $\Bbb D=\C$, and $\gamma=i, j, k$, if $\Bbb D=\Ha$. 

As we have seen in \eqref{I.second eq}, the map \eqref{geometryoftau1} is given by the formula
\[
R_N\ni\left(
\begin{array}{lll}
0 & 0\\
0 & 0\\
w_3 & w_6
\end{array}
\right)\longrightarrow
\left(
\begin{array}{lll}
2w_3 & w_6\\
-\overline w_6^t & 0
\end{array}
\right)\in\g\,.
\]
Since
\[
\g_N=\left\{
\left(
\begin{array}{lll}
0 & 0\\
0 & x_{22}
\end{array}
\right);\ x_{22}=-\overline{x_{22}}^t\in M_{d-m}(\Bbb D)
\right\}
\]
and 
\[
\g_N^{\perp_B}=\left\{
\left(
\begin{array}{lll}
x_{11} & x_{12}\\
-\overline{x_{12}}^t & 0
\end{array}
\right);\ x_{11}=-\overline{x_{11}}^t\in M_{m}(\Bbb D)\,,\ x_{12}\in M_{m,d-m}(\Bbb D)
\right\}
\]
the $\R$-linearity and bijectivity of the map \eqref{geometryoftau1} follows.

Also, this map sends an element of our orthonormal basis contained in the $w_3$ block to $2$ times an element of the orthonormal basis contained in the $x_{11}$ block. Furthermore, it sends an element of our orthonormal basis contained in the $w_6$ block to $\sqrt{2}$ times an element of the orthonormal basis contained in the $\left(
\begin{array}{lll}
0 & x_{12}\\
-\overline{x_{12}}^t & 0
\end{array}
\right)$ block. Hence, \eqref{geometryoftau2} follows.
\end{prof}
\begin{lemma}\label{geometryoftau10}
Let $\tau_N:\Wv_N\to\g_N$ be the unnormalized moment map. Then
\begin{equation}\label{geometryoftau10.1}
\tau(w_N+v+w)=\tau(w_N+v)+\tau_N(w)\qquad (v\in R_N\,,\ w\in\Wv_N)\,,
\end{equation}
where $\tau(w_N+v)\in\g_N^{\perp_B}$. 
If $\Wv_N=0$, 
then the map \eqref{geometryoftau10.1} coincides with the map \eqref{geometryoftau1}.
\end{lemma}
\begin{prof}
This is immediate from the formulas \eqref{the the set U} and \eqref{I.second eq}.
\end{prof}

\subsection{\textbf{The integral as a distribution on $\g$}}
\label{The integral as a distribution on g}
Recall the character $\chi(t)=e^{2\pi i t}$, $t\in\R$, and the imaginary Gaussians
\begin{equation}\label{imaginarygaussian}
\chi_x(w)=\chi\Big(\frac{1}{4}\langle x w, w\rangle\Big)=\chi\Big(\frac{1}{4}\tr_{\Bbb D/\R}(x\tau(w))\Big)
\qquad (x\in\g\,,\ w\in \Wv)\,.
\end{equation}
As usual, by \eqref{w_x}, we can consider $\chi_x$ as a function on $\ss1$ by setting 
$$
\chi_x(y)=\chi_x(w_y) \qquad (y\in \ss1)\,.
$$
Fix an element $\t c(0)\in \wt{\Sp}(\Wv)$ lifting $c(0)=-1$. Since $\g$ is simply connected, there is a unique continuous (in fact real analytic) lift $\t c:\g\to \wt\G$ passing through $\t c(0)$. Then $\t c:\g_N\to \wt\G_N$. Since $\G$ is compact, the Cayley transform $c$ maps $\g$ onto the dense subset of $-\G^0$ consisting of the elements $g$ such that $\det(g-1)\neq 0$. 
The fixed normalization of the measure on $\G$ is so that on $\widetilde{c(\g)} 
\subseteq \widetilde{-\G^0}$
we have 
\[
d\t c(x)=d c(x) \qquad (x\in \g)\,.
\]
\begin{lemma}\label{distributionong}
Recall the slice $U=N+[\so,N]^{\perp_B}$ through $N$, \eqref{definition of U}. As a distribution on $\g$, 
\begin{equation}\label{distributionong1}
\int_{U}\chi_x(u) \,dx\,du=
C\delta_{\g_N^{\perp_B}}(x^{\perp_B}) \frac{\Theta_{\Wv_N}(\t c(0)\t c(x_N))}{\Theta_{\Wv_N}(\t c(0))
\Theta_{\Wv_N}(\t c(x_N))} \, dx_N \qquad (x\in \g)\,,
\end{equation}
where $C=2^{\frac{3}{2}\dim\g_N^{\perp_B}}2^{-\frac{1}{2}\dim_\R\SHs_m(\Dc)}$, $x=x^{\perp_B}+x_N$, $x^{\perp_B}\in \g_N^{\perp_B}$, $x_N\in\g_N$, $\delta_{\g_N^{\perp_B}}$ is Dirac delta at $0$ on ${\g_N^{\perp_B}}$, and $\Theta_{\Wv_N}$ is the character of the Weil representation of $\wt{\Sp}(\Wv_N)$ attached to the same character $\chi$. If $\Wv_N=0$ then 
$\Theta_{\Wv_N}=1$.
\end{lemma}
\begin{prof} 
We see from Lemma \ref{geometryoftau10} that
\[
\int_U \chi_x(u) \,dx\,du=
\int_{R_N}\chi_{x^{\perp_B}}(w_N+v)\,dx^{\perp_B} dv \int_{\Wv_N}\chi_{x_N}(w) \,dx_N\,dw\,.
\]
Lemma \ref{geometryoftau} implies that 
\begin{align*}
\int_{R_N}\chi_{x^{\perp_B}}(w_N+v)\,dx^{\perp_B}\,dv 
&=2^{-\frac{1}{2}\dim\g_N^{\perp_B}} 2^{-\frac{1}{2}\dim_\R\SHs_m(\Dc)}\int_{\g_N^{\perp_B}}\chi\Big(\frac{1}{4}\tr_{\Dc/\R}(yx^{\perp_B})\Big)\,dx^{\perp_B}\,dy \\
&=2^{-\frac{1}{2}\dim\g_N^{\perp_B}}2^{-\frac{1}{2}\dim_\R\SHs_m(\Dc)}\delta_{\g_N^{\perp_B}}\Big(\frac{1}{4}x^{\perp_B}\Big)\\
&=2^{\frac{3}{2}\dim\g_N^{\perp_B}}2^{-\frac{1}{2}\dim_\R\SHs_m(\Dc)}\delta_{\g_N^{\perp_B}}(x^{\perp_B})\,.
\end{align*}
Furthermore, by evaluating both sides of the equation \cite[(139)]{AubertPrzebinda_omega} at $w=0$ we see that
\[
\Theta_{\Wv_N}(\t c(0))
\Theta_{\Wv_N}(\t c(x_N))\int_{\Wv_N}\chi_{x_N}(w) \,dx_N\,dw=\Theta_{\Wv_N}(\t c(0)\t c(x_N))\, dx_N\,.
\]
Here we are using the convention on ``generalized functions'' we introduced in subsection \ref{Normalization of measures}\,. So, with the notation of \cite[(139)]{AubertPrzebinda_omega}, $t(\t c(x))(w)=\chi_x(w)$ and $[t(\t c(0)) \natural t(\t c(x_N))](w)=[1 \natural \chi_{x_N}](w)$, where $\natural$ denotes the twisted convolution on $\Wv_N$. Since $\chi_{x_N}$ is even, we conclude that 
$[t(\t c(0)) \natural t(\t c(x_N))](0)=[1 \natural \chi_{x_N}](0)=\int_{\Wv_N}\chi_{x_N}(w)\,dw$.
\end{prof}
\subsection{\textbf{The integral as a distribution on $\wt{-\G^0}$}}
\label{The integral as a distribution on G}
\ 
As in \cite[(138)]{AubertPrzebinda_omega}, we consider the embedding
\begin{equation}\label{mapT-1}
T:\wt{\Sp}(\Wv)\to \Ss'(\Wv)
\end{equation}
of the metaplectic group into the space of tempered distributions on the symplectic space. In particular,
\begin{equation}\label{mapT}
T(\t c(x))=\Theta(\t c(x))\chi_x(w) \; dw
\qquad (x\in\g\,,\ w\in \Wv)\,,
\end{equation}
where $\Theta$ denotes the character of the Weil representation of $\wt \Sp(\Wv)$ attached to the character $\chi$.

Suppose $\Wv_N\ne 0$.
The structure of our dual pair is such that the metaplectic covering 
\[
\wt{\Sp}(\Wv)\supseteq\wt\G\to\G\subseteq \Sp(\Wv)
\] 
restricts to the metaplectic covering 
\[
\wt{\Sp}(\Wv_N)\supseteq\wt\G_N\to\G_N\subseteq \Sp(\Wv_N)\,.
\] 
Indeed, $\G_N$ consists of the elements of $\G$ of the block diagonal form 
$
\begin{pmatrix}
I_m & 0 \\
0 & g
\end{pmatrix}
$. The dual pair $(\G_N,\G'_N)$ is of the same type as $(\G,\G')$, with 
$\G'_N$ consisting of elements of the form 
$
\begin{pmatrix}
I_m & 0 & 0\\
0 & g' & 0 \\
0 & 0 & I_m
\end{pmatrix}
$.
The dimension of the defining space $\Vv'_N$ of $\G'_N$ is $d'-2m$, which has the same parity as $d'$. The claim therefore follows from \cite[Appendix D]{McKeePasqualePrzebindaWCSymmetryBreaking}.

In particular we have an inclusion $\iota: \wt\G_N\to \wt\G$
and hence the pull-back of test functions  $\iota^*: C_c^\infty(\wt\G)\to C_c^\infty(\wt\G_N)$ and push-forward of distributions $\iota_*: \mathcal{D}'(\wt\G_N) \to \mathcal{D}'(\wt\G)$. 
By restriction, we get
\begin{equation}\label{pushforwardofdistr}
\iota_*:\mathcal{D}'(\wt{-\G_N^0})
\to \mathcal{D}'(\wt{-\G^0})\,.
\end{equation}

If $\Wv_N= 0$ and $d>m$ (and hence the form $(\cdot,\cdot)'$ is split), then we still have \eqref{pushforwardofdistr}, where the coverings are in $\wt{\Sp}(\Wv)$. It follows from \cite[Proposition 4.28]{AubertPrzebinda_omega} that in the above two cases, 
the formula
\begin{equation}\label{distributiononG1}
\chi_+(\t g)=\frac{\Theta(\t g)}{|\Theta(\t g)|} \qquad (\t g\in\wt\G)
\end{equation}
defines a group homomorphism $\chi_+:\wt\G\to \C^\times$, because there is a complete polarization of $\Wv$ preserved by $\G$. Indeed, such a polarization is $\Wv=\Xv\oplus \Yv$, where $\Xv$ and $\Yv$ are the spaces of the first $m$ rows and of the last $m$ rows of $\Wv$, respectively.  
In particular, $\chi_+$ restricts to a character of $\wt\G_N$.
Notice that $\chi_+$ is a character of $\wt\G$ whenever there is a polarization is $\Wv=\Xv\oplus \Yv$ such that
$\G$ preserves $\Xv$ and $\Yv$ to fit into \cite[Proposition 4.28]{AubertPrzebinda_omega}. This is always the case when the form $(\cdot,\cdot)'$ is split. 

If $\Wv_N= 0$ and $d=m$, then $\G_N=1$. In this case we artificially enlarge $\G_N$ to be the center $\Zg=\{1, -1\}$ of the symplectic group $\Sp(\Wv)$. Then $\wt \G_N=\wt \Zg$ and, as checked in \cite[(22)]{McKeePasqualePrzebindaWCSymmetryBreaking} the formula
\begin{equation}\label{distributiononG1.0}
\chi_+(\t g)=\frac{\Theta(\t g)}{|\Theta(\t g)|} \qquad (\t g\in\wt\G_N)
\end{equation}
defines a group homomorphism $\chi_+:\wt\G_N\to \C^\times$.
\begin{lemma}\label{distributiononG}
Suppose $\Wv_N=0$. 
Then,
as a distribution on $\wt{-\G^0}$, 
\begin{equation}\label{distributiononG-1}
\int_{U}T(\t g)(u)\,d\t g\,du=
C 2^{-\frac{1}{2}\dim \Wv}\iota_*(\chi_+(\t g_N)\,d\t g_N) \,,
\end{equation}
where $C=2^{\frac{3}{2}\dim\g_N^{\perp_B}}2^{-\frac{1}{2}\dim_\R\SHs_m(\Dc)}$  and $d\t g_N$ is the Haar measure on $\wt{-\G_N^0}$.
\end{lemma}
\begin{prof} 
We compute using  Lemma \ref{distributionong},
\begin{align*}
\int_{U}T(\t c(x))(u)\, d\t c(x)\,du
&=\Theta(\t c(x))\int_U \chi_x(u) |\det(1-x)|^{-r}\,dx\,du\\
&=C \Theta(\t c(x_N)) \delta_{\g_N^{\perp_B}}(x^{\perp_B}) |\det(1-x_N)|^{-r}\,dx_N\\
&=C \delta_{\g_N^{\perp_B}}(x^{\perp_B})\chi_+(\t c(x_N))|\Theta(\t c(x_N))| |\det(1-x_N)|^{-r}\,dx_N\\
&=C 2^{-\frac{1}{2}\dim\Wv}\delta_{\g_N^{\perp_B}}(x^{\perp_B})\chi_+(\t c(x_N))|\det(1-x_N)|^{\frac{d'}{2}-r}\,dx_N\\
&=C 2^{-\frac{1}{2}\dim\Wv} \delta_{\g_N^{\perp_B}}(x^{\perp_B})\chi_+(\t c(x_N))\,d\t c(x_N)\,,
\end{align*}
because (a straightforward computation shows that) $\frac{d'}{2}-r=\frac{2\dim_\R\g_N}{\dim_\R \Vv_N}$, where $\Vv_N\subseteq \Vv$ is the defining module for $\G_N$.
\end{prof}
\begin{lemma}\label{distributiononGWN}
Suppose 
$\Wv_N\ne 0$.
 (Equivalently, $d>m$ and the form $(\cdot,\cdot)'$ is not split.) 
 Then,
as a distribution on $\wt{-\G^0}$, 
\begin{equation}\label{distributiononGWN1}
\int_{U}T(\t g)(u)\,d\t g\,du=
C 2 ^{\frac{1}{2}\dim \Wv-\dim\Wv_N}\chi_+(\t c(0))^{-1}\iota_*(\Theta_{\Wv_N}(\t c(0)\t g_N)\,d\t g_N) \,,
\end{equation}
where $C=2^{\frac{3}{2}\dim\g_N^{\perp_B}}2^{-\frac{1}{2}\dim_\R\SHs_m(\Dc)}$ and $d\t g_N$ is the Haar measure on $\wt{-\G_N^0}$.
\end{lemma}
\begin{prof} 
We compute using Lemma \ref{distributionong},
\begin{multline*}
\int_{U}T(\t c(x))(u)\, d\t c(x)\,du= C\delta_{\g_N^{\perp_B}}(x^{\perp_B})\Theta(\t c(x)) 
\frac{\Theta_{\Wv_N}(\t c(0)\t c(x_N))}{\Theta_{\Wv_N}(\t c(0))
\Theta_{\Wv_N}(\t c(x_N))}|\det(1-x_N)|^{-r}\,dx\\
=C\delta_{\g_N^{\perp_B}}(x^{\perp_B})\frac{1}{\Theta(\t c(0))}
\frac{\Theta(\t c(0))}{\Theta_{\Wv_N}(\t c(0))} \frac{\Theta(\t c(x_N))}{\Theta_{\Wv_N}(\t c(x_N))} 
\Theta_{\Wv_N}(\t c(0)\t c(x_N))|\det(1-x_N)|^{-r}\,dx\,.
\end{multline*}
Notice that
\begin{equation}
\label{eq:finite valued}
\g_N\ni x_N\to\frac{\Theta(\t c(x_N))}{\Theta_{\Wv_N}(\t c(x_N))}\left|\frac{\Theta_{\Wv_N}(\t c(x_N))}{\Theta(\t c(x_N))}\right|
=\frac{\Theta(\t c(x_N))}{|\Theta(\t c(x_N))|}\frac{|\Theta_{\Wv_N}(\t c(x_N))|}{\Theta_{\Wv_N}(\t c(x_N))}
\in\C^\times
\end{equation}
is a continuous function taking values in a finite set. 
The latter property is a consequence of \cite[Proposition 4.28]{AubertPrzebinda_omega}:  
$\G$ may be considered as a subgroup of $\GL(\Xv)$, where $\Xv\oplus \Yv$ is the polarization of $\Wv$. 
Then $\chi_+(\t g)$ is written in terms of $\det_\Xv^{-1/2}(\t g)$, which can assume a finite set of values because the image of $\det_\Xv|_{\wt\G}$ is a compact subgroup of $\R^\times$. 
Hence \eqref{eq:finite valued} is constant, equal to its value at $0$, which is $1$. So
\[
\frac{\Theta(\t c(x_N))}{\Theta_{\Wv_N}(\t c(x_N))}=
\left|\frac{\Theta(\t c(x_N))}{\Theta_{\Wv_N}(\t c(x_N))}\right|\,.
\]
Therefore
\[
\frac{\Theta(\t c(0))}{\Theta_{\Wv_N}(\t c(0))}\frac{\Theta(\t c(x_N))}{\Theta_{\Wv_N}(\t c(x_N))}=
\left|\frac{\Theta(\t c(0))}{\Theta_{\Wv_N}(\t c(0))}\frac{\Theta(\t c(x_N))}{\Theta_{\Wv_N}(\t c(x_N))}\right|\,.
\]
The only dual pair that satisfies the assumptions of this Lemma is $(\G,\G')=(\Ug_d, \Ug_{m, m+(d'-2m)})$ with $d'-2m>0$. In terms of matrices, as in the proof of Lemma \ref{lemma I.4}, 
\[
\G_N=\Ug_{d-m}\,,\ \ \ \Wv_N=M_{d'-2m, d-m}\,.
\]
Hence,
\[
r=\frac{2\dim_\R\g}{\dim_\R \C^{d}}=d\,,\ \ \ r_N=\frac{2\dim_\R\g_N}{\dim_\R \C^{d-m}}=d-m
\]
and therefore
\[
\frac{d'}{2}-\frac{d'-2m}{2}-r=-r_N\,.
\]
Thus
\[
\left|\frac{\Theta(\t c(x_N))}{\Theta_{\Wv_N}(\t c(x_N))}\right||\det(1-x_N)|^{-r}=
\left|\frac{\Theta(\t c(0))}{\Theta_{\Wv_N}(\t c(0))}\right||\det(1-x_N)|^{-r_N}\,.
\]
Therefore
\begin{align*}
\int_{U} &T(\t c(x))(u)\, d\t c(x)\,du\\
&= C\delta_{\g_N^{\perp_B}}(x^{\perp_B}) 
\frac{1}{\chi_+(\t c(0))}\frac{1}{|\Theta(\t c(0))|}
\left|\frac{\Theta(\t c(0))}{\Theta_{\Wv_N}(\t c(0))}\frac{\Theta(\t c(x_N))}{\Theta_{\Wv_N}(\t c(x_N))}\right| \\
& \qquad \times\Theta_{\Wv_N}(\t c(0)\t c(x_N))|\det(1-x_N)|^{-r}\,dx_N\\
&=C\left|\frac{\Theta(\t c(0))}{\Theta_{\Wv_N}(\t c(0))^2}\right|
\delta_{\g_N^{\perp_B}}(x^{\perp_B}) 
\frac{1}{\chi_+(\t c(0))}
\Theta_{\Wv_N}(\t c(0)\t c(x_N))|\det(1-x_N)|^{-r_N}\,dx_N
\end{align*}
and the formula follows.
\end{prof}
\section{\textbf{Proof of the main theorem}}
\label{A limit of an intertwining distribution}
\setcounter{thh}{0}
Here we verify Theorem \ref{the dilation limit of intertwining distribution, intro}. We begin with an intermediate statement. Recall the connected identity component  $\G^0\subseteq \G$. Retain the notation of the previous subsection.
\begin{thm}\label{the dilation limit of intertwining distribution, intro, intermediate 1}
Let $\Pi$ be an irreducible representation of $\wt\G$ that occurs in the restriction of the Weil representation to $\wt\G$.
Then, in the topology of $\Ss'(\Wv)$,
\begin{equation}\label{the dilation limit of intertwining distribution, intro, intermediate 1.1}
t^{\deg \mu_{\Oo_m}}M_{t^{-1}}^* T(\check\Theta_{\Pi}|_{\wt{-\G^0}})\underset{t\to 0+}{\longrightarrow }K\mu_{\Oo_m}\,,
\end{equation}
where $K\ne 0$. 

Suppose $d=m$ or $d>m$ and $(\cdot,\cdot)'$ is split. (Equivalently, suppose $(\G,\G')$ is different from 
$(\Ug_d,\Ug_{m,d'-m})$ with $d'-2m>0$.) Then 
\begin{equation}\label{the dilation limit of intertwining distribution2}
K=C 2 ^{\frac{1}{2}\dim \Wv} \chi_{\Pi\otimes\chi_+^{-1}}(-1)\int_{\G_N^0}\Theta_{\Pi\otimes\chi_+^{-1}}( g_N)\,d g_N
\,,
\end{equation}
where $C$ is as in Lemma \ref{distributiononG} and $\chi_+$ is the character defined in \eqref{distributiononG1}. 
The integral in \eqref{the dilation limit of intertwining distribution2} is equal to
the multiplicity of  the trivial representation of $\G_N^0$ in the restriction of $\Pi\otimes\chi_+^{-1}$ to $\G_N^0$. 

Suppose $(\G,\G')=(\Ug_d,\Ug_{m,d'-m})$ with $d'-2m>0$. Then 
\begin{equation}\label{the dilation limit of intertwining distribution-Un case}
K=C 2^{\frac{1}{2}\dim \Wv-\dim \Wv_N} \chi_{\Pi}(\tilde{c}(0))
\int_{\G_N} \Theta_{\Pi}(\wt{g_N}^{-1})\Theta_{\Wv_N}(\wt{g_N})\,d g_N
\,,
\end{equation}
where $C$ is as above, $\G_N=\Ug_{d-m}$, and $\Theta_{\Wv_N}$ is the character of the Weil representation of $\wt{\Sp}(\Wv_N)$. The integral in \eqref{the dilation limit of intertwining distribution-Un case} is equal to the sum of multiplicities of the irreducible component of 
$\Pi|_{\wt{\G_N}}$ in the restriction of $\omega_N$ to $\wt{\G_N}$.
\end{thm}
Notice that if $\G=\Ug_d$ or $\Sp_d$ then $-\G^0=\G^0=\G$. Hence, in these cases, Theorem \ref{the dilation limit of intertwining distribution, intro, intermediate 1} is equivalent to Theorem \ref{the dilation limit of intertwining distribution, intro}.

\begin{prof}
We first prove that the limit in \eqref{the dilation limit of intertwining distribution, intro, intermediate 1.1} exists and is a constant multiple of $\mu_{\Oo_m}$. For this, we use the expression of $T(\check\Theta_{\Pi}|_{\widetilde{-\G^0}})$ in terms of Harish-Chandra's almost elliptic orbital integrals $F(y)\in \mathcal{S}'(\Wv)^\Sg$ determined in \cite{McKeePasqualePrzebindaWCSymmetryBreaking}. We need some additional notation. 
If $l\leq l'$, let $(J_1,\dots,J_l)$ be the basis of $\h$ introduced in \cite[(42)] {McKeePasqualePrzebindaWCSymmetryBreaking}. If $l>l'$, extend $\h$ to the Cartan subalgebra $\h(\g)$ of $\g$, with basis $(J_1,\dots,J_l)$ defined as in \cite[(45)] {McKeePasqualePrzebindaWCSymmetryBreaking}. Then $(J_1,\dots,J_{l'})$ is a basis of $\h$. We denote by $(y_1,\dots,y_l)$ (respectively, $(y_1,\dots,y_{l'})$) the coordinates of $y\in \h$ with respect to these bases. 
Let $(J^*_1,\dots,J^*_l)$ be the dual basis of $\h^*$ if $l\leq l'$  (respectively, of $\h(\g)^*$ if $l> l'$), and set $e_j=-iJ^*_j$ for $1\leq j\leq l$. The Harish-Chandra parameter 
$\mu=\sum_{j=1}^l \mu_j e_j$ of $\Pi$ is strictly dominant. In this paper, this means that $\mu_1>\mu_2>\dots>\mu_l$. 

For $1\leq j\leq l$ set
$$
a_j=-\mu_j+\delta-1 \quad \text{and} \quad b_j=-\mu_j+\delta-1\,,
$$ 
where
$$
\delta-1=\begin{cases}
l'-l &\text{if $\G=\Og_{2l}$}\\
l'-l-\frac{1}{2} &\text{if $\G=\Og_{2l+1}$}\\
\frac{l'-l-1}{2} &\text{if $\G=\Ug_{l}$}\\
l'-l-1 &\text{if $\G=\Sp_{2l}$.}
\end{cases}
$$ 
Furthermore, set $\beta=4\pi$ if $\G=\Sp_l$ and $\beta=2\pi$ otherwise. 

Suppose first that $l\leq l'$. Then, according to \cite[Theorem 2] {McKeePasqualePrzebindaWCSymmetryBreaking}, 
\begin{equation}\label{intertwining distribution l<=l'}
T(\check\Theta_{\Pi}|_{\widetilde{-\G^0}})(\phi)
=C\int_{\h\cap\tau(\Wv)}
\left(\prod_{j=1}^l  \left(p_j(y_j) +q_j(\partial_{y_j})\delta_0(y_j)\right)\right)\cdot F(y)(\phi)\, dy
\qquad (\phi\in \mathcal{S}(\Wv))\,,
\end{equation}
where 
$$
p_j(y)=P_{a_j,b_j}(-\beta y)e^{-\beta|y|}
\quad \text{and} \quad 
q_j(y)=\beta^{-1}Q_{a_j,b_j}(\beta^{-1} y)\,,
$$
and $P_{a_j,b_j}$ and $Q_{a_j,b_j}$ are polynomial functions on $(-\infty,0]$ and on $[0,+\infty)$. The explicit expression of 
$P_{a_j,b_j}$ and $Q_{a_j,b_j}$ does not play any role here, but one needs to notice that $P_{a_j,b_j}=0$ 
if $a_j\leq 0$ and $b_j\leq 0$ (i.e. if $|\mu_j|\leq \delta-1$), and in this case $Q_{a_j,b_j}\neq 0$.

The domain of integration $\h\cap\tau(\Wv)$ is described in \cite[Lemma 3.5]{McKeePasqualePrzebindaWCestimates}. It agrees with $\h$ unless $\G=\Ug_l$. If $\G=\Ug_l$, then
$\h\cap\tau(\Wv)$ is a union of closed orthants associated with the fixed basis $(J_1,\dots,J_l)$ of 
$\h$. In all cases, the right-hand-side of \eqref{intertwining distribution l<=l'} is the constant $C$ times a finite sum of integrals of the form 
\begin{equation}
\label{interwining distribution limit}
\int_{Y_I} \prod_{j\in I} p_j(y_j) \Big.\Big( \prod_{j\in I^c} q_j(\partial_{y_j}) F(y)(\phi)\Big)\Big|_
{y_{I^c}=0}\, dy_I\,,
\end{equation}
where $I=\{j_1,\dots,j_I\}$ is a (possibly empty) subset of $\{j\in \{1,2,\dots,l\}:p_j\neq 0\}$, $I^c=\{1,2,\dots,l\}\setminus I$, the integration domain is $Y_I=\prod_{j\in I} Y_j$ where $Y_j$ can be 
$(-\infty,0]$, $[0,+\infty)$ or $\R$, and $dy_I=dy_{j_1} \cdots dy_{j_l}$.

Proposition \ref{the main limit pro}, the exponential decay of the $p_j$'s in \eqref{interwining distribution limit} and the Lebesgue Dominated Convergence Theorem imply that
\begin{multline}\label{the dilation limit of intertwining distribution'}
\underset{t\to 0+}{\lim}\ t^{\deg \mu_{\Oo_m}}M_{t^{-1}}^* T(\check\Theta_{\Pi}|_{\widetilde{-\G^0}})\\
=\left(C\int_{\h\cap\tau(\Wv)}
\left(\prod_{j=1}^l  \left(p_j(y_j) +q_j(\partial_{y_j})\delta_0(y_j)\right)\right)\cdot F(y)|_U(\Bbb I_U)\, dy \right)\mu_{\Oo_m}\,.
\end{multline}

Suppose now that $l>l'$. According to \cite[Theorem 3] {McKeePasqualePrzebindaWCSymmetryBreaking}, 
\begin{equation}\label{intertwining distribution l>l'}
T(\check\Theta_{\Pi}|_{\widetilde{-\G^0}})(\phi)
=C\int_{\tau'(\reg{\hs1})}\Big(\prod_{j\in I_0}  p_j((s_0^{-1}y)_j)\Big)\cdot F(y)(\phi)\,dy\,
\qquad (\phi\in \mathcal{S}(\Wv))\,,
\end{equation}
where $s_0$ is a suitable element of $W(\G,\h(\g))$ 
and 
\begin{equation}
\label{J0}
I_0=\begin{cases} 
\{1,\dots, q\}\cup\{l-p+1,\dots,l\} &\text{if $\G=\Ug_l$}\\
\{1,\dots, l'\} &\text{otherwise}\,.
\end{cases}
\end{equation}
With respect to the fixed basis $(J_1,\dots,J_{l'})$ of $\h$, the integration domain $\tau'(\reg{\hs1})$
is a dense subset of the positive orthant. As in the case $l\geq l'$, Proposition \ref{the main limit pro}, the exponential decay of the $p_j$'s and the Lebesgue Dominated Convergence Theorem imply that
\begin{eqnarray}\label{the dilation limit of intertwining distribution''}
\underset{t\to 0+}{\lim}\ t^{\deg \mu_{\Oo_m}}M_{t^{-1}}^* T(\check\Theta_{\Pi}|_{\widetilde{c(\g)}})
=\left(C\int_{\tau'(\reg{\hs1})}\Big(\prod_{j\in J_0}  p_j((s_0^{-1}y)_j)\Big) F(y)|_U(\Bbb I_U) \,dy \right)\mu_{\Oo_m}\,.
\end{eqnarray}
Thus, in each case, the limit is a constant multiple of the measure $\mu_{\Oo_m}$. 
This constant is the term in parenthesis in \eqref{the dilation limit of intertwining distribution'} or in \eqref{the dilation limit of intertwining distribution''}. It is equal to 
\begin{equation}\label{the dilation limit of intertwining distribution1}
T(\check\Theta_{\Pi}|_{\wt{-\G^0}})|_U(\Bbb I_U)
=\int_U\int_{-\G^0}\Theta_{\Pi}(\t g^{-1})T(\t g)(u)\,dg\,du
\,.
\end{equation}
We need to prove that it is non-zero. 
%%

%%%
Suppose $d= m$ (stable range) or $d>m$ and the form $(\cdot,\cdot)'$ is split.
%Suppose the pair is in the stable range, with $\G$-the smaller member. 
Then Lemma \ref{distributiononG} implies that \eqref{the dilation limit of intertwining distribution1} is equal to
\[
C 2 ^{\frac{1}{2}\dim \Wv} \int_{-\G_N^0}\Theta_{\Pi}(\t g_N^{-1})\chi_+(\t g_N)\,dg_N\,.
\]
Furthermore,
\begin{align*}
\int_{-\G_N^0}\Theta_{\Pi}(\t g_N^{-1})\chi_+(\t g_N)\,d g_N
=\int_{-\G_N^0}\Theta_{\Pi\otimes\chi_+^{-1}}(\t g_N^{-1})\,d g_N
=\int_{-\G_N^0}\Theta_{\Pi\otimes\chi_+^{-1}}( g_N)\,d g_N
\,,
\end{align*}
where in the last formula $\Pi\otimes\chi_+^{-1}$ is viewed as a representation of $\G_N$. Thus
\begin{equation}\label{the dilation limit of intertwining distribution1.-1}
T(\check\Theta_{\Pi}|_{\widetilde{-\G^0}})|_U(\Bbb I_U)
=C 2 ^{\frac{1}{2}\dim \Wv}  \int_{-\G_N^0}\Theta_{\Pi\otimes\chi_+^{-1}}( g_N)\,d g_N\,.
\end{equation}
Since $-1$ is in the center of $\Sp(\Wv)$, it acts via multiplication by a scalar $\chi_{\Pi\otimes\chi_+^{-1}}(-1)$ on $\Pi\otimes\chi_+^{-1}$. Therefore
\[
\int_{-\G_N^0}\Theta_{\Pi\otimes\chi_+^{-1}}(g_N)\,d g_N=
\chi_{\Pi\otimes\chi_+^{-1}}(-1)\int_{\G_N^0}\Theta_{\Pi\otimes\chi_+^{-1}}( g_N)\,d g_N\,.
\]
Hence, \eqref{the dilation limit of intertwining distribution2} follows.

The integral in \eqref{the dilation limit of intertwining distribution2} is the multiplicity of the trivial representation of $\G_N^0$ in the restriction of $\Pi\otimes\chi_+^{-1}$ to $\G_N^0$. 
If $\G_N=\{1\}$, i.e. $d=m$, then this multiplicity is equal to the degree of $\Pi$. Otherwise, there are three cases:
\begin{alignat*}{3}
&\G=\Og_d\,,\qquad &&\G'=\Sp_{2m}(\R),\qquad &&\G_N=\Og_{d-m}\,,\\
&\G=\Sp_d\,,\; &&\G'=\Og^*_{2m}(\R),\; &&\G_N=\Sp_{d-m}\,,\\
&\G=\Ug_d\,,\; &&\G'=\Ug_{m,m},\; && \G_N=\Ug_{d-m}\,.
\end{alignat*}
Suppose that $\G=\Og_d$ or $\Sp_d$. Since $\Pi$ occurs in Howe's correspondence, by 
\cite[(A.4.2.1) and (A.6.2)]{PrzebindaInfinitesimal}, the highest weight of $\Pi\otimes\chi_+^{-1}$
is $\lambda=(\lambda_1,\dots, \lambda_m, 0,\dots, 0)$, where the last $d-m$ entries are equal to $0$ and $\lambda_1\geq \dots \geq \lambda_m\geq 0$ are integers.  We then recognize that 
the trivial representation of $\G_N^0$ occurs in the restriction of $\Pi\otimes\chi_+^{-1}$ to $\G_N^0$ by iterating the branching laws $\SO_n \downarrow \SO_{n-1}$ or $\Sp_n \downarrow \Sp_{n-1}$, see e.g. \cite[Theorems 9.16 and 9.18]{KnappLie-2ndEd}.
If $\G=\Ug_d$, then by 
\cite[(A.5.2)]{PrzebindaInfinitesimal}, the highest weight of $\Pi\otimes\chi_+^{-1}$
is $\lambda=(\mu_1,\dots, \mu_s, 0,\dots, 0, -\nu_r, \dots, -\nu_1)$, where 
$0\leq s \leq m$, $0\leq r \leq m$, $r+s\leq d$, and  
$\mu_1\geq \dots \geq \mu_s>0$ and $\nu_1\geq \dots \geq \mu_r>0$ are integers.
Notice that there are $d-(r+s)$ zero entries in the central part of $\lambda$. The 
highest weights of the irreducible representations occurring in the branching $\Ug_n\downarrow \Ug_{n-1}$ interleave $\lambda$, see e.g. \cite[Theorems 9.14]{KnappLie-2ndEd}.
Iterating these branching laws $m$ times therefore allows the highest weight of all zero entries. Hence the trivial representation of $\G_N^0$ occurs in the restriction of $\Pi\otimes\chi_+^{-1}$ to $\G_N^0$ in this case too.

Let us now consider the remaining cases, i.e. when $\Wv_N\ne 0$. Lemma \ref{distributiononGWN} implies that \eqref{the dilation limit of intertwining distribution1} is equal to
\[
C 2 ^{\frac{1}{2}\dim \Wv-\dim\Wv_N}\int_{-\G_N^0}\Theta_{\Pi}(\t g_N^{-1})
\Theta_{\Wv_N}(\t c(0)\t g_N)\,d g_N\,.
\]
Notice that
\begin{align*}
\int_{-\G_N^0}\Theta_{\Pi}(\t g_N^{-1})
\Theta_{\Wv_N}(\t c(0)\t g_N)\,dg_N
&=\int_{\G_N^0}\Theta_{\Pi}(\t c(0)\t g_N^{-1})
\Theta_{\Wv_N}(\t g_N)\,dg_N\\
&=\chi_\Pi(\t c(0))\int_{\G_N^0}\Theta_{\Pi}(\t g_N^{-1})\Theta_{\Wv_N}(\t g_N)\,dg_N
\,,
\end{align*}
where $\chi_\Pi$ is the central character of $\Pi$.

Notice that $\G_N$ is isomorphic to $\Ug_{d-m}$. 
Hence $\G_N^0=\G_N$ and the centralizer of $\G_N$ in $\Sp(\Wv_N)$ is compact, isomorphic to $\Ug_{d'-2m}$. Thus we have the dual pair $(\Ug_{d-m}, \Ug_{d'-2m})$ inside $\wt\Sp(\Wv_N)$.
The restriction $\Pi|_{\wt{\G_{N}}}$ decomposes into a finite sum of irreducibles and the integral
\begin{equation}
\label{non-zero-constant-WN non zero}
\int_{\wt{\G_N}}\Theta_{\Pi}(\t g_N^{-1})\Theta_{\Wv_N}(\t g_N)\,d\t g_N
\end{equation}
is the sum of the multiplicities of those irreducibles that occur in the restriction of $\omega_N$ to $\wt{\G_{N}}$. 
Again, looking at the highest weight $\lambda$ of $\Pi$, 
\cite[(A.5.2)]{PrzebindaInfinitesimal}, and the branching rules
$\Ug_n\downarrow \Ug_{n-1}$, e.g. \cite[Theorems 9.14]{KnappLie-2ndEd}, we see that the irreducible representation of $\wt{\Ug_{d-m}}$ whose highest weight has the central $d-m$ components of $\lambda$ is a representation of $\wt{\Ug_{d-m}}$ occurring in both the restriction of $\Pi$ and the restriction of $\omega_N$ to $\wt{\Ug_{d-m}}$. Thus the number \eqref{non-zero-constant-WN non zero} is not zero.
\end{prof}

Now we consider the dual pairs $(\G, \G')$ for which $-\G^0\ne \G$. They are isomorphic to $(\Og_d, \Sp_{2l'}(\R))$. More precisely, $\G\setminus (-\G^0)=\G^0$ if $\G=\Og_{2l+1}$, and $\G\setminus (-\G^0)=\G\setminus \G^0$ if $\G=\Og_{2l}$. Here $\G\setminus (-\G^0)$ is the complement of $-\G^0$ in $\G$.
We need to know how to compute 
\begin{equation}\label{oddExceptionalCase}
\underset{t\to 0+}{\lim}t^{\deg \mu_{\Oo_m}}M_{t^{-1}}^* T(\check\Theta_{\Pi}|_{\wt{\G\setminus(-\G^0)}})\,.
\end{equation}

Suppose $d=1$. Then $\G=\Og_1$ and $\G^0=\{1\}$. Hence  $T(\check\Theta_{\Pi}|_{\wt{\G^0}})=T(1)=\delta$.
Also, $\Oo_m=\Wv\setminus\{0\}$ and $\mu_{\Oo_m}$ is the Lebesgue measure. Hence $\deg \mu_{\Oo_m}=0$ and we see that \eqref{oddExceptionalCase} is equal to
\begin{equation}\label{oddExceptionalCasen=1}
\underset{t \to 0+}{\lim}t^{\deg \mu_{\Oo_m}}M_{t^{-1}}^* T(\check\Theta_{\Pi}|_{\wt{\G\setminus(-\G^0)}})=
\underset{t\to 0+}{\lim} M_{t^{-1}}^*\delta=\underset{t\to 0+}{\lim} t^{\dim\Wv}\delta=0\,.
\end{equation}

Assume from now on that $d>1$. As shown in \cite[section 4]{McKeePasqualePrzebindaWCSymmetryBreaking}, there is a symplectic subspace $\Wv_s\subseteq \Wv$ such that
the restriction of the dual pair $(\G, \G')$ to $\Wv_s$ is isomorphic to $(\Og_{d-1}, \Sp_{2l'}(\R))$ and the following statements hold, where $T_s$ is the map \eqref{mapT-1} for the dual pair $(\G_s, \G')$.
\begin{thm}
Let $(\G, \G')=(\Og_{2l+1}, \Sp_{2l'}(\R))$ with $l\geq 1$. 
Then for $\phi\in\Ss(\Wv)$
\begin{equation}\label{main theorem for l<l', special odd 1}
\int_{\G^0}\check\Theta_\Pi(\t g) T(\t g)(\phi)\,dg=\int_{\G_s^0}\check\Theta_\Pi(\t g)\det(1-g) T_s(\t g)(\phi^\G|_{\Wv_s})\,dg\,,
\end{equation}
where
\begin{equation}
\label{phiG}
\phi^\G(w)=\int_\G \phi(g.w) \, dw \qquad (w\in \Wv)\,.
\end{equation}
\end{thm}

\begin{thm}\label{theorem:main theorem for l<l', special odd 1}
Let $(\G, \G')=(\Og_{2l}, \Sp_{2l'}(\R))$ and assume that the character $\Theta_\Pi$ is not supported on 
$\wt{\G^0}$. 
Suppose that $1\leq l\leq l'$ and the pair $(\Og_2, \Sp_2(\R))$ is excluded. 
Then for all $\phi\in\Ss(\Wv)$
\begin{equation}\label{main theorem for l<l' a, special odd 1}
\int_{\G\setminus\G^0}\check\Theta_\Pi(\t g) T(\t g)(\phi)\,dg
=C(\Pi)
\int_{-\G_s^0}\check\Phi_{\Pi}(\t g) 
T_s(\t g)(\phi^\G|_{\Wv_s})\,dg\,,
\end{equation}
where $C(\Pi)$ is a constant equal to $\pm 1$, the function $\check\Phi_{\Pi}(\wt{-g})$ is a finite linear combination of irreducible characters of $\wt{\G_s^0}$, 
and $T_s$ is the map $T$, see \eqref{T}, corresponding to the symplectic space $\Wv_s$.
\smallskip

If $l>l'$, then $\int_{\G \setminus \G^0}\check\Theta_\Pi(\t g) T(\t g)\,dg=\int_{\G^0}\check\Theta_\Pi(\t g) T(\t g)\,dg$.

If $(\G,\G')=(\Og_2,\Sp_2(\R))$, then $\Theta_\Pi$ is not supported in $\wt{\G^0}=\wt{\SO}_2$ if and only if $\Pi=\nu^{-1}$ where $\nu(g,\xi)=\det(g)^{1/2}$ for $(g,\xi)\in \wt{\Og}_2$. In this case, 
$\int_{\G \setminus \G^0}\check\Theta_\Pi(\t g) T(\t g)\,dg=\int_{\G^0}\check\Theta_\Pi(\t g) T(\t g)\,dg$.
\end{thm}
%%
%%
%\begin{thm}\label{main theorem for l<l', special}
%Let $(\G, \G')=(\Og_{2l}, \Sp_{2l'}(\R))$. Assume that the character $\Theta_\Pi$ is not supported on 
%$\wt{\G^0}$. Then $l\leq l'$, the pair $(\Og_2, \Sp_2(\R))$ is excluded, and for $\phi\in\Ss(\Wv)$
%%%
%\begin{equation}\label{main theorem for l<l' a, special}
%\int_{\G\setminus \G^0}\check\Theta_\Pi(\t g) T(\t g)(\phi)\,dg\\
%=C(\Pi)\left(\frac{i}{2}\right)^{l'}
%\int_{\G_s^0}\check\Theta_{\Pi_s}(\t g) T_s(\t g)(\phi^\G|_{\Wv_s})\,dg\,,
%\end{equation}
%%%
%where $C(\Pi)=\pm 1$, $\Pi_s$ is the genuine irreducible representation of $\wt{\G_s^0}$ that occurs in the restriction of $\Pi$ to $\wt{\G_s^0}$ and has highest weight equal to the restriction of highest weight of $\Pi$ to the intersection of the Cartan subalgebra with $\g_s$.
%%%

%%
The following lemma will allow us to reduce the integral on the right hand-side of 
\eqref{main theorem for l<l', special odd 1} to a linear combination of integrals as on the right hand-side of 
\eqref{main theorem for l<l' a, special odd 1}. 
\begin{lemma}\label{FiniteCombinationIrredChar}
Suppose $(\G,\G')=(\Og_{2l+1},\Sp_{2l'}(\R))$. 
The function $\G_s \ni \t g\to \check\Theta_\Pi(\t g)\det(1-g)\in \C$ 
%\eqref{main theorem for l<l' a, special odd 1} 
is a finite linear combination of irreducible characters of $\wt\G_s$.
\end{lemma}
\begin{prof}
Let
$\sigma$ denote the spin representation of $\G_s^0$ and let $\sigma^c$ be its contragradient representation. Then, by \cite[Ch. XI, III, p. 254]{Littlewood}
\begin{equation}\label{charcter of spin WF}
\det(1+g)=|\Theta_\sigma(g)|^2=\Theta_{\sigma\otimes\sigma^c}(g)\qquad (g\in \G_s^0)\,.
\end{equation}
Recall that for $(\G, \G')=(\Og_d, \Sp_{2l'}(\R))$, $\chi_+$ is a character of $\wt\G$.
Write $\check\Theta_\Pi(\t g)\det(1-g)=\check\Theta_\Pi(\t g)\chi_+(\t g)\det(1-g)\chi_+^{-1}(\t g)$. 
Decomposing $(\Pi\otimes \chi_+^{-1})^c\otimes\sigma\otimes\sigma^c=\sum_j \sigma_j$ into a finite sum of irreducible representations $\sigma_j$ of $\G_s$, we then obtain
\[
\check\Theta_\Pi(\t g)\det(1-g)=\Theta_{\Pi\otimes\chi_+^{-1}}\det(1-g)\chi_+^{-1}(\t g)=\sum_j \Theta_{\sigma_j}(g)\chi_+^{-1}(\t g)= \sum_j \check\Theta_{\sigma_j^c\otimes \chi_+}(\t g)\,,
\]
where $\check\Theta_{\sigma_j^c}(\t g)=\check\Theta_{\sigma_j^c}(g)$. 
\end{prof}

Let $\Oo_{m,s}\subseteq \Wv_s$ denote the maximal nilpotent $\G_s\times \G'$ orbit with invariant measure $\mu_{\Oo_{m,s}}\in \Ss'(\Wv_s)$. 
\begin{lemma}\label{inequalityof degrees}
The sharp inequality
\begin{equation}\label{inequalityof degrees1}
\deg \mu_{\Oo_{m}}>\deg \mu_{\Oo_{m,s}} +(\dim\Wv-\dim\Wv_s)
\end{equation}
holds, unless the dual pair $(\G,\G')$ is isomorphic to $(\Og_d, \Sp_{2l'}(\R))$ with $d>l'$. In this cases
\begin{equation}\label{inequalityof degrees2}
\deg \mu_{\Oo_{m}}=\deg \mu_{\Oo_{m,s}}+(\dim\Wv-\dim\Wv_s)\,.
\end{equation}
\end{lemma}
\begin{prof}
We know from Lemmas \ref{structure of t'-1tau(0)} and \ref{muSNK as a tempered homogeneous distribution} that
\[
\deg \mu_{\Oo_m}=\dim \Oo'_m-\dim \Wv=2l'\min\{d,l'\}-\min\{d,l'\}(\min\{d,l'\}-1)-d\,2l'\,,
\]
and similarly
\begin{multline*}
\deg \mu_{\Oo_{m,s}}=\dim \Oo'_{m,s}-\dim \Wv_s\\
=2l'\min\{d-1,l'\}-\min\{d-1,l'\}(\min\{d-1,l'\}-1)-(d-1)2l'\,.
\end{multline*}
Suppose $d\leq l'$. Then
\[
\deg \mu_{\Oo_m}=2l'd-d(d-1)-d\,2l'=-d(d-1)
\]
and, because $d-1< l'$,
\[
\deg \mu_{\Oo_{m,s}}=-(d-1)(d-2)=-d(d-1)+2(d-1)\,.
\]
Also,
\[
\dim\Wv-\dim\Wv_s=d2l'-(d-1)2l'=2l'>2(d-1)\,.
\]
Thus \eqref{inequalityof degrees1} follows.

Suppose $d> l'$. Then
\[
\deg \mu_{\Oo_m}=2l'l'-l'(l'-1)-d2l'
\]
and, because $d-1\geq l'$,
\[
\deg \mu_{\Oo_{m,s}}=2l'l'-l'(l'-1)-(d-1)2l'=\deg \mu_{\Oo_m}+2l'\,.
\]
Thus \eqref{inequalityof degrees2} follows.
\end{prof}
\begin{lemma}\label{easy case}
Suppose $d=m$. Then
\[
\underset{t\to 0+}{\lim}t^{\deg \mu_{\Oo_m}}M_{t^{-1}}^* T(\check\Theta_{\Pi}|_{\wt{\G\setminus(-\G^0)}})=0\,.
\]
\end{lemma}
\begin{prof}
Recall that
\[
M_{t^{-1}}^* T(\check\Theta_{\Pi}|_{\wt{\G\setminus(-\G^0)}})(\phi)
=T(\check\Theta_{\Pi}|_{\wt{\G\setminus(-\G^0)}})(\phi_{t^{-1}})\,,
\]
where
\[
\phi_{t^{-1}}(w)=t^{\dim \Wv}\phi(tw)\,.
\]
Suppose first we are in the situation described in Theorem \ref{theorem:main theorem for l<l', special odd 1}. Then 
$T(\check\Theta_{\Pi}|_{\wt{\G\setminus(-\G^0)}})(\phi)$ is a constant multiple of 
$T_s(\check{\Phi}_{\Pi}|_{\wt{-\G_s^0}})(\phi|_{\Wv_s})$ because $-\G_s^0=\G_s^0$. Notice that
\[
(\phi|_{\Wv_s})_{t^{-1}}(w)=t^{\dim \Wv_s}(\phi|_{\Wv_s})(tw)
\]
and as above
\[
M_{t^{-1}}^* T_s(\check{\Phi}_{\Pi}|_{\wt{-\G_s^0}})(\phi|_{\Wv_s})
=T_s(\check{\Phi}_{\Pi}|_{\wt{-\G_s^0}})((\phi|_{\Wv_s})_{t^{-1}})\,.
\]
Hence 
the decomposition of $\Pi_s$ into irreducibles and Theorem \ref{the dilation limit of intertwining distribution, intro, intermediate 1} imply
that
\[
t^{\deg \mu_{\Oo_{m,s}}}M_{t^{-1}}^* T_s(\check{\Phi}_{\Pi}|_{\wt{-\G_s^0}})(\phi|_{\Wv_s})\underset{t\to 0+}{\longrightarrow }K_s\mu_{\Oo_{m,s}}\,,
\]
where $K_s$ is a non-zero constant. Thus, for a constant $C_s$,
\begin{eqnarray*}
&&\hskip -1cm t^{\deg \mu_{\Oo_m}}M_{t^{-1}}^* T(\check\Theta_{\Pi}|_{\wt{\G\setminus(-\G^0)}})(\phi)\\
&&=C_s\, t^{\deg \mu_{\Oo_m}-\dim\Wv+\dim\Wv_s}
M_{t^{-1}}^* T_s(\check{\Phi}_{\Pi}|_{\wt{-\G_s^0}})(\phi|_{\Wv_s})\\
&&=C_s\, t^{\deg \mu_{\Oo_m}-\dim\Wv+\dim\Wv_s-\deg \mu_{\Oo_{m,s}}} \left(t^{\deg \mu_{\Oo_{m,s}}}M_{t^{-1}}^* T_s(\check{\Phi}_{\Pi}|_{\wt{-\G_s^0}})(\phi|_{\Wv_s})\right)\\
&&\underset{t\to 0+}{\longrightarrow } C_s\cdot 0\cdot K_s\mu_{\Oo_{m,s}}=0
\end{eqnarray*}
because, by Lemma \ref{inequalityof degrees}, 
\[
\deg \mu_{\Oo_m}-\dim\Wv+\dim\Wv_s-\deg \mu_{\Oo_{m,s}}>0\,.
\]
Lemma \ref{FiniteCombinationIrredChar} implies that a similar argument applies to the case of Theorem \ref{main theorem for l<l', special odd 1}.
\end{prof}

Because of \eqref{oddExceptionalCasen=1} and Lemma \ref{easy case}, it remains to compute \eqref{oddExceptionalCase} for $(\Og_d,\Sp_{2l'}(\R))$ with $d>m=l'$. According to Theorem \ref{theorem:main theorem for l<l', special odd 1}, we can also suppose $l'\geq l$ and $(l',l)\neq (1,1)$ when $d=2l$. This leads us to the cases $2l+1>l'$ for $(\G,\G')=(\Og_{2l+1},\Sp_{2l'}(\R))$, and $2l>l'\geq l$, with $(l',l)\neq (1,1)$, for 
$(\G,\G')=(\Og_{2l},\Sp_{2l'}(\R))$.

\begin{lemma}\label{O(2l+1)Sp(2m)}
Suppose $(\G,\G')=(\Og_{2l+1},\Sp_{2l'}(\R))$ with $2l+1>l'$. Let $\Pi$ be an irreducible representation of $\wt\G$ that occurs in the restriction of the Weil representation to $\wt\G$. Then, in the topology of $\Ss'(\Wv)$,
\begin{equation}\label{the dilation limit of intertwining distribution, intro, intermediate 1.1 O(2l+1)Sp(2m)}
t^{\deg \mu_{\Oo_m}}M_{t^{-1}}^* T(\check\Theta_{\Pi}|_{\wt{\G^0}})\underset{t\to 0+}{\longrightarrow}K^+\mu_{\Oo_m}\,,
\end{equation}
where 
\begin{equation}\label{the dilation limit of intertwining distribution, intro, intermediate 1.1 O(2l+1)Sp(2m)-2}
K^+=|S^{2l}| C_s 2^{\frac{1}{2}\dim\Wv_s}\chi_{\Pi\otimes \chi_+^{-1}}(-1) \int_{(\G_s^0)_N} \Theta_{\Pi\otimes \chi_+^{-1}}(g^{-1}) \det(1+g)\,dg\,,
\end{equation}
and
$|S^{2l}|$ is the area of the unit sphere, $C_s$ is as in Lemma \ref{distributiononG} for the group $\G_s$ acting on $\Wv_s$, and $(\G_s^0)_N$ is the stabilizer of $N$ in $\G_s^0$. 
\end{lemma}
\begin{prof}
Recall the formula \eqref{main theorem for l<l' a, special odd 1}. 
We know from Lemma \ref{FiniteCombinationIrredChar} that $\check\Theta_\Pi(\t g)\det(1-g)$ is a finite linear combination of irreducible characters of $\wt\G_s$. Since $\G^0_s=-\G^0_s$, we apply the argument used in the proof of Theorem \ref{the dilation limit of intertwining distribution, intro, intermediate 1}, together with \eqref{inequalityof degrees2}, to each individual representation of $\wt\G_s$ and sum the results. This shows that for $\phi\in\Ss(\Wv)$,
\begin{equation}\label{the dilation limit of intertwining distribution, intro, intermediate 1.1 O(2l+1)Sp(2m)1}
t^{\deg \mu_{\Oo_m}}M_{t^{-1}}^* T(\check\Theta_{\Pi}|_{\wt{\G^0}})(\phi)\underset{t\to 0+}{\longrightarrow }{K_s}\mu_{\Oo_{m,s}}(\phi^\G|_{\Wv_s})\,,
\end{equation}
where $\mu_{\Oo_{m,s}}$ is the normalized  measure on the maximal nilpotent $\G_s\G'$-orbit $\Oo_{m,s}\subseteq \Wv_s$ and
\begin{equation}\label{the dilation limit of intertwining distribution, intro, intermediate 1.1 O(2l+1)Sp(2m)2}
K_s=C_s 2^{\frac{1}{2}\Wv_s}\chi_{\Pi\otimes \chi_+^{-1}}(-1) \int_{(\G_s^0)_N} \Theta_{\Pi\otimes \chi_+^{-1}}(g^{-1}) \det(1+g)\,dg\,.
\end{equation}
Since, by Corollary \ref{diagram10}
\[
\mu_{\Oo_{m,s}}(\phi^\G|_{\Wv_s})=|S^{2l}|\mu_{\Oo_{m}}(\phi)\,,
\]
\eqref{the dilation limit of intertwining distribution, intro, intermediate 1.1 O(2l+1)Sp(2m)} follows.
\end{prof}
\begin{lemma}\label{O(2l)Sp(2m)}
Suppose $(\G,\G')=(\Og_{2l},\Sp_{2l'}(\R))$ with $2l>l'\geq l$. Let $\Pi$ be an irreducible representation of $\wt\G$ that occurs in the restriction of the Weil representation to $\wt\G$ and whose character is not supported on $\wt{\G^0}$. Then, in the topology of $\Ss'(\Wv)$,
\begin{equation}\label{the dilation limit of intertwining distribution, intro, intermediate 1.1 O(2l)Sp(2m)}
t^{\deg \mu_{\Oo_m}}M_{t^{-1}}^* T(\check\Theta_{\Pi}|_{\wt{\G\setminus\G^0}})\underset{t\to 0+}{\longrightarrow }K^+\mu_{\Oo_m}\,,
\end{equation}
where $K^+=0$ if $2l=l'+1$ and 
\begin{multline}\label{the dilation limit of intertwining distribution, intro, intermediate 1.1 O(2l)Sp(2m)-2}
K^+=C(\Pi)
|S^{2l-1}||S^{2l-2}| C_{ss} 2^{\frac{1}{2}\dim\Wv_{ss}}
\chi_{\Pi\otimes \chi_+^{-1}}(-1)
\int_{(\G_{ss}^0)_N} \Theta_{\Pi\otimes \chi_+^{-1}}(g^{-1}) \det(1+g)\,dg
\end{multline}
for $2l>l'+1$. In \eqref{the dilation limit of intertwining distribution, intro, intermediate 1.1 O(2l)Sp(2m)-2}, 
$C(\Pi)=\pm 1$ and  $C_{ss}$ is as in Lemma \ref{distributiononG} for the group $\G_{ss}$, isomorphic to $\Og_{2l-2}$, acting on $\Wv_{ss}$. 
\end{lemma}
\begin{prof}
Formulas \eqref{main theorem for l<l' a, special odd 1} and \eqref{inequalityof degrees2} imply that
\begin{equation}\label{the dilation limit of intertwining distribution, intro, intermediate 1.1 O(2l)Sp(2m).1}
\underset{t\to 0+}{\lim}t^{\deg \mu_{\Oo_m}}M_{t^{-1}}^* T(\check\Theta_{\Pi}|_{\wt{\G\setminus \G^0}})(\phi)
=C(\Pi) 
\underset{t\to 0+}{\lim} t^{\deg \mu_{\Oo_{m,s}}}
M_{t^{-1}}^* T_s(\check\Theta_{\Pi_s}|_{\wt{\G_s^0}})(\phi^\G|_{\Wv_s})\,.
\end{equation}
Suppose first that $2l=l'+1$. The defining space of $\G_s$ has dimension $d_s=2l-1=l'$. Lemma \ref{easy case} applies then to the dual pair $(\G_s,\G')$, yielding 
$$
\underset{t\to 0+}{\lim} t^{\deg \mu_{\Oo_{m,s}}}
M_{t^{-1}}^* T_s(\check\Theta_{\Pi_s}|_{\wt{\G_s^0}})(\phi^\G|_{\Wv_s})=0.
$$
Suppose now that $2l>l'+1$. Then $2(l-1)+1=2l-1>l'$ and 
Lemma \ref{O(2l+1)Sp(2m)} shows that
\begin{equation}\label{the dilation limit of intertwining distribution, intro, intermediate 1.1 O(2l)Sp(2m).2}
\underset{t\to 0+}{\lim} t^{\deg \mu_{\Oo_{m,s}}}
M_{t^{-1}}^* T_s(\check\Theta_{\Pi_s}|_{\wt{\G_s^0}})(\phi^\G|_{\Wv_s}) 
=K^+_s \mu_{\Oo_{m,s}}(\phi^\G|_{\Wv_s})\,,
\end{equation}
where
\begin{equation}\label{the dilation limit of intertwining distribution, intro, intermediate 1.1 O(2l)Sp(2m).3}
K^+_s=|S^{2l-2}| C_{ss} 2^{1+\frac{1}{2}\dim\Wv_{ss}}\chi_{\Pi\otimes \chi_+^{-1}}(-1) \int_{(\G_{ss}^0)_N} \Theta_{\Pi\otimes \chi_+^{-1}}(g^{-1}) \det(1+g)\,dg\,,
\end{equation}
with  $C_{ss}$ as in Lemma \ref{distributiononG} for the group $\G_{ss}$, isomorphic to $\Og_{2l-2}$, acting on $\Wv_{ss}$.

Since, by Corollary \ref{diagram10}
\[
\mu_{\Oo_{m,s}}(\phi^\G|_{\Wv_{s}})=|S^{2l-1}|\mu_{\Oo_{m}}(\phi)\,,
\]
\eqref{the dilation limit of intertwining distribution, intro, intermediate 1.1 O(2l)Sp(2m)} follows.
\end{prof}
\begin{lemma}\label{theConstants}
Let $K$ be as in Theorem \ref{the dilation limit of intertwining distribution, intro, intermediate 1}.
With the notation and assumptions of Lemmas \ref{O(2l+1)Sp(2m)} and \ref{O(2l)Sp(2m)},
\begin{equation}\label{theConstants1}
K+K^+\ne 0\,.
\end{equation}
\end{lemma}
\begin{prof}
Recall that 
$$
K=C 2 ^{\frac{1}{2}\dim \Wv} \chi_{\Pi\otimes\chi_+^{-1}}(-1)\int_{\G_N^0}\Theta_{\Pi\otimes\chi_+^{-1}}( g_N)\,d g_N\,.
$$
In the situation of Lemma \ref{O(2l+1)Sp(2m)},
$$
K^+=|S^{2l}| C_s 2^{\frac{1}{2}\dim\Wv_s}\chi_{\Pi\otimes \chi_+^{-1}}(-1) \int_{(\G_s^0)_N} \Theta_{\Pi\otimes \chi_+^{-1}}(g^{-1}) \det(1+g)\,dg
\,.
$$
The constants $C$ and $C_{s}$ as well as both integrals are integers, and $\chi_{\Pi\otimes \chi_+^{-1}}(-1)=\pm 1$. Moreover, % (see for instance \cite[6.1.8 and 6.1.12]{AbramowitzStegun}), 
$$
|S^{2l}|=\frac{2\pi^{\frac{2l+1}{2}}}{\Gamma(l+\frac{1}{2})}=\frac{2\cdot 4^l l!}{(2l)!} \,\pi^l\,,
$$
which is an irrational number.
Hence \eqref{theConstants1} follows.

In the situation of Lemma \ref{O(2l)Sp(2m)}, with $2l>l'+1$, 
\begin{multline*}
K^+=C(\Pi)\left(\frac{i}{2}\right)^{l'}
|S^{2l-1}||S^{2l-2}| C_{ss} 2^{1+\frac{1}{2}\dim\Wv_{ss}}\\ \times\chi_{\Pi\otimes \chi_+^{-1}}(-1) 
\int_{(\G_{ss}^0)_N} \Theta_{\Pi\otimes \chi_+^{-1}}(g^{-1}) \det(1+g)\,dg,
\end{multline*}
where both $C_{ss}$ and the integral are integers. Since 
$$
|S^{2l-1}||S^{2l-2}|=\frac{2\pi^l}{(l-1)!} \, \frac{2\pi^{l-\frac{1}{2}}}{\Gamma((l-1)+\frac{1}{2})}=
4^l \frac{\pi^{2l-1}}{(2l-2)!}
$$
is irrational, \eqref{theConstants1} follows. Finally, if $2l=l'+1$, then $K+K^+=K\neq 0$.
\end{prof}

Now we easily deduce Theorem \ref{the dilation limit of intertwining distribution, intro} from Theorem \ref{the dilation limit of intertwining distribution, intro, intermediate 1} and Lemmas \ref{easy case} to \ref{theConstants}.

\section{\textbf{The wave front set of $\Pi'$}}
\label{WFSPI'}
\setcounter{thh}{0}
Recall from Theorem \ref{the dilation limit of intertwining distribution, intro}
that
\begin{equation}\label{noproofWF1}
t^{\deg \mu_\Oo}M_{t^{-1}}^*T(\check\Theta_\Pi) \underset{t\to 0+}\to C\,\mu_{\Oo_m}\,,
\end{equation}
as tempered distributions on $\Wv$, where $C$ is a non-zero constant. 
Hence,
in the topology of $\Ss'(\g')$,
\begin{equation}\label{noproofWF2}
t^{\dim{\Oo_m'}}M_{t^{2}}^*\widehat{\tau'_*( T(\check\Theta_\Pi))}\underset{t\to 0+}{\longrightarrow }C\widehat{\mu_{\Oo_m'}}\,,
\end{equation}
where $$\tau'_*( T(\check\Theta_\Pi))(\psi)=T(\check\Theta_\Pi)(\psi\circ\tau')\,,$$  $\widehat{\tau'_*( T(\check\Theta_\Pi))}$ is a Fourier transform of the tempered distribution $\tau'_*( T(\check\Theta_\Pi))$ on $\g'$, and similarly for $\mu_{\Oo_m'}$.

There is an easy-to-verify inclusion $WF(\Pi')\subseteq\overline{\Oo'}$, \cite[(6.14)]{PrzebindaUnipotent} and a formula for the character $\Theta_{\Pi'}$ in terms of $\widehat{\tau'_*(T(\check\Theta_\Pi))}$, namely,
\begin{equation}\label{noproofWF3}
\frac{1}{\sigma}\cdot \t c_-^*\Theta_{\Pi'}=\widehat{\tau'_*(T(\check\Theta_\Pi))}\,, 
\end{equation}
where $\sigma$ is a smooth function, \cite[Theorem 6.7]{PrzebindaUnipotent}.
By combining this with Lemma \ref{wave front set 1} one completes the argument.

%%%%%%%%%%%%%%%%%%
%%%%%     APPENDICES
%%%%%%%%%%%%%%%%%%
\appendix
\renewcommand{\thethh}{\Alph{section}.\fontindex{thh}}
\renewcommand{\theequation}{\Alph{section}.\fontindex{equation}}

\setcounter{thh}{0}
\setcounter{equation}{0}
\section{\textbf{Nilpotent orbits and moment maps}}
\label{Appendix nil-orbits}
\subsection{\textbf{Proof of Lemma \ref{structure of t'-1tau(0)}}\label{proof of lemma 1}}
The equality $w^*w=0$ means that the pullback of the form $(\cdot,\cdot)'$ via $w\in
\Wv=\Hom(\Vv,\Vv')$ is zero. 
Equivalently, the range of $w$ is an isotropic subspace of $\Vv'$. Let us fix a maximal isotropic subspace $\Xv'\subseteq \Vv'$. We may assume that the range of $w$ is contained in $\Xv'$. Thus $w\in \Hom(\Vv,\Xv')$. Under the action of $\G$ and $\GL(\Xv')$, the set $\Hom(\Vv,\Xv')$ breaks down into a union of orbits. Each orbit consists of  maps of rank $k\in\{0, 1, 2, \dots, m\}$.
Since by Witt's Theorem $\GL(\Xv')\subseteq \G'$ and since the action of $\G'$ cannot change the rank of an element of $\Hom(\Vv,\Vv')$, 
\eqref{structure of t'-1tau(0)0} will follow as soon as we compute the dimension of $\Oo_k$. We shall do it in terms of matrices.
We keep the notation introduced in section \ref{A slice through nilpotent element}. Let $F,F'$ be as in 
\eqref{matrix F.I} and choose
\begin{equation}\label{the nilpotent element a}
N=\begin{pmatrix}
I_k & 0\\
0 & 0\\
0 & 0
\end{pmatrix}
\end{equation}
as in \eqref{the nilpotent element}. 
The Lie algebra $\g$ consists of the skew-hermitian matrices of size $d$ with coefficients in $\Dc$ and $\g'$ of matrices of size $d'$ and coefficients in $\Dc$, described in $(k,d'-2k,k)$ block-form as
\[
x'=\begin{pmatrix}
x_{11} & x_{12} & x_{13}\\
x_{21} & x_{22} & -F'{}^{-1}\overline{x_{12}}^t\\
x_{31} & -\overline{x_{21}}^t F' & -\overline{x_{11}}^t
\end{pmatrix}\,,\ \ \ x_{13}=\overline{x_{13}}^t\,,\ \ \ x_{31}=\overline{x_{31}}^t\,,\ \ \ \overline{x_{22}}^tF'+F' x_{22}=0\,.
\]
The Lie algebra of the stabilizer of $N$ in $\G \times \G'$ consists of pairs of matrices $(x, x')\in\g\times\g'$ such that
\[
x=\begin{pmatrix}
y_{11} & 0\\
0 & y_{22}
\end{pmatrix}\,\ \ \ 
x'=\begin{pmatrix}
x_{11} & x_{12} & x_{13}\\
0 & x_{22} & -F'{}^{-1}\overline{x_{12}}^t\\
0 & 0 & -\overline{x_{11}}^t
\end{pmatrix}\,,\ \ \ x_{11}=y_{11}\,.
\]
This implies the dimension formula in \eqref{structure of t'-1tau(0)0}. 
Since 
\[
NN^*
=\left(
\begin{array}{cccc}
0 & 0 & I_k\\
0 & 0 & 0\\
0 & 0 & 0
\end{array}
\right)\,,
\]
the stabilizer of $NN^*$ in $\g'$ consists of matrices of the form
\[
\left(
\begin{array}{cccc}
x_{11} & x_{12} & x_{13}\\
0 & x_{22} & -F'{}^{-1}\overline{x_{12}}^t\\
0 & 0 & -\overline{x_{11}}^t
\end{array}
\right)\,,\ \ \ x_{11}=-\overline{x_{11}}^t\,,
\]
and (\ref{structure of t'-1tau(0)1}) follows.

\begin{remark}
The fact that, for $\G$ compact, $\tau^{-1}(0)$ is the closure of a single $\G\G'$-orbit and a finite union of $\G\G'$-orbits was proved in \cite[Lemma (2.16)]{PrzebindaUnipotent}. If in addition the pair $(\G,\G')$ is in the stable range with $\G$ the smaller member, then 
\cite[Lemma (2.19)]{PrzebindaUnipotent} also computes the dimension of the maximal orbit. So, Lemma \ref{structure of t'-1tau(0)} is a generalization of these statements. 

As for other references in the literature, notice that given a dual pair $(\G,\G')$, there are two moment maps one usually considers:
$$
\tau_{\g'}: \Wv\to {\g'}^*
\qquad \text{and} \qquad \tau_{\mathfrak{s}'}:\Wv\to {\mathfrak{s}'}^*\,,
$$
where $\g'=\k'\oplus\mathfrak{s}'$ is a Cartan decomposition and the second map is obtained from the first one by composing with the restriction to $\mathfrak{s}'$. 
The first map leads to $\G'$-orbits and the second to $\K'_\C$-orbits.

Our Lemma \ref{structure of t'-1tau(0)} deals with the maps $\tau_{\g'}$, whereas 
the articles \cite{Nishiyama-Zhu-Duke,Nishiyama-Zhu-Transactions} deal with the map $\tau_{\mathfrak{s}'}$ only. Therefore they do not provide any direct proof of Lemma \ref{structure of t'-1tau(0)}. Moreover, these references consider only dual pairs in the stable range. We do not have this assumption in our Lemma. 

Furthermore, these
two moment maps are sort of ``equivalent'' in the stable range as was shown in
\cite{DaszKrasPrzebindaK-S2}, but they are not ``equivalent'' beyond the stable range. 
\end{remark}

\subsection{\textbf{Proof of Lemma \ref{muSNK as a tempered homogeneous distribution}}}\label{proof of lemma 2}
Let $N\in \Oo_k$ as in \ref{proof of lemma 1}. The stabilizer of the image of $N$ in $\Vv'$ is a parabolic subgroup $\Pg'\subseteq \G'$ with Langlands decomposition $P'=\GL_k(\Dc)\G''\N'$, where $\G''$ is an isometry group of the same type as $\G'$ and $\N'$ is the unipotent radical. As a $\GL_k(\Dc)$--module, $\n'$, the Lie algebra of $\N'$, is isomorphic to $M_{k,d'-2k}(\Dc)\oplus \Hs_k(\Dc)$, where $\Hs_k(\Dc)\subseteq  M_{k}(\Dc)$ stands for the space of the hermitian matrices. 
In the notation of \ref{proof of lemma 1}, 
\begin{align*}
\n'&=\left\{ \begin{pmatrix}
0 & x_{12} & x_{13}\\
0 & 0 & -F'{}^{-1}\overline{x_{12}}^t\\
0 & 0 & 0
\end{pmatrix}; \, x_{12}\in M_{k,d'-2k}(\Dc), \, x_{13}\in \mathcal{H}_d(\Dc)\right\}\,,\\
\g''&=\left\{ \begin{pmatrix}
0 & 0 & 0\\
0 & x_{22} & 0\\
0 & 0 & 0
\end{pmatrix}; \, \overline{x_{22}}^tF'+F' x_{22}=0\right\}\,,\\
\GL_k(\Dc)&\equiv \left\{ \begin{pmatrix}
a & 0 & 0\\
0 & I_{d'-2k} & 0\\
0 & 0 & (\overline{a}^t)^{-1}
\end{pmatrix}; \, a\in \GL_k(\Dc)\right\}\,.
\end{align*}

Hence the absolute value of the determinant of the adjoint action of an element $a\in \GL_k(\Bbb D)$ on the real vector space $\n'$ is equal to 
\[
|\det_\R \Ad(a)_{\n'}|=|\det_\R(a)|^{d'-2k+\frac{2\dim_\R \Hs_k(\Bbb D)}{k\dim_\R\Bbb D}}\,.
\]
Since $\G'=\K'\Pg'$, where $\K'$ is a maximal compact subgroup, the Haar measure on $\G'$ may be written as
\[
dg'=|\det_\R \Ad(a)_{\n'}|\,dk\,da\,dg''\,dn'\,.
\]

Since the stabilizer of $N$ in $\G'$ is equal to $\G''\N'\subseteq \Pg'$, the $\G'$ orbit of 
$N$ defines a tempered distribution on $\Wv$ by
\[
\int_\Wv \phi(w)\,d\mu_{\G' N}(w)=\int_{\GL_k(\Bbb D)}\int_{\K'}\phi(kaN)|\det \Ad(a)_{\n'}|\,dk\,da \qquad (\phi\in \Ss(\Wv))\,.
\]
Since
$$
t \begin{pmatrix}
a & 0 & 0\\
0 & I_{d'-2k} & 0\\
0 & 0 & (\overline{a}^t)^{-1}
\end{pmatrix} 
\begin{pmatrix}
I_k & 0\\
0 & 0  \\
0 & 0
\end{pmatrix}
=
\begin{pmatrix}
ta & 0\\
0 & 0  \\
0 & 0
\end{pmatrix}
=
\begin{pmatrix}
ta & 0 & 0\\
0 & I_{d'-2k} & 0\\
0 & 0 & (\overline{ta}^t)^{-1}
\end{pmatrix} 
\begin{pmatrix}
I_k & 0\\
0 & 0  \\
0 & 0
\end{pmatrix}
$$
and for $t>0$
\begin{align*}
\int_\Wv \phi_t(w)\, d\mu_{\G' N}(w)
&=t^{-\dim\Wv}
\int_{\GL_k(\Bbb D)}\int_{\K'}\phi(k(t^{-1}a)N)|\det \Ad(a)_{\n'}|\,dk\,da\\
&= t^{-\dim\Wv}
\int_{\GL_k(\Bbb D)}\int_{\K'}\phi(kaN)|\det_\R(ta)|^{d'-2k+\frac{2\dim_\R \Hs_k(\Bbb D)}{k\dim_\R\Bbb D}}\,dk\,da\\
&=t^{-\dim\Wv +\Big(d'-2k+\frac{2\dim_\R \Hs_k(\Bbb D)}{k\dim_\R\Bbb D}\Big)k\dim_\R\Bbb D} \int_\Wv \phi(w)\, d\mu_{\G' N}(w)\,,
\end{align*}
this distribution is homogeneous of degree
\[
(d'-2k+\frac{2\dim \Hs_k(\Bbb D)}{k\dim_\R\Bbb D})\dim_\R\Bbb D -\dim\Wv\,.
\]
Thus it remains to check that
\[
(d'-2k+\frac{2\dim \Hs_k(\Bbb D)}{k\dim_\R\Bbb D})\,k \dim_\R\Bbb D=d'k \dim_\R(\Dc)-2\dim_\R\SHs_k(\Dc),
\]
which is easy, because $M_{k,k}(\Dc)=\Hs_k(\Bbb D)\oplus \SHs_k(\Dc)$. In order to conclude the proof we notice that the orbital integral on the $\G\times \G'$-orbit of $N$ is (up to a positive multiple) the $\G$-average of the orbital integral we just considered:
\[
\int_{\Wv} \phi(w)\,d\mu_{\Oo_k}(w)=\int_\G\int_{\GL_k(\Bbb D)}\int_{\K'}\phi(kaNg)\det \Ad(a)_{\n'}\,dk\,da\,dg\,.
\]
\subsection{\textbf{A few facts about nilpotent orbits}}\label{a few facts}
Let $\g'$ be a semisimple Lie algebra over $\C$. Then there is a unique non-zero nilpotent orbit in $\g'$ of minimal dimension which is contained in the closure of any non-zero nilpotent orbit, \cite[Theorem 4.3.3, Remark 4.3.4]{CollMc}. The dimension of that orbit is equal to one plus the number of positive roots not orthogonal to the highest root, relative to a choice of a Cartan subalgebra and a choice of positive roots,
\cite[Lemma 4.3.5]{CollMc}. Thus in the case $\g'=\sp_{2l'}(\C)$,  the dimension of the minimal non-zero nilpotent orbit is equal to $2l'$.
This is precisely the dimension of the defining module for the symplectic group $\Sp_{2l'}(\C)$, which may be viewed as the symplectic space for the dual pair $(\Og_1, \Sp_{2l'}(\C))$. 

Consider the dual pair $(\G, \G')=(\Og_1, \Sp_{2l'}(\R))$, with the symplectic space $\Wv$ and the unnormalized moment map $\tau':\Wv\to \g'$. 
Since $\Wv\setminus\{0\}$ is a single $\G'$-orbit, so is $\tau'(\Wv\setminus\{0\})$. Further, $\dim(\tau'(\Wv\setminus\{0\}))=\dim(\Wv)=2l'$.
Hence, $\tau'(\Wv\setminus\{0\})\subseteq \g'$ is a minimal non-zero $\G'$-orbit. In fact, there are only two such orbits, \cite[Theorem 9.3.5]{CollMc}. In terms of dual pairs, the second one is obtained from the same dual pair, with the symplectic form replaced by its negative (or equivalently the symmetric form on the defining module for $\Og_1$  replaced by its negative).

Consider an irreducible dual pair $(\G, \G')$ with $\G$ compact. Denote by $l$ the dimension of a Cartan subalgebra of $\g$ and by $l'$  the dimension of a Cartan subalgebra of $\g'$. 
Let us identify the corresponding symplectic space $\Wv$ with $\Hom(\V_{\overline 1}, \V_{\overline 0})$ as in \cite[sec.2]{PrzebindaUnipotent}. 

Recall that $\Wv_\g$ denotes the maximal subset of $\Wv$ on which the restriction of the unnormalized moment map $\tau:\Wv \to \g$ is a submersion. 
Then \cite[Lemma 2.6]{PrzebindaUnipotent} shows that $\Wv_\g$ consists of all the elements $w\in\Wv$ such that for any $x\in\g$,
\begin{equation}\label{xw=0 implies x=0}
xw=0\ \text{implies}\ x=0\,.
\end{equation}
The condition (\ref{xw=0 implies x=0}) means that $x$ restricted to the image of $w$ is zero. But in that case $x$ preserves the orthogonal complement of that image. Thus we need to know that the Lie algebra of the isometries of that orthogonal complement is zero. This happens if $w$ is surjective or if $\G$ is the orthogonal group and  the dimension of the image of $w$ in $\V_{\overline 0}$ is $\geq \dim(\V_{\overline 0})-1$. Thus
\begin{equation}\label{wg non-empty}
\Wv_\g\ne \emptyset\; \text{if and only if}\; l\leq l'\,.
\end{equation}
Consider in particular the dual pair $(\G, \G')=(\Og_3,\Sp_{2l'}(\R))$ with $1\leq l'$. We see from the above discussion that $\Wv_\g$ consists of elements of rank $\geq 2$. Hence, $\Wv\setminus(\Wv_\g\cup\{0\})$ consists of elements $w$ 
of rank equal to $1$. By replacing $\V_{\overline 0}$ with the image of $w$, we may consider $w$ as an element of the symplectic space for the pair $(\Og_1, \Sp_{2l'})$. Hence the image of $w$ under the moment map generates a minimal non-zero nilpotent orbit in $\g'$. 

If  $(\G, \G')=(\Og_2,\Sp_{2l'}(\R))$, with $1\leq l'$, then $\Wv_\g$ consists of elements of rank $\geq 1$. Therefore $\Wv\setminus\Wv_\g=\{0\}$.
%%

%% A different indexation for the appendix
%\setcounter{section}{2}
\setcounter{thh}{0}
\setcounter{equation}{0}
\section{\textbf{Pull-back of a distribution via a submersion}}

%%%
We collect here some textbook results which are attributed to Ranga Rao in \cite{BarVogAs}. These results date back to the time before the textbook \cite{Hormander} was available. 

We shall use the definition of a smooth manifold and a distribution on a smooth manifold as described in \cite[sec. 6.3]{Hormander}. Thus, if $M$ is a smooth manifold of dimension $m$ and 
\[
M\supseteq M_\kappa\overset{\kappa}{\longrightarrow}\widetilde{M}_\kappa\subseteq \R^m
\]
is any coordinate system on $M$,
then a distribution $u$ on $M$ is the collection of distributions $u_\kappa\in \mathcal D'(\widetilde{M}_\kappa)$ such that
\begin{equation}\label{charts and distributions}
u_{\kappa_1}=(\kappa\circ\kappa_1^{-1})^*u_\kappa\,.
\end{equation}
Suppose $W$ is another smooth manifold of dimension $n$ and $v$ is a distribution on $W$. Thus for any coordinate system 
\[
W\supseteq W_\lambda\overset{\lambda}{\longrightarrow}\widetilde{W}_\lambda\subseteq \R^n
\]
we have a distribution $v_\lambda\in \mathcal D'(\widetilde{W}_\lambda)$ such that the condition (\ref{charts and distributions}) holds. Suppose 
\[
\sigma:M\to W
\]
is a submersion. Then for every $\kappa$ there is a unique distribution $u_\kappa\in \mathcal D'(\widetilde{M}_\kappa)$ such that
\begin{equation}\label{definition of pullback of distributions on manifolds}
u_{\kappa}|_{(\lambda\circ \sigma\circ\kappa^{-1})^{-1}(\widetilde{W}_\lambda)}=(\lambda\circ \sigma\circ\kappa^{-1})^*v_\lambda\,.
\end{equation}
Since
\begin{eqnarray*}
(\kappa\circ\kappa_1^{-1})^*(\lambda\circ \sigma\circ\kappa^{-1})^*v_\lambda
=(\lambda\circ \sigma\circ\kappa^{-1}\circ\kappa\circ\kappa_1^{-1})^*v_\lambda
=(\lambda\circ\sigma\circ\kappa_1^{-1})^*v_\lambda\,,
\end{eqnarray*}
the $u_\kappa$ satisfy the condition (\ref{charts and distributions}). The resulting distribution $u$ is denoted by $\sigma^*v$ and is called the pullback of $v$ from $W$ to $M$ via $\sigma$.
\begin{proposition}\label{I.1}
Let $M$ and $W$ be smooth manifolds and let $\sigma:M\to W$ be a surjective submersion. Suppose $u_n\in \mathcal D'(W)$ is a sequence of distributions such that
\begin{equation}\label{I.1.1}
\underset{n\to\infty}{\lim}\ \sigma^*u_n=0 \quad \text{in the topology of $\mathcal D'(M)$}\,.
\end{equation}
Then
\begin{equation}\label{I.1.2}
\underset{n\to\infty}{\lim}\ u_n=0 \quad \text{in the topology of $\mathcal D'(W)$}\,.
\end{equation}
In particular, the map $\sigma^*:\mathcal D'(W)\to \mathcal D'(M)$ is injective.  

More generally, if
$u_n\in \mathcal D'(W)$ and $\t u\in \mathcal D'(M)$ are  such that
\begin{equation}\label{I.1.10}
\underset{n\to\infty}{\lim}\ \sigma^*u_n=\t u \quad \text{in the topology of $\mathcal D'(M)$}\,,
\end{equation}
then there is a distribution $u\in  \mathcal D'(W)$ such that
\begin{equation}\label{I.1.20}
\underset{n\to\infty}{\lim}\ u_n=u \quad \text{in the topology of $\mathcal D'(W)$}
\end{equation}
and $\t u=\sigma^* u$.
\end{proposition}
\begin{prof}
By the definition of a distribution on a manifold, as in \cite[sec.6.3]{Hormander}, we may assume that $M$ is an open subset of $\R^m$ and $W$ is an open subset of $\R^n$. 

We recall the definition of the pull-back
\begin{equation}\label{I.pullback} 
\sigma^*:\mathcal D'(W)\to \mathcal D'(M)
\end{equation}
from the proof of Theorem 6.1.2 in \cite{Hormander}. Fix a point $x_0\in M$ and a smooth map $g:M\to \R^{m-n}$ such that
\[
\sigma\oplus g:M\to \R^n\times \R^{m-n}
\]
has a bijective differential at $x_0$. By the Inverse Function Theorem there is an open neighborhood $M_0$ of $x_0$ in $M$ such that
\[
\left(\sigma\oplus g\right)|_{M_0}:M_0\to Y_0
\]
is a diffeomorphism onto an open neighborhood $Y_0$ of $(\sigma\oplus g)(x_0)=(\sigma(x_0), g(x_0))$ in $\R^n\times \R^{m-n}$.
Let 
\[
h:Y_0\to M_0
\]
denote the inverse. For $\phi\in C_c^\infty(M_0)$ define $\Phi\in C_c^\infty(Y_0)$ by 
\begin{equation}\label{Phiphi}
\Phi(y)=\phi(h(y))|\det\,h'(y)| \qquad (y\in Y_0)\,.
\end{equation}
Then
\begin{equation}\label{I.1.3}
\sigma^*u(\phi)=u\otimes 1(\Phi) \qquad (u\in \mathcal D'(W),\ \phi\in C_c^\infty(M_0))\,.
\end{equation}
By localization this gives the pull-back (\ref{I.pullback}).

Let $W_0$ be an open neighborhood of $\sigma(x_0)$ in $W$ and let $X_0$ be an open neighborhood of $g(x_0)$ in $\R^{m-n}$ such that
\[
W_0\times X_0\subseteq Y_0\,.
\]
Fix a function $\eta\in C_c^\infty(X_0)$ such that
\[
\int_{X_0}\eta(x)\,dx=1\,.
\]
Given $\psi\in C_c^\infty(W_0)$ define
\[
\Phi(x',x'')=\psi(x')\eta(x'') \qquad (x'\in W_0,\ x''\in X_0)\,.
\]
Then $\Phi$ defines $\phi$ via (\ref{Phiphi}) and
\[
\sigma^*u(\phi)=u(\psi)\,.
\]
Hence the assumption (\ref{I.1.1}) implies 
\[
\underset{n\to\infty}{\lim}\ u_n(\psi)=0 \qquad (\psi\in C_c^\infty(W_0))\,.
\]
Thus, by \cite[Theorem 2.1.8]{Hormander}, 
\[
\underset{n\to\infty}{\lim\ }u_n|_{W_0}=0
\]
in $\mathcal D'(W_0)$. Since the point $x_0\in M$ is arbitrary, the claim (\ref{I.1.2}) follows by localization.

Similarly,  the assumption (\ref{I.1.10}) implies  that for any $\psi\in C_c^\infty(W_0)$
\[
\underset{n\to\infty}{\lim}\ u_n(\psi)= \underset{n\to\infty}{\lim}\ \sigma^*u_n(\phi)=\t u(\phi) 
\]
exists. Thus, by \cite[Theorem 2.1.8]{Hormander}, there is $u\in \mathcal D'(W_0)$ such that
\[
\underset{n\to\infty}{\lim\ }u_n|_{W_0}=u\,.
\]
By the continuity of $\sigma^*$, $\sigma^* u=\t u$. Again, since the point $x_0\in M$ is arbitrary, the claim follows by localization.
\end{prof}
\begin{lemma}\label{pull-back of differential operators}
Let $M$ and $W$ be smooth manifolds and let $\sigma:M\to W$ be a surjective submersion. Then for any smooth differential operator $D$ on $W$ there is, not necessary unique, smooth differential operator $\sigma^*D$ on $M$ such that
\[
\sigma^*(u\circ D) =(\sigma^*u)\circ (\sigma^*D) \qquad (u\in \mathcal D'(W))\,.
\]
If $D$ annihilates constants then so does $\sigma^*D$. 
The operator $\sigma^*D$ is unique if $\sigma$ is a diffeomorphism.
\end{lemma}
\begin{prof}
Suppose $\sigma$ is a diffeomorphism between two open subsets of $\R^n$. Then
\[
\sigma^*u(\phi)=u(\phi\circ\sigma^{-1}|\det((\sigma^{-1})')|) \qquad (\phi\in C_c^\infty(M))\,.
\]
Let 
\[
(\sigma^*D)(\phi)=(D(\phi\circ\sigma^{-1}))\circ\sigma \qquad (\phi\in C_c^\infty(M))\,.
\]
Hence
\begin{align*}
\sigma^*(u\circ D)(\phi)&=(u\circ D)(\phi\circ\sigma^{-1}|\det((\sigma^{-1})')|)\\
&=u(D(\phi\circ\sigma^{-1}|\det((\sigma^{-1})')|))\\
&=u((D(\phi\circ\sigma^{-1})\circ\sigma)\circ\sigma^{-1}  |\det((\sigma^{-1})')|)\,.
\end{align*}

Using the local classification of the submersions modulo the diffeomorphism \cite[16.7.4]{DieudonneElements}, we may assume that $\sigma$ is a linear projection
\[
\sigma:\R^{m+n}\ni (x,y)\to x\in \R^n\,,
\]
in which case the lemma is obvious.
\end{prof}

Suppose $M$ is a Lie group. Then there are functions $m_\kappa\in C^\infty(\widetilde{M}_\kappa)$ such that the formula
\begin{equation}\label{haar measure locally}
\int_M \phi\circ \kappa(y)\,d\mu_M(y)=\int_{\tilde M_\kappa}\phi(x) m_\kappa(x)\,dx \qquad (\phi\in C_c^\infty(\widetilde{M}_\kappa))
\end{equation}
defines a left-invariant Haar measure on $M$. We shall tie the normalization of the Haar measure $d\mu_M(y)$ on $M$ to the normalization of the Lebesgue measure $dx$ on $\R^m$ by requiring that near the identity,
\begin{equation}\label{haar measure locally normalized}
m_{\exp^{-1}}(x)=\det\left(\frac{1-e^{-\ad(x)}}{\ad(x)}\right)\,,
\end{equation}
as in \cite[Theorem 1.14, page 96]{HelgasonGeomtric}. Collectively, the distributions $m_\kappa(x)\,dx\in \mathcal D'(\widetilde{M}_\kappa)$ form a distribution density on $M$. (See \cite[sec. 6.3]{Hormander} for the definition of a distribution density.) 

Suppose $W$ is another Lie group with left Haar measure given by
\[
\int_W \psi\circ \lambda(y)\,d\mu_W(y)=\int_{\tilde W_\lambda}\phi(x) w_\lambda(x)\,dx \qquad (\psi\in C_c^\infty(\widetilde{W}_\lambda))\,,
\]
and let $\sigma:M\to W$ be a submersion. 
Given any distribution density $v_\lambda\in  \mathcal D'(\widetilde{W}_\lambda)$ we associate to it a distribution on $W$ given by $\frac{1}{w_\lambda}v_\lambda\in  \mathcal D'(\widetilde{W}_\lambda)$. We may pullback this distribution to $M$ and obtain another distribution. Then we multiply by the $m_\kappa$ and obtain a distribution density. Thus, if $\sigma:M_\kappa\to W_\lambda$ then
\begin{equation}\label{pullback of distribution densities}
(\sigma^*v)_\kappa=m_\kappa (\lambda\circ \sigma\circ \kappa^{-1})^*(\frac{1}{w_\lambda}v_\lambda)\,.
\end{equation}
Distribution densities on $W$ are identified with the continuous linear forms on $C_c^\infty(W)$ by
\[
v(\psi\circ \lambda)=v_\lambda(\psi) \qquad (\psi\in C_c^\infty(\widetilde{W}_\lambda))\,.
\]
(Here $v$ stands for the corresponding continuous linear form.) In particular, if $F\in C(W)$, then $F\mu_W$ is a continuous linear form on $C_c^\infty(W)$ and for $\psi\in C_c^\infty(\widetilde{W}_\lambda)$,
\[
(F\mu_W)_\lambda(\psi)=(F\mu_W)(\psi\circ \lambda)=\int_W \psi\circ \lambda(y)F(y)\,d\mu_W(y)
=\int_{\tilde W_\lambda} \psi(x) (F\circ \lambda^{-1})(x) w_\lambda(x)\,dx\,. 
\]
Hence, for $\phi\in C_c^\infty(\widetilde{M}_\kappa)$, with $\sigma:M_\kappa\to W_\lambda$,
\begin{align*}
(\sigma^*(F\mu_W))_\kappa(\phi)&=(\lambda\circ \sigma\circ \kappa^{-1})^*(\frac{1}{w_\lambda}(F \mu_W)_\lambda)(m_\kappa \phi)\\
&=\int_{\tilde M_\kappa} m_\kappa (x) \phi(x) F\circ\lambda^{-1}\circ (\lambda \circ \sigma\circ \kappa^{-1})(x)\,dx\\
&=\int_{\tilde M_\kappa} \phi(x) (F \circ \sigma)\circ \kappa^{-1}(x)m_\kappa (x) \,dx\\
&=\int_M \phi\circ \kappa(y) (F \circ \sigma)(y) \,d\mu_M(y)\,.
\end{align*}
Thus
\begin{equation}\label{pullback of function times Haar}
\sigma^*(F\mu_W)=(F\circ \sigma) \mu_M\,.
\end{equation}
As explained above, we identify $\mathcal D'(M)$ with the space of the continuous linear forms on $C_c^\infty(M)$ and similarly for $W$ and obtain
\begin{equation}\label{final pullback}
\sigma^*:\mathcal D'(M)\to \mathcal D'(W)
\end{equation}
as the unique continuous extension of (\ref{pullback of function times Haar}). Our identification of distribution densities with continuous linear forms on on the space of the smooth compactly supported functions applies also to submanifolds of Lie groups.

Let $\Sg$ be a Lie group acting on another Lie group $W$ and let $U\subseteq W$ be a submanifold. (In our applications $W$ is going to be a vector space.) We shall consider the following function
\begin{equation}\label{the map f}
\sigma:\Sg\times U\ni (s,u)\to s.u\in W\,.
\end{equation}
The following fact is easy to check.
\begin{lemma}\label{intersetion of orbit with U}
If $\mathcal O\subseteq W$ is an $\Sg$-orbit then $\sigma^{-1}(\mathcal O)=\Sg\times (\mathcal O\cap U)$. 
\end{lemma}
Assume that the map (\ref{the map f}) is submersive. Let us fix Haar measures on $\Sg$ and on $W$ so that the pullback
\[
\sigma^*:\mathcal D'(W)\to \mathcal D'(\Sg\times U)
\]
is well defined, as in (\ref{final pullback}).
Denote by $\Sg^U\subseteq \Sg$ the stabilizer of $U$.
\begin{lemma}\label{main pullback lemma}
Assume that the map (\ref{the map f}) is submersive and surjective. Let
$\mathcal O\subseteq W$ be an $\Sg$-orbit and let $\mu_{\mathcal O}\in \mathcal D'(W)$ be an $\Sg$-invariant positive measure supported on the closure on $\mathcal O$. Let $\mu_{\mathcal O}|_{U}\in \mathcal D'(U)$ be the restriction of $\mu_{\mathcal O}$ to $U$ in the sense of \cite[Cor. 8.2.7]{Hormander}. Then $\mu_{\mathcal O}|_{U}$ is a positive $\Sg^U$-invariant measure supported on the closure of $\mathcal O\cap U$ in $U$. Moreover, 
%there is a left invariant Haar measure $\mu_\Sg$ on $\Sg$ such that 
%%
\begin{equation}\label{main pullback lemma1}
\sigma^*\mu_{\mathcal O}=\mu_\Sg\otimes \mu_{\mathcal O}|_{U}\,.
\end{equation}
\end{lemma}
\begin{prof}
Let $s\in\Sg^U$. Then
\[
s^*\left(\mu_{\mathcal O}|_{U}\right)=\left(s^*\mu_{\mathcal O}\right)|_{U}=\mu_{\mathcal O}|_{U}\,.
\]
Hence the distribution $\mu_{\mathcal O}|_{U}$ is $\Sg^U$-invariant. Lemma \ref{I.1} implies that $\mu_{\mathcal O}|_{U}\ne 0$ and Lemma \ref{intersetion of orbit with U} that $\mu_{\mathcal O}|_{U}$ is supported in the closure of $\mathcal O\cap U$ in $U$. Since the pullback of a positive measure is a non-negative measure, $\mu_{\mathcal O}|_{U}$ is a positive $\Sg^U$-invariant measure supported on the closure of $\mathcal O\cap U$ in $U$.

Theorem 3.1.4' in \cite{Hormander} implies that there is a positive measure $\mu_{\mathcal O\cap U}$ on $U$ such that 
\[
\sigma^*\mu_{\mathcal O}=\mu_\Sg\otimes \mu_{\mathcal O\cap U}\,.
\]
Consider the embedding
\[
\sigma_1:U\ni u\to (1,u)\in \Sg\times U\,.
\]
Then $\sigma\circ \sigma_1:U\to W$ is the inclusion of $U$ into $W$. Hence,
\[
(\sigma\circ \sigma_1)^*\mu_{\mathcal O}=\mu_{\mathcal O}|_{U}\,.
\]
The conormal bundle to $\sigma_1$, as defined in \cite[Theorem 8.2.4]{Hormander}, is equal to
\[
N_{\sigma_1}=T^*(\Sg)\times 0|_{\{1\}\times U}^* \subseteq T^*(\Sg)\times 0 \subseteq T^*(\Sg\times U)\,.
\]
By the $\Sg$-invariance of $\sigma^*\mu_{\mathcal O}$, 
\[
WF(\mu_\Sg\otimes \mu_{\mathcal O\cap U})\subseteq 0\times T^*(U)\subseteq T^*(\Sg\times U)\,.
\]
Hence 
\[
N_{\sigma_1}\cap WF(\mu_\Sg\otimes \mu_{\mathcal O\cap U})=0\,.
\]
Therefore
\[
\mu_{\mathcal O}|_{U}=(\sigma\circ \sigma_1)^*\mu_{\mathcal O}=\sigma_1^*\circ \sigma^*\mu_{\mathcal O}=\sigma_1^*(\mu_\Sg\otimes \mu_{\mathcal O\cap U})=\mu_{\mathcal O\cap U}\,.
\]
This implies (\ref{main pullback lemma1}).
\end{prof}
%%

%% A different indexation for the appendix
\setcounter{thh}{0}
\setcounter{equation}{0}
\section{\textbf{Wave front set of an asymptotically homogeneous distribution}}\label{Appendix WF}
Let
\[
\mathcal F f(x)=\int_{\R^n}f(y)e^{-2\pi ix\cdot y}\,dy
\]
denote the usual Fourier transform on $\R^n$. Recall that for $t>0$ the function $M_t:\R^n \to \R^n$ is defined by $M_t(x)=tx$.
\begin{lemma}\label{wave front set 1}
Suppose $f, u\in\Ss'(\R^n)$, $u$ is homogeneous of degree $d\in \C$, and
\begin{equation}\label{wave front set 1.1}
t^dM_{t^{-1}}^*f (\psi)\underset{t\to 0+}{\longrightarrow }u (\psi) \qquad (\psi\in \Ss(\R^n))\,.
\end{equation}
Then
\begin{equation}\label{wave front set 1.2}
WF_0(\mathcal F^{-1} f)\supseteq \supp u\,.
\end{equation}
\end{lemma}
\begin{prof}
Suppose $\Phi\in C_c^\infty(\R^n)$ is such that $\Phi(0)\ne 0$. We need to show that the localized Fourier transform
\[
\mathcal F((\mathcal F^{-1} f) \Phi)
\]
is not rapidly decreasing in any open cone $\Gamma$ which has a non-empty intersection with $\supp u$. (See \cite[Definition 8.1.2]{Hormander}.) In order to do it, we will choose a function $\psi\in C_c^\infty(\Gamma)$ such that $u(\psi)\ne 0$ and show that
\begin{eqnarray}\label{wave front set 1.3}
\int_{\R^n}(t^{-1})^{-d}\mathcal F((\mathcal F^{-1} f) \Phi)(t^{-1}x)\psi(x)\,dx
\underset{t\to 0+}{\longrightarrow } u(\psi)\,,
\end{eqnarray}
assuming $\Phi(0)=1$.
Let $\phi=\mathcal F \Phi$. Then $\int_{\R^n} \phi(x)\,dx=1$. Notice that
\begin{eqnarray}\label{wave front set 1.4}
t^dM_{t^{-1}}^*(f*\phi)=(t^dM_{t^{-1}}^*f)*(t^{-n}M_{t^{-1}}^*\phi)\,,
\end{eqnarray}
so that, by setting $\check \psi(x)=\psi(-x)$, we have
\begin{multline}\label{wave front set 1.5}
\int_{\R^n}(t^{-1})^{-d}\mathcal F((\mathcal F^{-1} f) \Phi)(t^{-1}x)\psi(x)\,dx\\
=t^dM_{t^{-1}}^*(f*\phi)*\check\psi(0)=(t^dM_{t^{-1}}^*f)*\left((t^{-n}M_{t^{-1}}^*\phi)*\check\psi\right)(0)\,.
\end{multline}
We will check that for an arbitrary $\psi\in\Ss(\R^n)$
\begin{eqnarray}\label{wave front set 1.6}
(t^{-n}M_{t^{-1}}^*\phi)*\psi\underset{t\to 0+}{\longrightarrow }\psi
\end{eqnarray}
in the topology of $\Ss(\R^n)$. This, together with (\ref{wave front set 1.5}) and Banach-Steinhaus Theorem, \cite[Theorem 2.6]{RudinFunc}, will imply (\ref{wave front set 1.3}). Explicitly,
\begin{eqnarray}\label{wave front set 1.7}
\big((t^{-n}M_{t^{-1}}^*\phi)*\psi\big)(x)-\psi(x)=\int_{\R^n}\phi(y)(\psi(x-ty)-\psi(x))\,dy\,.
\end{eqnarray}
Fix $N=0, 1, 2, \dots$ and $\epsilon>0$. Choose $R>0$ so that
\begin{eqnarray}\label{wave front set 1.8}
\int_{|y|\geq R}|\phi(y)|\,dy \cdot \left((1+|y|)^N+1\right)\,\sup_{x\in\R^n} (1+|x|)^N|\psi(x)|<\epsilon\,.
\end{eqnarray}
Let $0<t\leq 1$. Then
\begin{eqnarray}\label{wave front set 1.9}
&&\hskip -2cm (1+|x|)^N\int_{|y|\geq R}|\phi(y)||\psi(x-ty)|\,dy \\
&\leq& \int_{|y|\geq R}|\phi(y)|(1+|ty|)^N(1+|x-ty|)^N|\psi(x-ty)|\,dy\nn\\
&\leq& \int_{|y|\geq R}|\phi(y)|(1+|y|)^N\,dy\cdot \sup_{x\in\R^n}(1+|x|)^N|\psi(x)|\nn
\end{eqnarray}
and 
\begin{eqnarray}\label{wave front set 1.10}
(1+|x|)^N\int_{|y|\geq R}|\phi(y)||\psi(x)|\,dy\leq \int_{|y|\geq R}|\phi(y)|\,dy\cdot \sup_{x\in\R^n}(1+|x|)^N|\psi(x)|
\end{eqnarray}
so that, by (\ref{wave front set 1.8}),
\begin{eqnarray}\label{wave front set 1.11}
(1+|x|)^N\left|\int_{|y|\geq R}\phi(y)(\psi(x-ty)-\psi(x))\,dy\right|<\epsilon \qquad (0<t\leq 1,\ x\in\R^n)\,.
\end{eqnarray}
Choose $r>0$ so that
\begin{eqnarray}\label{wave front set 1.12}
(1+|x|)^N\left|\int_{|y|\leq R}\phi(y)(\psi(x-ty)-\psi(x))\,dy\right|<\epsilon \qquad (0<t\leq 1,\ |x|\geq r)\,.
\end{eqnarray}
Since the function $\psi$ is uniformly continuous,
\begin{eqnarray}\label{wave front set 1.13}
\underset{t\longrightarrow 0+}{\rm limsup}\sup_{|x|\leq r}\left|\int_{|y|\leq R}\phi(y)(\psi(x-ty)-\psi(x))\,dy\right|=0\,.
\end{eqnarray}
Hence,
\begin{eqnarray}\label{wave front set 1.14}
\underset{t\longrightarrow 0+}{\rm limsup}\sup_{x\in\R^n}(1+|x|)^N\left|\int_{|y|\leq R}\phi(y)(\psi(x-ty)-\psi(x))\,dy\right|\leq \epsilon\,.
\end{eqnarray}
By combining (\ref{wave front set 1.11}) and (\ref{wave front set 1.14}), we see that
\begin{eqnarray}\label{wave front set 1.15}
\underset{t\longrightarrow 0+}{\rm limsup}\sup_{x\in\R^n}(1+|x|)^N\left|\int_{\R^n}\phi(y)(\psi(x-ty)-\psi(x))\,dy\right|\leq 2\epsilon\,.
\end{eqnarray}
Since the $\epsilon>0$ is arbitrary, (\ref{wave front set 1.15}) and (\ref{wave front set 1.7}) show that
\begin{eqnarray}\label{wave front set 1.16}
\underset{t\longrightarrow 0+}{\rm limsup}\sup_{x\in\R^n}(1+|x|)^N\left|(t^{-n}M_{t^{-1}}^*\phi)*\psi(x)-\psi(x)\right|=0\,.
\end{eqnarray}
Since the differentiation commutes with the convolution, (\ref{wave front set 1.16}) implies (\ref{wave front set 1.6})  and we are done.
\end{prof}
%%

%% A different indexation for the appendix
\setcounter{thh}{0}
\setcounter{equation}{0}
\section{\textbf{A restriction of a nilpotent orbital integral}}\label{Restriction of a nilpotent orbital integral}
Let $\Wv$ be a Euclidean space, isomorphic to $\R^M$ with the usual dot product. The Lebesgue measure on any subspace of $\Wv$ will be normalized so that the volume of the unit cube is $1$. This is consistent with \cite{Hormander}.

Consider the following diagram
\begin{equation}\label{diagram1}
\begin{tikzpicture}
  % Tell it where the nodes are
  \node (A) {$\Wv$};
  \node (B) [below=of A] {$\underline\Vv$};
  \node (C) [right=of A] {$\Wv$};
  \node (D) [right=of B] {$\Vv$};
  % Tell it what arrows to draw
  \draw[-stealth] (B)-- node[left] {\small $\underline\iota$} (A);
  \draw[-stealth] (B)-- node [below] {\small $\kappa$} (D);
  \draw[-stealth] (A)-- node [above] {\small $\kappa$} (C);
  \draw[-stealth] (D)-- node [right] {\small $ \iota$} (C);
\end{tikzpicture}
\end{equation}
where $\underline \iota: \underline{\Vv}\to \Wv$ and $\iota: \Vv\to \Wv$ are submanifolds and $\kappa(\underline\Vv)=\Vv$. Then we have the following formula for the pull-backs of distributions,
\begin{equation}\label{diagram2}
\underline\iota^*f=(\kappa|_{\underline\Vv})^*\iota^*(\kappa^{-1})^* f
\end{equation}
\cite[Theorems 6.1.2 and 8.2.4]{Hormander}, where $f\in C_c^\infty(\Wv)^*$ is such that these pullbacks are well defined. 

Assume further that $\Wv$ is the direct sum of two orthogonal subspaces
\begin{equation}\label{diagram3}
\Wv=\Uv\oplus \Vv\,,
\end{equation}
that $\underline \Vv=\kappa^{-1}(\Vv)$ and that
\[
\kappa^{-1}(u+v)=\kappa^{-1}(u)+v \qquad (u\in\Uv, v\in\Vv)\,.
\]
Then 
\[
\underline \Vv=N+\Vv\,,
\]
where $N=\kappa^{-1}(0)$. Let
\[
\iota_\Uv:\Uv\to\Wv\,,\ \ \ p_\Uv:\Wv\to\Uv
\]
be the injection and the projection defined by the decomposition \eqref{diagram3}.
\begin{lemma}\label{diagram4}
Suppose $a\in C^\infty(\Uv)$ and 
\[
f(\phi)=\int_\Uv \big(\phi\circ\kappa^{-1}\big)(u) a(u)\,du \qquad (\phi\in C_c^\infty(\Wv))\,.
\]
Then
\[
\underline \iota^* f= |\det((p_\Uv\kappa^{-1}\iota_\Uv)')(0)|a(0)\delta_N\in\Ss'(N+\Vv)\,.
\]
\end{lemma}
\begin{prof}
By taking the derivative of both sides of the equation $I=\kappa\circ\kappa^{-1}$ we see that
\[
I=(\kappa\circ\kappa^{-1})\circ(\kappa^{-1})'\,.
\]
Hence,
\[
\det(\kappa')\circ\kappa^{-1}\circ \iota_\Uv=\frac{1}{\det((p_\Uv\kappa^{-1}\iota_\Uv)')}\,.
\]
Therefore, by \cite[Theorems 6.1.2]{Hormander},
\begin{align*}
(\kappa^{-1})^* f(\phi)=f(\phi\circ \kappa |\det \kappa'|)
&=\int_\Uv \phi\circ \kappa\circ\kappa^{-1}(u) |\det(\kappa'\circ\kappa^{-1})(u)| a(u)\,du\\
&=\int_\Uv \phi(u) |\det((p_\Uv\kappa^{-1}\iota_\Uv)')(u)| a(u)\,du
\end{align*}
and we deduce from \cite[Example 8.2.8]{Hormander} that
\[
\iota^*(\kappa^{-1})^* f(\phi)=|\det((p_\Uv\kappa^{-1}\iota_\Uv)')(0)| a(0)\phi(0)\,.
\]
Now the claim follows from \eqref{diagram2}.
\end{prof}

From now on we specialize to $\Wv=M_{2m,n}(\R)$ with $m\leq n$. Let $\Oo\subseteq \Wv$ denote the $\Sp_{2m}(\R)\times\Og_n$ - orbit through
\[
N=\begin{pmatrix}
I_m & 0\\
0 & 0
\end{pmatrix}\in \Wv\,.
\]
Denote by $\mathcal{H}_m(\R)\subseteq M_m(\R)$ the subspace of the symmetric matrices.
\begin{lemma}\label{diagram5}
The following formula
\[
f(\phi)=\int_{\mathcal{H}_m(\R)}\int_{M_{m,n}(\R)}\phi
\begin{pmatrix}
X\\
CX
\end{pmatrix}
|\det(XX^t)|^{\frac{m+1-n}{2}}\,dX\,dC
\]
defines an invariant measure $f\in\Ss'(\Wv)$ on the orbit $\Oo$.
\end{lemma}
\begin{prof}
Since for $g\in\GL_m(\R)$ and $B,C\in SM_m(\R)$,
\[
\begin{pmatrix}
I_m & 0\\
C & I_m
\end{pmatrix}
\begin{pmatrix}
g & 0\\
0 & (g^t)^{-1}
\end{pmatrix}
\begin{pmatrix}
I_m & B\\
0 & I_m
\end{pmatrix}
\begin{pmatrix}
I_m & 0\\
0 & 0
\end{pmatrix}
=
\begin{pmatrix}
g & 0\\
Cg & 0
\end{pmatrix}
\]
we see that
\[
\Oo=\left\{
\begin{pmatrix}
X\\
CX
\end{pmatrix};\ X\in M_{m,n}(\R)\,, \ \text{rank}(X)=m\,, \ C\in 
\mathcal{H}_m(\R)
\right\}\,.
\]
Furthermore the elements 
\[
\begin{pmatrix}
I_m & 0\\
C & I_m
\end{pmatrix}\,,\ \ \ 
\begin{pmatrix}
g & 0\\
0 & (g^t)^{-1}
\end{pmatrix}\,, \ \ \ 
\begin{pmatrix}
0 & I_m\\
-I_m & 0
\end{pmatrix}
\]
generate $\Sp_{2m}(\R)$ and it is easy to check that $f$ is invariant under the action of these elements, assuming  the following two formulas:
\begin{align*}
\int_{\mathcal{H}_m(\R)}\psi(gCg^t)\,dC&=|\det g|^{m+1}\int_{\mathcal{H}_m(\R)}\psi(C)\,dC\,,\\
\int_{\mathcal{H}_m(\R)}\psi(C^{-1})\,dC&=\int_{\mathcal{H}_m(\R)}\psi(C)|\det C|^{-m-1}\,dC\,.
\end{align*}
\end{prof}

The space tangent to $\Oo$ at $N$ may be identified with 
\[
\Uv=\left\{
\begin{pmatrix}
u_{1,1} & u_{1,2}\\
B & 0
\end{pmatrix}
;\ 
u_{1,1}\in M_m(\R)\,,\ u_{1,2}\in M_{m, n-m}(\R)\,,\ B\in 
\mathcal{H}_m(\R)\right\}\,.
\]
Then the orthogonal complement is equal to
\[
\Vv=\left\{
\begin{pmatrix}
0 & 0\\
D & u_{2,2}
\end{pmatrix}
;\ 
D=-D^t\in M_m(\R)\,,\ u_{2,2}\in M_{m, n-m}(\R)\right\}\,.
\]
Set $\underline \Vv=N+\Vv$. Then we have the inclusion $\underline\iota: \underline{\Vv}\to \Wv$.
\begin{lemma}\label{diagram6}
Let $f$ be as in Lemma \ref{diagram5}. Then
\[
\underline \iota ^* f=\delta_N\in\Ss'(N+\Vv)\,.
\]
\end{lemma}
\begin{prof}
First we rewrite $f$ as an integral over $\Uv$. Let
\[
N_1=(I_m\ 0)\in M_{m,n}(\R)\,.
\]
Then
\[
f(\phi)=\int_{\Uv}\phi
\begin{pmatrix}
u_1+N_1\\
B(u_1+N_1)
\end{pmatrix}
|\det (u_1+N_1)(u_1+N_1)^t|^{\frac{m+1-n}{2}}\,du\,dB\,,
\]
where 
\[
u=
\begin{pmatrix}
u_{1,1} & u_{1,2}\\
B & 0
\end{pmatrix}\,,\ \ \ u_1=(u_{1,1} \ u_{1,2})\,.
\]
Next we introduce the diffeomorphism
\[
\kappa^{-1}(u+v)=
\begin{pmatrix}
u_1+N_1\\
B(u_1+N_1)
\end{pmatrix}+v
\qquad (u\in\Uv\,,\ v\in\Vv)\,.
\]
Then
\[
p_\Uv \kappa^{-1}\iota_\Uv(u)=
\begin{pmatrix}
u_{1,1}+N_1 & u_{1,2}\\
\frac{1}{2}(B(u_{1,1}+N_1)+(u_{1,1}+N_1)^tB) &0
\end{pmatrix}\,.
\]
Hence
\[
(p_\Uv \kappa^{-1}\iota_\Uv)'(0)(\Delta u)=
\begin{pmatrix}
\Delta u_{1,1} & \Delta u_{1,2}\\
\Delta B &0
\end{pmatrix}
\]
and consequently
\[
\det((p_\Uv \kappa^{-1}\iota_\Uv)'(0))=1\,.
\]
Since
\[
|\det (u_1+N_1)(u_1+N_1)^t|^{\frac{m+1-n}{2}}|_{u=0}=1\,,
\]
the claim follows from Lemma \ref{diagram4}.
\end{prof}

\begin{lemma}\label{diagram7}
Suppose $m\leq n$. For $\psi\in\Ss(M_{m,n}(\R))^{\Og_n}$
\[
\int_{M_{m,n}(\R)} \psi(X)\,dX=|S^{n-1}| \int_{M_{m,n-1}(\R)}|\det(XX^t)|^{\frac{1}{2}} \psi|_{M_{m,n-1}}(X)\,dX\,.
\]
\end{lemma}
\begin{prof}
By working in spherical coordinates of decreasing dimensions on the rows of $X$, we see that the left-hand side is equal to
\begin{multline}\label{diagram9}
\int_{\R^{\frac{m(m-1)}{2}}}\int_{(\R^+)^m}\int_{S^{n-1}} \cdots\int_{S^{n-m}}
\psi
\left(
\scriptsize{%
\begin{array}{ccccc|ccc}
r_1\sigma_{1,1} & r_1\sigma_{1,2} & \cdots & \cdots & r_1\sigma_{1,m} & r_1\sigma_{1,m+1} & \cdots& r_1\sigma_{1,n}\\
\cline{1-1}
\multicolumn{1}{c|}{x_{2,1}} & r_2\sigma_{2,2} & \cdots & \cdots & r_2\sigma_{2,m} & r_2\sigma_{2,m+1 } & \cdots& r_2\sigma_{2,n}\\
\cline{2-2}
x_{3,1} & \multicolumn{1}{c|}{x_{3,2}} & r_3\sigma_{3,3} & \cdots & r_3\sigma_{3,m} & r_3\sigma_{3,m+1} & \cdots & r_3\sigma_{3,n}\\
\hhline{*{2}{~}-*{4}{~}} 
\vdots & \vdots  & \ddots & \ddots & \vdots & \vdots & & \vdots\\
\hhline{*{3}{~}-*{3}{~}} 
x_{m,1} & x_{m,2} & \cdots& \multicolumn{1}{c|}{x_{m,m-1}} & r_m\sigma_{m,m} &  r_m\sigma_{m,m+1} & \cdots& r_m\sigma_{m,n}
\end{array}
}
\right)\\[.2em]
\times d\sigma_m \cdots d\sigma_1\,r_1^{n-1}r_2^{n-2}\cdots\,r_m^{n-m}\,dr_m\cdots\,dr_2\,dr_1\,dx_{2,1}\cdots dx_{m,m-1}\,,
\end{multline}
where 
\begin{align*}
\sigma_1&=(\sigma_{1,1}, \sigma_{1,2}, \dots, \sigma_{1,n})\in S^{n-1}\,,\\
\sigma_2&=(\sigma_{2,2}, \sigma_{2,3}, \dots, \sigma_{2,n})\in S^{n-2}\,,\\
&\;\,\vdots \\
\sigma_m&=(\sigma_{m,m}, \sigma_{m,m+1}, \dots, \sigma_{m,n})\in S^{n-m}\,.
\end{align*}
The $\Og_n$-invariance implies that \eqref{diagram9} is equal to
\begin{align*}
&\int_{\R^{\frac{m(m-1)}{2}}}\int_{(\R^+)^m} |S^{n-1}|\cdots |S^{n-m}|
\psi
\left(
{\footnotesize
\begin{array}{ccccc|ccc}
r_1 & 0 & \cdots & \cdots & 0 & 0 & \cdots & 0\\
x_{2,1} & r_2 & 0 & \cdots  & 0 & 0 & \cdots & 0\\
\vdots & \ddots &  \ddots & \ddots  &\vdots & \vdots & &\vdots \\
\vdots & & \ddots &  \ddots  & 0 & \vdots & &\vdots \\
x_{m,1} & \dots  & \dots & x_{m,m-1} & r_m & 0 & \cdots & 0
\end{array}}
\right)\\
&\hskip 2cm\times\,r_1^{n-1}r_2^{n-2}\cdots r_m^{n-m}\,dr_m \cdots dr_2\,dr_1\,dx_{2,1} \cdots dx_{m,m-1}\\
&\qquad=|S^{n-1}|\int_{\R^{\frac{m(m-1)}{2}}}\int_{(\R^+)^m}|S^{n-2}|\cdots |S^{n-m}|
\psi
\left(
{\footnotesize
\begin{array}{cccccccc|c}
r_1 & 0 & \cdots & \cdots & 0 & 0 & \cdots & 0 & 0\\
x_{2,1} & r_2 & 0 & \cdots  & 0 & 0 & \cdots & 0 & 0\\
\vdots & \ddots &  \ddots & \ddots  &\vdots & \vdots & &\vdots &\vdots \\
\vdots & & \ddots &  \ddots  & 0 & \vdots & &\vdots &\vdots\\
x_{m,1} & \dots  & \dots & x_{m,m-1} & r_m & 0 & \cdots & 0 & 0 
\end{array}}
\right)\\
&\hskip 2cm\times (r_1r_2 \cdots r_m)\,\,r_1^{n-2}r_2^{n-3} \cdots \,r_m^{n-1-m}\,dr_m\cdots \,dr_2\,dr_1\,dx_{2,1}\cdots dx_{m,m-1}\\
&\qquad=|S^{n-1}|\int_{M_{m,n-1}(\R)} \psi|_{M_{m,n-1}(\R)}(X)|\det(XX^t)|^{\frac{1}{2}}\,dX\,.
\end{align*}
For the last equality, we consider spherical coordinates, as before, but on the first $n-1$ columns only, noticing that for 
$$
X=\left(\begin{array}{c|c}
T & 0
\end{array}\right)
=
\left(
\begin{array}{ccccc|ccc}
r_1 & 0 & \cdots & \cdots & 0 & 0 & \cdots & 0 \\
x_{2,1} & r_2 & 0 & \cdots  & 0 & 0 & \cdots & 0 \\
\vdots & \ddots &  \ddots & \ddots  &\vdots & \vdots & &\vdots\\
\vdots & & \ddots &  \ddots  & 0 & \vdots & &\vdots \\
x_{m,1} & \dots  & \dots & x_{m,m-1} & r_m & 0 & \cdots & 0  
\end{array}
\right)\in M_{m,n-1}(\R)\,,
$$
we have 
$$
XX^t=\left(\begin{array}{c|c}
T & 0
\end{array}\right)
\left(\begin{array}{c}
T^t\\
\cline{1-1}
0
\end{array}\right)
=TT^t\,.
$$
Hence 
$$
\det(XX^t)=\det(TT^t)=\det(T)^2=(r_1r_2\cdots r_m)^2\,.
$$
\end{prof}

\begin{corollary}\label{diagram10}
Let us denote the measure $f\in \Ss'(M_{m,n})$ defined in Lemma \ref{diagram5} by $f_n$ and assume $n>m$. 
Then for $\phi\in\Ss(M_{m,n}(\R))^{\Og_n}$
\[
f_n(\phi)=|S^{n-1}|f_{n-1}(\phi|_{M_{m,n-1}(\R)})\,,
\]
where
\[
\phi|_{M_{m,n-1}(\R)}(X)=\phi(\!\!\begin{array}{c|c}
X & 0
\end{array}
) \qquad(X\in M_{m,n-1}(\R))\,.
\]
\end{corollary}
\begin{prof}
This is clear from Lemmas \ref{diagram5} and \ref{diagram9}.
\end{prof}

\bigskip

\subsection*{Acknowledgement} {\rm We thank the referee for careful reading, insightful remarks and questions. }

\bibliographystyle{alpha}
%\addcontentsline{toc}{section}{References}

\end{document}